\documentclass[11pt,a4paper,reqno]{amsart} 
\usepackage{amsmath}
\usepackage{amssymb}
\usepackage{euscript}
\usepackage{graphicx}
\usepackage[all]{xy}
\usepackage{xcolor}
\usepackage[colorlinks=true,linktocpage=true,pagebackref=false, citecolor=black,linkcolor=black]{hyperref}
\usepackage{pifont}
\usepackage{tikz,pgfplots}
\pgfplotsset{compat=1.10}
\usepgfplotslibrary{fillbetween}
\usetikzlibrary{cd,arrows,decorations.pathmorphing,backgrounds,automata,positioning,fit,matrix}
\usepackage{caption,subcaption}

\pgfplotsset{soldot/.style={color=black,only marks,mark=*}} \pgfplotsset{holdot/.style={color=black,fill=white,only marks,mark=*}}

\newtheorem{thm}{Theorem}[section]

\newtheorem{lem}[thm]{Lemma}
\newtheorem{prop}[thm]{Proposition}

\newtheorem{cor}[thm]{Corollary}
\newtheorem{conj}[thm]{Conjecture}

\theoremstyle{definition}
\newtheorem{defn}[thm]{Definition}

\theoremstyle{remark}
\newtheorem{remark}[thm]{Remark}
\newtheorem{remarks}[thm]{Remarks}
\newtheorem{example}[thm]{Example}
\newtheorem{examples}[thm]{Examples}

\newtheorem{prob}[thm]{Problem}

\numberwithin{equation}{section}
\numberwithin{figure}{section}

 \newcommand{\R}{{\mathbb R}}
 \newcommand{\C}{{\mathbb C}}

\newcommand{\sph}{{\mathbb S}} 
 \newcommand{\PP}{{\mathbb P}}

 \newcommand{\Cont}{{\mathcal C}}

\newcommand{\Ff}{{\EuScript F}}
\newcommand{\Pp}{{\EuScript P}}
\newcommand{\Ss}{{\EuScript S}}
\newcommand{\Tt}{{\EuScript T}}

\newcommand{\Bb}{{\EuScript B}}
\newcommand{\Cc}{{\EuScript C}}
\newcommand{\pol}{{\EuScript K}}
\newcommand{\Qq}{{\EuScript Q}}

\newcommand{\Uu}{{\EuScript U}}

\newcommand{\Reg}{\operatorname{Reg}}
\newcommand{\Sing}{\operatorname{Sing}}
\newcommand{\Int}{\operatorname{Int}}
\newcommand{\cl}{\operatorname{Cl}}
\newcommand{\dist}{\operatorname{dist}}
\newcommand{\id}{\operatorname{id}}

\newcommand{\zar}{\operatorname{zar}}

\newcommand{\x}{{\tt x}} \newcommand{\y}{{\tt y}} 
\newcommand{\z}{{\tt z}} \renewcommand{\t}{{\tt t}}
 
 \newcommand{\e}{{\tt e}}

\newcommand{\veps}{\varepsilon}
\newcommand{\ol}{\overline}
\usepackage{scalerel}
\DeclareMathOperator*{\Bigast}{\scalerel*{\ast}{\Theta}}
\begin{document}

\title[Surjective Nash maps between semialgebraic sets]{Surjective Nash maps between semialgebraic sets}

\author{Antonio Carbone}
\address{Dipartimento di Matematica, Via Sommarive, 14, Universit\`a di Trento, 38123 Povo (ITALY)}
\email{antonio.carbone@unitn.it}
\thanks{The first author is supported by GNSAGA of INDAM}

\author{Jos\'e F. Fernando}
\address{Departamento de \'Algebra, Geometr\'\i a y Topolog\'\i a, Facultad de Ciencias Matem\'aticas, Universidad Complutense de Madrid, Plaza de Ciencias 3, 28040 MADRID (SPAIN)}
\email{josefer@mat.ucm.es}
\thanks{The second author is supported by Spanish STRANO PID2021-122752NB-I00 and Grupos UCM 910444. This article has been developed during several one month research stays of the second author in the Dipartimento di Matematica of the Universit\`a di Trento. The second author would like to thank the department for the invitations and the very pleasant working conditions.
}

%\dedicatory{Dedicated to Prof. J.M. Gamboa on occasion of his 65th birthday}

\date{01/06/2023}

\begin{abstract}
In this work we study the existence of surjective Nash maps between two given semialgebraic sets $\Ss$ and $\Tt$. Some key ingredients are: the irreducible components $\Ss_i^*$ of $\Ss$ (and their intersections), the analytic-path connected components $\Tt_j$ of $\Tt$ (and their intersections) and the relations between dimensions of the semialgebraic sets $\Ss_i^*$ and $\Tt_j$. A first step to approach the previous problem is the former characterization done by the second author of the images of affine spaces under Nash maps.

The core result of this article to obtain a criterion to decide about the existence of surjective Nash maps between two semialgebraic sets is: {\em a full characterization of the semialgebraic subsets $\Ss\subset\R^n$ that are the image of the closed unit ball $\ol{\Bb}_m$ of $\R^m$ centered at the origin under a Nash map $f:\R^m\to\R^n$}. The necessary and sufficient conditions that must satisfy such a semialgebraic set $\Ss$ are: {\em it is compact, connected by analytic paths and has dimension $d\leq m$}. 

Two remarkable consequences of the latter result are the following: (1) {\em pure dimensional compact irreducible arc-symmetric semialgebraic sets of dimension $d$ are Nash images of $\ol{\Bb}_d$}, and (2) {\em compact semialgebraic sets of dimension $d$ are projections of non-singular algebraic sets of dimension $d$ whose connected components are Nash diffeomorphic to spheres (maybe of different dimensions)}. 
\end{abstract}

\maketitle

%\tableofcontents

\section{Introduction}\label{s1}

Although it is usually said that the first work in Real Geometry is due to Harnack \cite{hr}, who obtained an upper bound for the number of connected components of a non-singular real algebraic curve in terms of its genus, modern Real Algebraic Geometry was born with Tarski's article \cite{ta}, where it is proved that a projection of a semialgebraic set is a semialgebraic set. 

A subset $\Ss\subset\R^n$ is \em semialgebraic \em when it admits a description in terms of a finite boolean combination of polynomial equalities and inequalities, which we will call a {\em semialgebraic description}. A {\em Nash manifold} $M\subset\R^n$ is a semialgebraic set that is a smooth submanifold of $\R^n$. The category of semialgebraic sets is closed under basic boolean operations but also under usual topological operations: taking closures (denoted by $\cl(\cdot)$), interiors (denoted by $\Int(\cdot)$), connected components, etc. If $\Ss\subset\R^m$ and $\Tt\subset\R^n$ are semialgebraic sets, a map $f:\Ss\to\Tt$ is \em semialgebraic \em if its graph is a semialgebraic set. 

A relevant class of (continuous) semialgebraic maps is that of {\em Nash maps}, that is, smooth maps on a semialgebraic neighborhood of $\Ss$ that are in addition semialgebraic. More precisely, a \em Nash map on an open semialgebraic set \em $U\subset\R^m$ is a semialgebraic smooth map on $f:U\to\R^n$. Given a semialgebraic set $\Ss\subset\R^m$, a \em Nash map on $\Ss$ \em is the restriction to $\Ss$ of a Nash map on an open semialgebraic neighborhood $U\subset\R^m$ of $\Ss$. These Nash maps include: polynomial and regular maps. A map $f:=(f_1,\ldots,f_n):\R^m\to\R^n$ is \em polynomial \em if its components $f_k\in\R[{\tt x}]:=\R[{\tt x}_1,\ldots,{\tt x}_m]$ are polynomials. Analogously, $f$ is \em regular \em if its components can be represented as quotients $f_k=\frac{g_k}{h_k}$ of two polynomials $g_k,h_k\in\R[{\tt x}]$ such that $h_k$ never vanishes on $\R^m$.

By Tarski's Theorem the image of a semialgebraic set under a semialgebraic map is a semialgebraic set, because it is a projection of a semialgebraic set. We are interested in studying `inverse type problems' to Tarski's result. A first enlightening example is the following: 

\begin{prob}[B\"archen-Sch\"afchen Problem]\label{probbs}
Is it possible to transform a semialgebraic Teddy bear $\Tt\subset\R^3$ onto a semialgebraic sheep $\Ss\subset\R^3$ by means of a Nash map and/or viceversa? 
\end{prob}

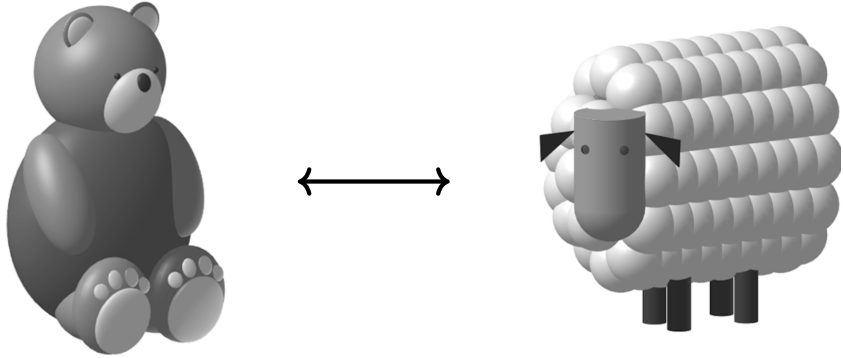
\begin{figure}[ht]
\centering
\begin{minipage}{0.3\textwidth}
\begin{flushright}
\includegraphics[width=3cm]{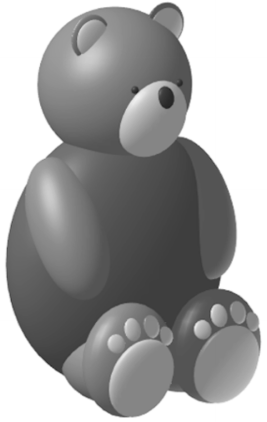}
\end{flushright}
\end{minipage}
\hfil
\begin{minipage}{0.10\textwidth}
\begin{center}
\begin{tikzpicture}[scale=1]
\draw[<->,ultra thick] (-1,0)--(1,0);
\end{tikzpicture}
\end{center}
\end{minipage}
\hfil
\begin{minipage}{0.4\textwidth}
\begin{center}
\hspace*{-1cm}\includegraphics[width=4cm]{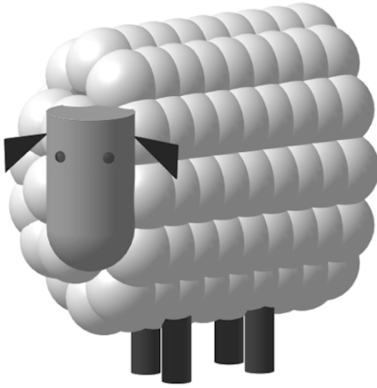}
\end{center}
\end{minipage}
\caption{\small{Semialgebraic Teddy bear and semialgebraic Sheep (figures borrowed from \cite[Fig.1.3]{fu6}).}\label{bear}}
\end{figure} 

The general problem is the following:

\begin{prob}[Tarski's inverse problem]\label{probst}
Let $\Ss\subset\R^m$ and $\Tt\subset\R^n$ be semialgebraic sets. Under which conditions does a surjective Nash map $f:\Ss\to\Tt$ exist? 
\end{prob} 

\subsection{State of the art}
In the 1990 \em Oberwolfach reelle algebraische Geometrie \em week \cite{g} Gamboa proposed (see also \cite[\S3.IV, p.69]{ei}) the following problem: 

\begin{prob}\label{prob0}
Characterize the (semialgebraic) subsets of $\R^n$ that are either polynomial or regular images of $\R^m$. 
\end{prob}
The previous problem pretends to characterize the semialgebraic sets that admit affine spaces as a kind of `algebraic models'. During the last two decades we have attempted to better understand polynomial and regular images of $\R^m$. Our main objectives have been the following: 
\begin{itemize}
\item To find obstructions to be either polynomial or regular images. 
\item To (constructively) prove that large families of semialgebraic sets with piecewise linear boundary (convex polyhedra, their interiors, complements and the interiors of their complements) are either polynomial or regular images of affine spaces. 
\end{itemize}

In \cite{fg1,fg2,fu1,fgu2} we presented first steps to approach Problem \ref{prob0}. The most relevant one \cite{fu1} shows that \em the set of points at infinite of $\Ss$ is a connected set\em. In \cite{fe1} a complete solution for Problem \ref{prob0} appears for the $1$-dimensional case, whereas in \cite{fgu1,fgu3,fgu5,fu2,fu3,fu4,fu5,fu6,u1,u2} a constructive full answer is provided for the representation as either polynomial or regular images of the semialgebraic sets with piecewise linear boundary commented above \cite[Table 1]{fu4}. A survey concerning these topics, which provides the reader a global idea of the state of the art, can be found in \cite{fgu4}. 

\subsubsection{First alternative approach.}
The rigidity of polynomial and regular maps on $\R^m$ makes it difficult to approach Problem \ref{prob0} in its full generality. A first possibility to overcome this is to change the domain of definition (using, for instance, closed unit balls or unit spheres). When considering compact domains (and of course compact images), one has more tools and Weierstrass' polynomial approximation plays an important role. In \cite[\S5.Prob.1]{kps} the following concrete related problem was proposed:

\begin{prob}\label{probl2}
Let $\Pp$ be an arbitrary (compact) convex polygon in $\R^2$. Construct explicit polynomials $f$ and $g$ in $\R[{\tt u},{\tt v},{\tt w}]$ such that $\Pp=(f,g)(\ol{\Bb}_3)$.
\end{prob}

A first main result in \cite{fu6} provides a positive answer to a natural strong generalization to arbitrary dimension of Problem \ref{probl2}.

\begin{thm}[{\cite[Thm.1.2]{fu6}}]\label{main1-i2}
Let $\Ss\subset\R^n$ be the union of a finite family of $n$-dimensional convex (compact) polyhedra. The following assertions are equivalent:
\begin{itemize}
\item[(i)] $\Ss$ is connected by analytic paths.
\item[(ii)] There exists a polynomial map $f:\R^n\to\R^n$ such that $f(\ol{\Bb}_n)=\Ss$.
\end{itemize}
\end{thm}

The techniques involved to prove Theorem \ref{main1-i2} are generalized in \cite{fu4} to show Theorem \ref{main2-i} below. A set $\Ss\subset\R^n$ is \em strictly radially convex (with respect to a point $p\in\Int(\Ss)$) \em if for each ray $\ell$ with origin at $p$, the intersection $\ell\cap\Ss$ is a segment whose relative interior is contained in $\Int(\Ss)$. Convex sets are particular examples of strictly radially convex sets (with respect to any of its interior points \cite[Lem.11.2.4]{ber1}).

\begin{thm}[{\cite[Thm.1.3]{fu6}}]\label{main2-i}
Let $\Ss\subset\R^n$ be the union of a finite family of strictly radially convex semialgebraic sets that are polynomial images of the closed unit ball $\ol{\Bb}_m$. The following assertions are equivalent:
\begin{itemize}
\item[(i)] $\Ss$ is connected by analytic paths.
\item[(ii)] There exists a polynomial map $f:\R^{m+1}\to\R^n$ such that $f(\ol{\Bb}_{m+1})=\Ss$
\end{itemize}
\end{thm}

\subsubsection{Second alternative approach.}
Another possibility is to change the polynomial and regular maps by more flexible maps like {\em Nash maps} (smooth semialgebraic maps) \cite{fe3} or {\em regulous maps} (continuous rational maps) \cite{ffqu}. Gamboa and Shiota proposed in 1990 to approach the following variant of Problem \ref{prob0} in the same line as Problem \ref{probbs}. 

\begin{prob}\label{prob1}
Characterize the (semialgebraic) subsets of $\R^n$ that are Nash images of $\R^m$. 
\end{prob}

The set $\Reg(\Ss)$ of \em regular points of a semialgebraic set $\Ss\subset\R^n$ \em is defined as follows. Let $X$ be the Zariski closure of $\Ss$ in $\R^n$ and $\widetilde{X}$ the complexification of $X$, that is, the smallest complex algebraic subset of $\C^n$ that contains $X$. Define $\Reg(X):=X\setminus\Sing(\widetilde{X})$ and let $\Reg(\Ss)$ be the interior of $\Ss\setminus\Sing(\widetilde{X})$ in $\Reg(X)$. Observe that $\Reg(\Ss)$ is a finite union of disjoint {\em Nash manifolds} maybe of different dimensions. We refer the reader to \cite[\S2.A]{fe3} for further details concerning the set of regular points of a semialgebraic set. In 1990 Shiota proposed the following conjecture in order to provide a satisfactory answer to Problem \ref{prob1}. 

\begin{conj}[Shiota]\label{conj0}
Let $\Ss\subset\R^n$ be a semialgebraic set of dimension $d$. Then $\Ss$ is a Nash image of $\R^d$ if and only if $\Ss$ is pure dimensional and there exists an analytic path $\alpha:[0,1]\to\Ss$ whose image meets all connected components of the set of regular points of $\Ss$.
\end{conj}

In \cite{fe3} we provided a proof for Shiota's conjecture as a particular case of the following characterization of the semialgebraic sets $\Ss\subset\R^n$ of dimension $d$ that are images of affine spaces under Nash maps. 

\begin{thm}[Nash images {\cite[Main Thm.1.4]{fe3}}]\label{main0}
Let $\Ss\subset\R^n$ be a semialgebraic set of dimension $d$. The following assertions are equivalent:
\begin{itemize}
\item[(i)] $\Ss$ is a Nash image of $\R^d$.
\item[(ii)] $\Ss$ is a Nash image of $\R^m$ for some $m\geq d$.
\item[(iii)] $\Ss$ is connected by Nash paths.
\item[(iv)] $\Ss$ is connected by analytic paths.
\item[(v)] $\Ss$ is pure dimensional and there exists a Nash path $\alpha:[0,1]\to\Ss$ whose image meets all the connected components of the set of regular points of $\Ss$.
\item[(vi)] $\Ss$ is pure dimensional and there exists an analytic path $\alpha:[0,1]\to\Ss$ whose image meets all the connected components of the set of regular points of $\Ss$.
\end{itemize}
\end{thm}

\subsubsection{Analytic path-connected components of a semialgebraic set}
In order to present the main results of this article, we recall the concept of {\em analytic path-connected components of a semialgebraic set} introduced in \cite[\S9]{fe3}.

\begin{defn}
A semialgebraic set $\Ss\subset\R^n$ admits a decomposition into \textit{analytic path-connected components} if there exist semialgebraic sets $\Ss_1,\ldots,\Ss_r\subset\Ss$ such that:
\begin{itemize}
\item[\rm{(i)}] Each $\Ss_i$ is connected by analytic paths.
\item[\rm{(ii)}] If $\Tt\subset\Ss$ is a semialgebraic set connected by analytic paths that contains $\Ss_i$, then $\Ss_i=\Tt$.
\item[\rm{(iii)}] $\Ss_i\not\subset\bigcup_{j\neq i} S_j$.
\item[\rm{(iv)}] $\Ss=\bigcup_{i=1}^r\Ss_i$.
\end{itemize}
\end{defn}

In \cite[Thm.9.2]{fe3} the existence and uniqueness of the analytic path-connected components of a semialgebraic set is shown.

\begin{thm}[{\cite[Thm.9.2]{fe3}}]
Let $\Ss\subset\R^n$ be a semialgebraic set. Then $\Ss$ admits a decomposition into analytic path-connected components and this decomposition is unique. In addition, the analytic path-connected components of a semialgebraic set are closed in $\Ss$. 
\end{thm}

\subsubsection{Irreducibility and irreducible components of a semialgebraic set}
The ring ${\mathcal N}(\Ss)$ of Nash functions on a semialgebraic set $\Ss\subset\R^n$ is a noetherian ring \cite[Thm.2.9]{fg3} and we say that $\Ss$ is \em irreducible \em if and only if ${\mathcal N}(\Ss)$ is an integral domain \cite{fg3}. We next recall the concept of irreducible components of a semi-algebraic set. 

\begin{defn}
A semi-algebraic set $\Ss\subset\R^n$ admits a decomposition into \textit{irreducible components} if there exist semi-algebraic sets $\Ss_1,\ldots,\Ss_r\subset\Ss$ such that:
\begin{itemize}
\item[\rm{(i)}] Each $\Ss_i$ is irreducible.
\item[\rm{(ii)}] If $\Tt\subset\Ss$ is an irreducible semi-algebraic set that contains $\Ss_i$, then $\Ss_i=\Tt$.
\item[\rm{(iii)}] $\Ss_i\not\subset\bigcup_{j\neq i} S_j$.
\item[\rm{(iv)}] $\Ss=\bigcup_{i=1}^r\Ss_i$.
\end{itemize}
\end{defn}

In \cite[Thm.4.3, Rmk.4.4]{fg3} the following result is proposed concerning the irreducible components of a semi-algebraic set.

\begin{thm}
Let $\Ss\subset\R^n$ be a semi-algebraic set. Then $\Ss$ admits a decomposition into irreducible components and this decomposition is unique. In addition, the irreducible components of a semi-algebraic set are closed in $\Ss$. 
\end{thm}

\subsection{Main results}
Our purpose in this work is to approach Problem \ref{probst} (combining somehow the alternative approaches to Problem \ref{prob0} suggested above). The image of a semialgebraic set connected by analytic paths under a Nash map is connected by analytic paths as a consequence of Theorem \ref{main0}. In addition, the image of an irreducible semialgebraic set under a Nash map is an irreducible semialgebraic set \cite[\S3.1]{fg3}. Consequently, as we will explain in detail later, obstructions to construct a surjective Nash map $f:\Ss\to\Tt$ between semialgebraic sets $\Ss$ and $\Tt$ concentrate on the configuration of the intersections of pairwise different analytic path-connected components $\{\Ss_i\}_{i=1}^r$ (resp. irreducible components $\{\Ss_j^*\}_{j=1}^\ell$) of $\Ss$ and the configuration of their images, which are semialgebraic subsets $\Tt_i:=f(\Ss_i)$ of $\Tt$ connected by analytic paths (resp. irreducible semialgebraic subsets $\Tt_j^*:=f(\Ss_j^*)$ of $\Tt$). 

\subsubsection{General surjective Nash maps}
If $\Ss\subset\R^m$ is a semialgebraic set of dimension $d$, the set $\Ss^{(e)}$ of points of $\Ss$ of dimension $e$ is a semialgebraic subset of $\Ss$, which is closed (in $\Ss$) if $d=e$. In order to soften the quoted obstructions above, we assume that each irreducible component $\Ss_i^*$ of $\Ss$ is mapped onto a semialgebraic subset $\Tt_i$ of $\Tt$ connected by analytic paths and that $\bigcap_{i=1}^r\Tt_i\neq\varnothing$. Under these assumptions we propose the following solution to Problem \ref{probst}.

\begin{thm}[Surjective Nash maps]\label{snmbss}
Let $\Ss\subset\R^m$ and $\Tt\subset\R^n$ be semialgebraic sets, $\{\Ss_i^*\}_{i=1}^r$ the irreducible components of $\Ss$ and $\{\Tt_i\}_{i=1}^r$ a family of (non-necessarily distinct) semialgebraic subsets of $\Tt$ connected by analytic paths such that $\bigcap_{i=1}^r\Tt_i\neq\varnothing$ and $\Tt=\bigcup_{i=1}^r\Tt_i$. Denote $d_i:=\dim(\Ss_i^*)$ and assume that the set $\Ss_i^{*,(d_i)}$ of points of $\Ss_i^*$ of dimension $d_i$ is non-compact if $\Tt_i$ is non-compact for $i=1,\ldots,r$. Then there exists a surjective Nash map $f:\Ss\to\Tt$ such that $f(\Ss_i^*)=\Tt_i$ for $i=1,\ldots,r$ if and only if $e_i:=\dim(\Tt_i)\leq\dim(\Ss_i^*)=:d_i$ for $i=1,\ldots,r$.
\end{thm}

\subsubsection{Key results}
The proof of Theorem \ref{snmbss} strongly relies on the following two results. The first of them provides a positive answer to Problem \ref{probbs}, whereas the second is the counterpart for the non-compact case.

\begin{thm}[B\"archen-Sch\"afchen's Theorem]\label{nice0}
Let $\Ss\subset\R^m$ be a semialgebraic set of dimension $d$ and $\Tt\subset\R^n$ a compact semialgebraic set connected by analytic paths of dimension $e\leq d$. Then there exists a Nash map $f:\R^m\to\R^n$ such that $f(\Ss)=\Tt$.
\end{thm}

The previous result shows that any semialgebraic set $\Ss\subset\R^m$ of dimension $d\geq e$ is a model to represent any compact semialgebraic set $\Tt\subset\R^n$ connected by analytic paths of dimension $e$ as a Nash image. In particular, {\em a semialgebraic sheep is a Nash image of a semialgebraic Teddy bear and viceversa} (see Problem \ref{probbs}).

\begin{thm}[Non-compact case]\label{nice2}
Let $\Ss\subset\R^m$ be a semialgebraic set such that $\cl(\Ss^{(d)})\cap\Ss$ is non-compact for some $d\geq 2$ and let $\Tt\subset\R^n$ be semialgebraic set connected by analytic paths. If $\Tt$ has dimension $e\leq d$, there exists a Nash map $f:\R^m\to\R^n$ such that $f(\Ss)=\Tt$.
\end{thm}
\begin{remark}
The converse of the previous result is also true. Suppose that $\cl(\Ss^{(d)})\cap\Ss$ is compact for each $d\geq e$. Define $\Ss_1:=\bigcup_{d\geq e}(\cl(\Ss^{(d)})\cap\Ss)$ and $\Ss_2:=\bigcup_{d<e}(\cl(\Ss^{(d)})\cap\Ss)$. Then $\Ss_1$ is compact, $\dim(\Ss_2)<e$ and $\Ss=\Ss_1\cup\Ss_2$. Suppose there exists a Nash map $f:\R^m\to\R^n$ such that $f(\Ss)=\Tt$. Then $\Tt=f(\Ss_1)\cup f(\Ss_2)$, where $f(\Ss_1)$ is compact and $f(\Ss_2)$ has dimension $e'<e$. This contradicts the fact that $\Tt$ is a non-compact semialgebraic set of dimension $d$.
\end{remark}

The previous result (together with the corresponding remark) determines the semialgebraic sets $\Ss\subset\R^m$ of dimension $d\geq e$ that are models to represent a non-compact semialgebraic set $\Tt\subset\R^n$ connected by analytic paths of dimension $e$ as a Nash image.

\subsubsection{Nash images of the closed unit ball}
The core results to prove Theorems \ref{nice0} and \ref{nice2} are Theorem \ref{main0} already proved in \cite{fe3} and Theorem \ref{main1} where we characterize the compact semialgebraic sets $\Ss\subset\R^n$ that are images of closed unit balls under Nash maps. This result provides a full answer to the natural counterpart to Problem \ref{prob1} for the compact case:

\begin{prob}\label{prob1b}
Characterize the (compact semialgebraic) subsets of $\R^n$ that are Nash images of the closed unit ball of $\R^m$. 
\end{prob}

The statement of Theorem \ref{main0} does not take into account whether $\Ss$ is compact or not and the involved Nash maps are rarely proper if $d\geq2$. As closed unit balls $\ol{\Bb}_d$ are compact, the restrictions to $\ol{\Bb}_d$ of the involved Nash maps are always proper maps.

\begin{thm}[Compact Nash images]\label{main1}
Let $\Ss\subset\R^n$ be a $d$-dimensional compact semialgebraic set. The following assertions are equivalent:
\begin{itemize}
\item[(i)] There exists a Nash map $f:\R^d\to\R^n$ such that $f(\ol{\Bb}_d)=\Ss$.
\item[(ii)] There exists a Nash map $f:\R^m\to\R^n$ such that $f(\ol{\Bb}_m)=\Ss$ for some $m\geq d$.
\item[(iii)] $\Ss$ is connected by Nash paths.
\item[(iv)] $\Ss$ is connected by analytic paths.
\item[(v)] $\Ss$ is pure dimensional and there exists a Nash path $\alpha:[0,1]\to\Ss$ whose image meets all the connected components of $\Reg(\Ss)$.
\item[(vi)] $\Ss$ is pure dimensional and there exists an analytic path $\alpha:[0,1]\to\Ss$ whose image meets all the connected components of $\Reg(\Ss)$.
\end{itemize}
\end{thm}

In Subsection \ref{acm} we will show that we can take `equivalent' compact models to the closed unit ball to represent compact semialgebraic sets as their images under Nash maps. The technicalities of the constructions we develop in this article make the simplicial prism the suitable model to squeeze. In Subsection \ref{dimensione1} we separately treat the $1$-dimensional case and we characterize $1$-dimensional Nash images of affine spaces in terms of their irreducibility. 

\begin{prop}[The $1$-dimensional case]\label{curves}
Let $\Ss\subset\R^n$ be a $1$-dimensional compact semialgebraic set. Then $\Ss$ is a Nash image of some $\ol{\Bb}_m$ if and only if $\Ss$ is irreducible. In addition, if such is the case, $\Ss$ is a Nash image of the compact interval $[-1,1]$.
\end{prop}

\subsection{Two consequences}
We next present two additional consequences of Theorem \ref{main1}. 

\subsubsection{Representation of arc-symmetric compact semialgebraic sets} 
Arc-symmetric semialgebraic sets were introduced by Kurdyka in \cite{k} and subsequently studied by many authors. Recall that a semialgebraic set $\Ss\subset\R^n$ is \em arc-symmetric \em if for each analytic arc $\gamma:(-1,1)\to\R^n$ with $\gamma((-1,0))\subset\Ss$ it holds that $\gamma((-1,1))\subset\Ss$. In particular arc-symmetric semialgebraic sets are closed subsets of $\R^n$. An arc-symmetric semialgebraic set $\Ss\subset\R^n$ is \em irreducible \em if it cannot be written as the union of two proper arc-symmetric semialgebraic subsets \cite[\S2]{k}. It follows from Theorem \ref{main1} and \cite[Cor.2.8]{k} that a pure dimensional irreducible compact arc-symmetric semialgebraic set is a Nash image of $\ol{\Bb}_d$ where $d:=\dim(\Ss)$. 

\begin{cor}\label{consq1}
Let $\Ss\subset\R^n$ be a pure dimensional irreducible compact arc-symmetric semialgebraic set of dimension $d$. Then $\Ss$ is a Nash image of $\ol{\Bb}_d$. 
\end{cor}

\subsubsection{Elimination of inequalities} 
Another converse problem to Tarski's Theorem is to find an algebraic set in $\R^{n+k}$ whose projection is a given semialgebraic subset of $\R^n$. This is known as the \em problem of eliminating inequalities\em. Motzkin proved in \cite{m2} that this problem always has a solution for $k=1$. However, his solution is rather complicated and in general is a reducible algebraic set. Andradas--Gamboa proved in \cite{ag1,ag2} that if $\Ss\subset\R^n$ is a closed semialgebraic set whose Zariski closure is irreducible, then $\Ss$ is the projection of an irreducible algebraic set in some $\R^{n+k}$. Pecker \cite{pe} provides some improvements on both results: for the first one by finding a construction of an algebraic set in $\R^{n+1}$ that projects onto the given semialgebraic subset of $\R^n$, far simpler than the original construction of Motzkin; for the second one by proving that if $\Ss$ is a locally closed semialgebraic subset of $\R^n$ with an interior point, then $\Ss$ is the projection of an irreducible algebraic subset of $\R^{n+1}$. 

In \cite{fe3} it is proved that each semialgebraic set $\Ss\subset\R^n$ is the projection of a non-singular algebraic set $X\subset\R^{n+k}$ whose connected components are Nash diffeomorphic to affine spaces (maybe of different dimensions). In this article we improve the previous result when $\Ss$ is compact and we prove that there exists a non-singular compact algebraic set $X\subset\R^{n+k}$ that is Nash diffeomorphic to a finite pairwise disjoint union of spheres (maybe of different dimensions) and projects onto $\Ss$.

\begin{cor}\label{consq2}
Let $\Ss\subset\R^n$ be a compact semialgebraic set of dimension $d$. We have:
\begin{itemize}
\item[(i)] If $\Ss$ is connected by analytic paths, it is the projection of an irreducible compact non-singular algebraic set $X\subset\R^{n+k}$ (for some $k\geq0$) that has at most two connected components Nash diffeomorphic to the sphere $\sph^d$. In addition,
\begin{itemize}
\item[(1)] Both connected component of $X$ projects onto $\Ss$.
\item[(2)] There exists an automorphism of $X$ that swaps the connected components of $X$. 
\end{itemize}
\item[(ii)] In general $\Ss$ is the projection of an algebraic set $X\subset\R^{n+k}$ (for some $k\geq0$) of dimension $d$ that is Nash diffeomorphic to a finite pairwise disjoint union (of dimension $d$) of spheres (maybe of different dimensions).
\end{itemize}
\end{cor}

Even for dimension $1$, it is not possible to impose the connectedness of $X$ (see Lemma \ref{nint} and Example \ref{nint2}). Contrast the previous result with \cite[Cor.1.8]{fe3}.

\subsection{Structure of the article}

The article is organized as follows. In Section \ref{s2} we recall some simple models (`equivalent' to the closed unit ball) to represent in Section \ref{s3} $d$-dimensional compact semialgebraic sets $\Ss\subset\R^n$ connected by analytic paths as images of such simple models under Nash maps $\R^d\to\R^n$. We also recall a procedure to construct polynomial paths inside $d$-dimensional semialgebraic sets $\Ss\subset\R^d$ connected by analytic paths that passes through certain control points at certain control times (Lemma \ref{icsl}). In Section \ref{s3} we prove Proposition \ref{curves} and Theorem \ref{main1}, which is the core result of this article and requires the development of our most sophisticated techniques. Such techniques are strongly inspired by those proposed in \cite[Thm.1.3]{fu6} to prove Theorem \ref{main2-i}. In Section \ref{s4} we prove Theorems \ref{nice0} and \ref{nice2}, whereas Theorem \ref{snmbss} is proved in Section \ref{s5}. Finally in Section \ref{s6} we approach Corollaries \ref{consq1} and \ref{consq2}. We also show that Corollary \ref{consq2} is somehow sharp (Lemma \ref{nint} and Example \ref{nint2}). 

\subsection*{Acknowledgements}
The authors are very grateful to S. Schramm for a careful reading of the final version and for the suggestions to refine its redaction.

\section{Compact models and preliminary results}\label{s2}

In this section we present (most times without proofs, but with the corresponding references) some objects, tools and results that will be useful in the development of this article.

\subsection{Compact models}\label{acm}

Our first purpose is to present simple compact models to represent compact semialgebraic sets connected by analytic paths as their images under Nash maps. We analyze some relationships between these models, which will allow us to choose the most suitable one in each case. More precisely, in \cite{fu6} we found surjective polynomial and regular maps between the following compact models: 
\begin{itemize}
\item the \textit{standard sphere} $\sph^d:=\{x\in\R^{d+1}\, :\, \|x\|^2=1\}$, 
\item the \textit{closed unit ball} $\ol{\Bb}_d:=\{x\in \R^d\, :\, \|x\|^2\leq 1\}$, 
\item the \textit{cylinder} $\Cc_d:=\ol{\Bb}_{d-1}\times[-1,1]$, 
\item the \textit{hypercube} $\Qq_d:=[-1,1]^d$,
\item the \textit{standard simplex} $\Delta_{d}:=\{x\in \R^d\, :\, \x_1\geq 0,\ldots,\x_{d}\geq 0,\, \x_1+\cdots+\x_{d}\leq 1\}$,
\item the \textit{simplicial prism} $\Delta_{d-1}\times[-1,1]$.
\end{itemize}

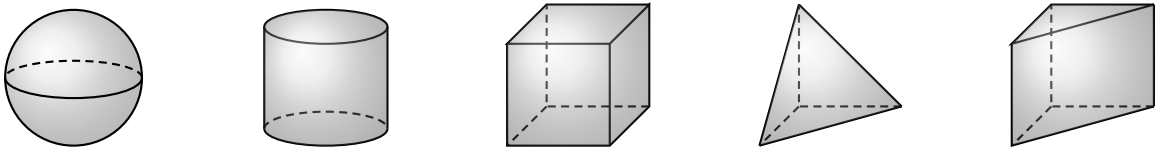
\begin{figure}[ht]
\centering
\begin{subfigure}[b]{0.17\textwidth}
\centering
\begin{tikzpicture}[scale=0.45]
\shade[ball color = gray!40, opacity = 0.4] (0,0) circle (2cm);
\draw[thick] (0,0) circle (2cm);
\draw[thick] (-2,0) arc (180:360:2 and 0.6);
\draw[densely dashed, thick] (2,0) arc (0:180:2 and 0.5);
\end{tikzpicture}
\end{subfigure}
\hfill
\begin{subfigure}[b]{0.17\textwidth}
\centering
\begin{tikzpicture}[scale=0.45]
\draw[thick] (0,0) ellipse (1.8 and 0.5);
\draw[thick] (-1.8,0) -- (-1.8,-3);
\draw[thick] (-1.8,-3) arc (180:360:1.8 and 0.5);
\draw[densely dashed, thick] (-1.8,-3) arc (180:360:1.8 and -0.5);
\draw[thick] (1.8,-3) -- (1.8,0);
\shade[ball color = gray!40, opacity = 0.4] (-1.8,0) -- (-1.8,-3) arc (180:360:1.8 and 0.5) -- (1.8,0) arc (0:180:1.8 and -0.5) -- (0,0) ellipse (1.8 and 0.5) ;
\end{tikzpicture}
\end{subfigure}
\hfill
\begin{subfigure}[b]{0.17\textwidth}
\centering
\begin{tikzpicture}[scale=0.45]
\draw[thick](3,3,0)--(0,3,0)--(0,3,3)--(3,3,3)--(3,3,0)--(3,0,0)--(3,0,3)--(0,0,3)--(0,3,3);
\draw[thick](3,3,3)--(3,0,3);
\draw[densely dashed, thick](3,0,0)--(0,0,0)--(0,3,0);
\draw[densely dashed, thick](0,0,0)--(0,0,3);
\shade[ball color = gray!40, opacity = 0.4] (0,3,0) -- (0,3,3) -- (0,0,3) -- (3,0,3) -- (3,0,0) -- (3,3,0);
\end{tikzpicture}
\end{subfigure}
\hfill
\begin{subfigure}[b]{0.17\textwidth}
\centering
\begin{tikzpicture}[scale=0.45]
\draw[thick](3,0,0) -- (0,3,0);
\draw[thick](3,0,0)--(0,0,3);
\draw[thick](0,0,3) -- (0,3,0);
\draw[densely dashed, thick](3,0,0)--(0,0,0)--(0,3,0);
\draw[densely dashed, thick](0,0,0)--(0,0,3);
\shade [ball color = gray!40, opacity = 0.4] (0,3,0) -- (0,0,3) -- (3,0,0);
\end{tikzpicture}
\end{subfigure}
\hfill
\begin{subfigure}[b]{0.17\textwidth}
\centering
\begin{tikzpicture}[scale=0.45]
\draw[densely dashed, thick](3,0,0)--(0,0,0)--(0,3,0);
\draw[densely dashed, thick](0,0,0)--(0,0,3);
\draw[thick] (3,0,0)--(3,3,0);
\draw[thick] (3,3,0)--(0,3,0) -- (0,3,3) -- (3,3,0);
\draw[thick] (0,3,3)--(0,0,3)--(3,0,0);
\shade [ball color = gray!40, opacity = 0.4] (3,3,0) -- (0,3,0) -- (0,3,3) -- (0,0,3)-- (3,0,0)--(3,3,0);
\end{tikzpicture}
\end{subfigure}
\caption{\small{Compact models to represent semialgebraic sets as their Nash images.}}
\end{figure}

The cylinder $\Cc_d$ is by \cite[Lem.2.1]{fu6} a polynomial image of $\ol{\Bb}_d$, whereas the standard $d$-dimensional simplex $\Delta_{d}$ is by \cite[Lem.2.5]{fu6} a polynomial image of $\ol{\Bb}_d$. The simplicial prism $\Delta_{d-1}\times[-1,1]$ is by \cite[Cor.2.8]{fu6} a polynomial image of $\ol{\Bb}_d$, whereas the hypercube $\Qq_d:=[-1,1]^d$ is by \cite[Cor.2.9]{fu6} a polynomial image of $\ol{\Bb}_d$. Conversely, the $d$-dimensional closed ball $\ol{\Bb}_d$ is by \cite[Lem.2.10]{fu6} a polynomial image of the $d$-dimensional hypercube $\Qq_d$. The $d$-sphere is by \cite[Lem.A.4]{fu6} a regular image of the hypercube $\Qq_d$ (although it is not a polynomial image \cite[\S.1.2]{fu6}) and the closed ball $\ol{\Bb}_d$ is the projection of the $d$-sphere. We complete below some missing relations between the previous models for the sake of completeness. As a consequence, we can choose up to our convenience any of the previous models to represent a semialgebraic set $\Ss\subset\R^n$ connected by analytic paths as a Nash image and immediately we know that $\Ss$ is a Nash image of each of them. 

\begin{cor}\label{cyl}
The $d$-dimensional closed ball $\ol{\Bb}_d$ is a polynomial image of the $d$-dimensional cylinder $\Cc_d$.
\end{cor}
\begin{proof}
If $d=1$, we have $\ol{\Bb}_0\times[-1,1]=\ol{\Bb}_1=[-1,1]$, so we can consider the case $d\geq 2$. The hypercube $\Qq_{d-1}:=[-1,1]^d$ is by \cite[Cor.2.9]{fu6} a polynomial image of $\ol{\Bb}_{d-1}$, so the hypercube $\Qq_d$ is a polynomial image of the $d$-dimensional cylinder $\Cc_d=\ol{\Bb}_{d-1}\times[-1,1]$. The $d$-dimensional closed ball $\ol{\Bb}_d$ is by \cite[Lem.2.10]{fu6} a polynomial image of the $d$-dimensional hypercube $\Qq_d$, so it is also a polynomial image of the $d$-dimensional cylinder $\Cc_d$, as required. 
\end{proof}

\begin{lem}\label{triangolo}
The $d$-dimensional closed ball $\ol{\Bb}_d$ is a polynomial image of the $d$-dimensional simplex $\Delta_d$.
\end{lem}
\begin{proof}
As $\Delta_1=[0,1]$ and $\ol{\Bb}_1=[-1,1]$, the polynomial function $h(\t)=2\t-1$ satisfies $h(\Delta_1)=[-1,1]$, so we assume $d\geq 2$. We proceed similarly to the proof of \cite[Lem.2.10]{fu6} and we consider the univariate polynomial 
$$
h(\t):=\t^2\frac{(\t-2d^2)^{2(2d^2-1)}}{(2d^2-1)^{2(2d^2-1)}}\in \R[\t]\quad\leadsto\quad h'(\t)=\frac{4d^2\t(\t-2d^2)^{2(2d^2-1)-1}}{(2d^2-1)^{2(2d^2-1)}}(\t-1).
$$
It satisfies $h(0)=h(2d^2)=0$ and $h(1)=1$. In addition, $h'$ is positive on $(0,1)$ and negative on $(1,2d^2)$, so $h$ has a global maximum at $t=1$ and it satisfies $0\leq h(\t)\leq 1$ on the interval $[0,2d^2]$. 

Consider the simplex $\Delta'_d:=\{\x_1\geq-1,\ldots,\x_d\geq -1,\x_1+\cdots+\x_d\leq\sqrt{d}\}$. A tangent hyperplane to $\sph^{d-1}$ (which is the boundary of the closed unit ball $\ol{\Bb}_d$) is parallel to $\x_1+\cdots+\x_d=0$ if and only if the tangent point $p\in\sph^{d-1}$ has all its coordinates equal. As $\Delta'_d\subset\{\x_1\geq-1,\ldots,\x_d\geq-1\}$, we pick the point $p=(\frac{1}{\sqrt{d}},\ldots,\frac{1}{\sqrt{d}})$ and the tangent hyperplane $\{\x_1+\cdots+\x_d=\sqrt{d}\}$, so $\ol{\Bb}_d\subset\Delta'_d$. In addition, we claim: $\Delta'_d\subset \ol{\Bb}_d(0,\sqrt{2}d)$. 

As $\Delta'_d$ is the convex hull of its vertices and $\ol{\Bb}_d(0,\sqrt{2}d)$ is convex, it is enough to check that the vertices of $\Delta'_d$ belong to $\ol{\Bb}_d(0,\sqrt{2}d)$. The vertices of $\Delta'_d$ are 
$$
v_i:=(-1,\ldots,-1,\sqrt{d}+d-1,-1,\ldots,-1). 
$$
We have 
$$
\|v_i\|^2=d-1+(\sqrt{d}+d-1)^2=d^2+2\sqrt{d}(d-1)<2d^2,
$$ 
so $\ol{\Bb}_d\subset\Delta'_d\subset \ol{\Bb}_d(0,\sqrt{2}d)$. Consider the polynomial map 
$$
g:\R^d\to\R^d, \, x\mapsto h(\|x\|^2)x.
$$
Observe that $g(\ol{\Bb}_d)=g(\ol{\Bb}_d(0,\sqrt{2}d))=\ol{\Bb}_d$, so $g(\Delta'_d)=\ol{\Bb}_d$. If $h:\R^d\to\R^d$ is an affine map such that $h(\Delta_d)=\Delta'_d$ and the polynomial map $f:=g\circ h$ satisfies $f(\Delta_d)=\ol{\Bb}_d$, as required.
\end{proof}

\begin{cor}\label{prismball}
The $d$-dimensional closed ball $\ol{\Bb}_d$ is a polynomial image of the $d$-dimensional prism $\Delta_{d-1}\times[-1,1]$. 
\end{cor}
\begin{proof}
If $d=1$, we have $\Delta_0\times[-1,1]=\ol{\Bb}_1=[-1,1]$, so we consider the case $d\geq 2$. By Lemma \ref{triangolo} $\ol{\Bb}_{d-1}$ is a polynomial image of $\Delta_{d-1}$, so $\ol{\Bb}_{d-1}\times[-1,1]$ is a polynomial image of $\Delta_{d-1}\times[-1,1]$. By Corollary \ref{cyl} we conclude that $\ol{\Bb}_d$ is a polynomial image of $\Delta_{d-1}\times[-1,1]$, as required.
\end{proof}

\subsection{Necessary conditions}
All the (equivalent) models quoted above are by Theorem \ref{main0} Nash images of $\R^n$. Thus, Nash images of the previous models are, apart from compact, connected by analytic paths \cite[Cor.6.3]{fe3}, pure dimensional \cite[Cor.6.3]{fe3} and irreducible \cite[Lem.7.3]{fe3}. The following example borrowed from \cite[Ex.7.12]{fe3} (studied in detail in \cite[Ex.1.2]{fu6}) confirms that `pure dimensionality' and `irreducibility' are not enough to guarantee `connection by Nash paths'.

\begin{example}[{\cite[Ex.7.12]{fe3}, \cite[Ex.1.2]{fu6}}]
The irreducible and pure dimensional semialgebraic set (see Figure \ref{notconnected}) $\Ss:=\{(4\x^2-\y^2)(4\y^2-\x^2)\geq0,\, \y\geq0\}\subset\R^2$ \textit{is not connected by Nash paths.}
\end{example}

In addition, recall that by Theorem \ref{main0} the following properties for a semialgebraic set $\Ss\subset\R^n$ are equivalent:
\begin{itemize}
\item[(i)] $\Ss$ is connected by Nash paths.
\item[(ii)] $\Ss$ is connected by analytic paths.
\item[(iii)] $\Ss$ is pure dimensional and there exists a Nash path $\alpha:[0,1]\to\Ss$ whose image meets all the connected components of the set of regular points of $\Ss$.
\item[(iv)] $\Ss$ is pure dimensional and there exists an analytic path $\alpha:[0,1]\to\Ss$ whose image meets all the connected components of the set of regular points of $\Ss$.
\end{itemize}

\begin{figure}[!ht]
\begin{center}
\begin{tikzpicture}[scale=0.75]
%\draw[style=help lines,step=0.5cm] (0,0) grid (20,6);

\draw[fill=gray!60,opacity=0.75,dashed,draw] (4.5,1) -- (0.5,3) arc (153.43494882292201:116.56505117707798:4.47213595499958cm) -- (4.5,1) -- (8.5,3) arc (26.56505117707799:63.43494882292202:4.47213595499958cm) -- (4.5,1);

\draw[->,thick=1.5pt] (4.5,1) -- (0.5,3);
\draw[->,thick=1.5pt] (4.5,1) -- (6.5,5);
\draw[->,thick=1.5pt] (4.5,1) -- (8.5,3);
\draw[->,thick=1.5pt] (4.5,1) -- (2.5,5);

\draw[->] (4.5,0) -- (4.5,6);
\draw[->] (0,1) -- (9,1);

\draw[fill=black] (4.5,1) circle (0.75mm);
\draw[fill=black] (5.5,2) circle (0.75mm);
\draw[fill=black] (3.5,2) circle (0.75mm);

\draw (6.1,2.4) node{\footnotesize$(1,1)$};
\draw (2.9,2.4) node{\footnotesize$(-1,1)$};

\draw (2,4.25) node{\small$\Cc_1$};
\draw (7,4.25) node{\small$\Cc_2$};
\draw (5,5) node{\small$\Ss$};
\end{tikzpicture}
\end{center}
\caption{\small{The semialgebraic set $\Ss:=\{(4\x^2-\y^2)(4\y^2-\x^2)\geq0,\, \y\geq0\}\subset\R^2$ (figure borrowed from \cite[Fig.1.1]{fu6})}\label{notconnected}}
\end{figure}

\subsection{Checkerboard sets}\label{checkerb}

Let $X\subset Y\subset\R^n$ be algebraic sets such that $Y$ is non-singular and has dimension $d$. Recall that $X$ is a \em normal-crossings divisor of $Y$ \em if for each point $x\in X$ there exists a regular system of parameters $\x_1,\ldots,\x_d$ for $Y$ at $x$ such that $X$ is given on an open Zariski neighborhood of $x$ in $Y$ by the equation $\x_1\cdots \x_k=0$ for some $k\leq d$. In particular, the irreducible components of $X$ are non-singular and have codimension $1$ in $Y$. If $\Ss\subset\R^m$ is a semialgebraic set, we write $\partial \Ss:=\cl(\Ss)\setminus\Reg(\Ss)$, which is in general different from the set $\Sing(\Ss):=\Ss\setminus\Reg(\Ss)$ presented in the Introduction. We denote $\ol{\ \cdot\ }^{\zar}$ the Zariski closures operator. 

A pure dimensional semialgebraic set $\Tt\subset\R^n$ is a \textit{checkerboard set} if it satisfies the following properties:

\begin{itemize}
\item $\ol{\Tt}^{\zar}$ is a non-singular algebraic set.
\item $ \ol{\partial\Tt}^{\zar}$ is a normal-crossings divisor of $\ol{\Tt}^{\zar}$.
\item $\Reg(\Tt)$ is connected.
\end{itemize}

Each checkerboard set is connected by analytic paths \cite[Main Thm.1.4, Lem.8.2]{fe3}. We will use in our proof of Theorem \ref{main1} the following result from \cite{fe3} in an essential way.

\begin{thm}[{\cite[Thm.8.4]{fe3}}]\label{ridwell}
Let $\Ss\subset\R^m$ be a semialgebraic set connected by analytic paths of dimension $d\geq 2$. Then there exists a checkerboard set $\Tt\subset\R^n$ of dimension $d$ and a proper regular map $f:\ol{\Tt}^{\zar}\to\ol{\Ss}^{\zar}$ such that $f(\Tt)=\Ss$.
\end{thm}

\subsubsection{Reduction to the case of checkerboard sets.}\label{ridchekcer7}

In order to prove Theorem \ref{main1}, we `only' need to prove the following: \textit{If $\Ss\subset\R^m$ is a compact semialgebraic set connected by analytic paths of dimension $d$, there exists a Nash map $f:\R^d\to\R^m$ such that $f(\ol{\Bb}_d)=\Ss$.}

We will prove Theorem \ref{main1} for dimension 1 in Subsection \ref{dimensione1}, so let us assume $\dim(\Ss)\geq 2$. In this case Theorem \ref{ridwell} provides a checkerboard set $\Tt\subset\R^n$ and a proper regular map $f:\ol{\Tt}^{\zar}\to\ol{\Ss}^{\zar}$ such that $f(\Tt)=\Ss$. As the map $f$ is proper, if the semialgebraic set $\Ss$ is compact, we may assume that also the checkerboard set $\Tt$ is compact (see the proof of \cite[Thm.8.4]{fe3}). Consequently, we are reduced to prove the following:

\begin{thm}\label{riduzione0}
Let $\Tt\subset\R^n$ be a compact checkerboard set of dimension $d\geq 2$. Then there exists a Nash map $G:\R^d\to\R^n$ such that $G(\ol{\Bb}_d)=\Tt$.
\end{thm}
 
By \cite[Cor.2.8]{fu6} there exists a polynomial map $f:\R^d\to \R^d$ such that $f(\ol{\Bb}_d)=\Delta_{d-1}\times [0,1]$. Consider the inverse of the stereographic projection
$$
\varphi:\R^d\to\sph^d\setminus\{(0,\ldots,1)\},\ 
x:=(x_1,\ldots,x_d)\mapsto\Big(\frac{2x_1}{1+\|x\|^2},\ldots,\frac{2x_d}{1+\|x\|^2},\frac{-1+\|x\|^2}{1+\|x\|^2}\Big)
$$
and let $\pi:\R^{d+1}\to \R^d$ be the projection onto the first $d$ coordinates. The regular map $g:=\pi\circ\varphi:\R^d\to\R^d$ satisfies $g(\R^d)=g(\ol{\Bb}_d)=\ol{\Bb}_d$. If there exists a Nash map $F:\Delta_{d-1}\times [0,1]\to\R^n$ such that $F(\Delta_{d-1}\times [0,1])=\Tt$, the composition $G:=F\circ f\circ g:\R^d\to\R^d$ is a well-defined Nash map such that $G(\ol{\Bb}_d)=\Tt$.

Thus, in order to show Theorem \ref{riduzione0}, we can use the (more convenient) compact model $\Delta_{d-1}\times[0,1]$ and we are reduced to show the following:

\begin{thm}\label{riduzione}
Let $\Tt\subset\R^n$ be a compact checkerboard set of dimension $d\geq 2$. Then there exists a Nash map $F:\Delta_{d-1}\times [0,1]\to\R^n$ such that $F(\Delta_{d-1}\times [0,1])=\Tt$.
\end{thm}

\subsection{Polynomial paths inside semialgebraic sets}\label{paths}

We recall next a smart polynomial curve selection lemma (\cite[Thm.1.6]{fe4} and \cite[Lem.3.1]{fu5}). It allows to approximate continuous semialgebraic paths inside the closure of an open semialgebraic set by polynomial paths, with strong control on the derivatives. This lemma will be one of the main ingredients in our proof of Theorem \ref{main1}. In \cite[Thm.1.6]{fe4} and \cite[Lem.3.1]{fu5} an extended study of polynomial and Nash paths inside the closure of open semialgebraic sets is made. We only need a simplified version of the results obtained in \cite[Thm.1.6]{fe4} and \cite[Lem.3.1]{fu5} that we state in Lemma \ref{icsl}. 

We endow the space $\Cont^\nu([a,b],\R)$ of differentiable functions of class $\Cont^\nu$ on the interval $[a,b]$ with the $\Cont^\nu$ compact-open topology. Recall that a basis of open neighborhoods of $g\in\Cont^\nu([a,b],\R)$ in this topology is constituted by the sets of the type:
$$
{\mathcal U}^\nu_{g,\veps}:=\{f\in\Cont^\nu([a,b],\R):\ \|f^{(\ell)}-g^{(\ell)}\|_{[a,b]}<\veps:\ \ell=0,\ldots,\nu\}
$$
where $\veps>0$ and $\|h\|_{[a,b]}:=\max\{h(x):\, x\in[a,b]\}$. One has $\Cont^\nu([a,b],\R^n)=\Cont^\nu([a,b],\R)\times\cdots\times\Cont^\nu([a,b],\R)$ and we endow this space with the product topology. If $X\subset[a,b]$, one analogously defines the $\Cont^\nu$ compact-open topology of the space $\Cont^\nu(X,\R^n)$. 

If $\alpha:[a,b]\to\R^n$ is a continuous semialgebraic path, recall that by \cite[Prop.2.9.10]{bcr} there exists a finite set $\eta(\alpha)\subset[a,b]$ such that $\alpha$ is not Nash at the points of $\eta(\alpha)$, but $\alpha|_{[a,b]\setminus\eta(\alpha)}$ is a Nash map. We denote the Taylor expansion of degree $\ell\geq1$ of $\alpha$ at $t_0\in[a,b]\setminus\eta(\alpha)$ with $T_{t_0}^\ell\alpha:=\sum_{k=0}^\ell\frac{1}{\ell!}\alpha^{(k)}(t_0)(\t-t_0)^k$.

\begin{lem}[Smart polynomial curve selection lemma]\label{icsl}
Let $\Ss\subset\R^n$ be an open semialgebraic set and $\{p_1,\ldots,p_r\}\subset\cl(\Ss)$ a finite set of points, not necessarily distinct. Let $0<t_1<\cdots<t_r<1$ and $\alpha:[0,1]\to\Ss\cup\{p_1,\ldots,p_r\}$ be a continuous semialgebraic path such that $\alpha(t_i)=p_i$ for $i=1,\ldots, r$ that satisfies $\eta(\alpha)\cap\{t_1,\ldots,t_r\}=\varnothing$ and $\alpha([0,1]\setminus\{t_1,\ldots,t_r\})\subset\Ss$. For each $\veps>0$ and each $m\geq0$ there exists a polynomial path $\beta:[0,1]\to\Ss\cup\{p_1,\ldots,p_r\}$ such that: $\|\alpha^{(k)}-\beta^{(k)}\|<\veps$ for $k=0,\ldots,m$, $T^m_{t_i}\beta=T^m_{t_i}\alpha$ for $i=1,\ldots,r$ and $\beta([0,1]\setminus\{t_1,\ldots,t_r\})\subset\Ss$.
\end{lem}

\section{Building Nash images of the simplicial prism}\label{s3}

The purpose of this section is to prove Theorem \ref{riduzione}, which provides a complete characterization of the Nash images of the closed ball. The proof is quite involved and intricate and we begin with some preliminary results to lighten the proof. 

We will start with the 1-dimensional case, that requires a different proof. Then we will focus on the $d$-dimensional case for $d\geq 2$. For the general case we will take advantage of the fact that each checkerboard set $\Tt\subset\R^n$ admits `nice' triangulations. Roughly speaking, we `build' $\Tt$ as Nash image of the prism $\Delta_{d-1}\times [0,1]$ `simplex by simplex'. 

We consider a suitable subset of the space of Nash maps $\mathcal{N}(\R^d,\R^n)$ and we will parameterize it (of course not in an injective way) using an open semialgebraic set $\Theta_0$ of a large affine space. In this space, a continuous semialgebraic path $\sigma:[0,1]\to\Theta_0$ provides a continuous semialgebraic map $\Delta_{d-1}\times[0,1]\to\R^n$ that is Nash on the horizontal slices $\Delta_{d-1}\times\{t\}$. Using Lemma \ref{icsl}, we approximate the path $\sigma$ by a polynomial path in order to obtain a Nash map $\Delta_{d-1}\times[0,1]\to\R^n$. A difficult point is to guarantee that the obtained Nash map has $\Tt$ as its target space and that it is surjective.

\subsection{The $1$-dimensional case.}\label{dimensione1}
Nash images of closed balls contained in the real line are its compact intervals and all of them are affinely equivalent to the interval $\ol{\Bb}_1:=[-1,1]$. Nash images of closed balls contained in a circumference are its connected compact subsets and all of them are Nash images of $\ol{\Bb}_1$.

\begin{examples}\label{circle}
(i) The circumference $\sph^1:=\{\x^2+\y^2=1\}$ is a Nash image of $\ol{\Bb}_1$. Consider the inverse of the stereographic projection from the point $(0,1)$, which is the map
$$
f:\R\to\sph^1\setminus\{(0,1)\},\ t\mapsto\Big(\frac{2t}{1+t^2},\frac{1-t^2}{1+t^2}\Big).
$$
Next, we identify $\R^2$ with $\C$ and the coordinates $(x,y)$ with $x+\sqrt{-1}y$. Consider the map 
$$
g:\C\to\C,\ z:=x+\sqrt{-1}y\mapsto z^2=(x^2-y^2)+\sqrt{-1}(2xy).
$$
The image of $\ol{\Bb}_1$ under $g\circ f$ is $\sph^1$. 

(ii) Any connected compact proper subset $\Ss$ of $\sph^1$ that is not a point is a Nash image of $\ol{\Bb}_1$ because it is Nash diffeomorphic to $[-1,1]$.
\end{examples}

We prove Proposition \ref{curves} next:

\begin{proof}[Proof of Proposition \em \ref{curves}]
Assume $\Ss$ is irreducible. Let $X$ be the Zariski closure of $\Ss$ in $\R^n$ and $\widetilde{X}$ its complexification in $\C^n$. Let $(\widetilde{Y},\pi)$ be the normalization of $\widetilde{X}$ and $\widehat{\sigma}$ the involution of $\widetilde{Y}$ induced by the involution $\sigma$ of $\widetilde{X}$ that arises from the restriction to $\widetilde{X}$ of the complex conjugation in $\C^n$. We may assume that $\widetilde{Y}\subset\C^m$ and that $\widehat{\sigma}$ is the restriction to $\widetilde{Y}$ of the complex conjugation of $\C^m$. By \cite[Thm.3.15]{fg3} and since $\Ss$ is irreducible, $\pi^{-1}(\Ss)$ has a $1$-dimensional connected component $\Tt$ such that $\pi(\Tt)=\Ss$. As $\pi$ is proper and $\Ss$ is compact, also $\Tt$ is compact. As $X$ has dimension $1$, it is a coherent analytic set, so $\Tt\subset Y:=\widetilde{Y}\cap\R^m$. As $\widetilde{Y}$ is a normal-curve, $Y$ is a non-singular real algebraic curve. We claim: \em the connected components of $Y$ are Nash diffeomorphic either to $\sph^1$ or to the real line $\R$\em.

By \cite[Thm.VI.2.1]{sh} there exist a compact affine non-singular real algebraic curve $Z$, a finite set $F$, which is empty if $Y$ is compact, and a union $Y'$ of some connected components of $Z\setminus F$ such that $Y$ is Nash diffeomorphic to $Y'$ and $\cl(Y')$ is a compact Nash curve with boundary $F$. As $Z$ is a compact affine non-singular real algebraic curve, its connected components are diffeomorphic to $\sph^1$, so by \cite[Thm.VI.2.2]{sh} the connected components of $Z$ are in fact Nash diffeomorphic to $\sph^1$. Now, each connected component of $Y$ is Nash diffeomorphic to an open connected subset of $\sph^1$, that is, Nash diffeomorphic either to $\sph^1$ or to the real line $\R$, as claimed.

As $\Tt$ is connected, compact and 1-dimensional, it is Nash diffeomorphic to a connected compact 1-dimensional semialgebraic subset of either $\sph^1$ or $\R$, so $\Tt$ is Nash diffeomorphic either to $\sph^1$ or the compact interval $[-1,1]$. By Examples \ref{circle} the semialgebraic set $\Tt$ is a Nash image of $\ol{\Bb}_1$, so $\Ss$ is also a Nash image of $\ol{\Bb}_1$. The converse follows from \cite[Lem.7.3]{fe3}, as required.
\end{proof}

\subsection{Covering simplices with Nash maps.} 

Given a convex polyhedron $\pol$, we denote its relative interior with $\Int(\pol)$ and its boundary $\pol\setminus\Int(\pol)$ with $\partial\pol$. We start the procedure to prove Theorem \ref{riduzione} with some lemmas that will allow us to cover simplices with the images of suitable family of Nash maps whose domains are simplicial prisms. Denote 
\begin{equation}\label{simplexn1}
\Delta_{n-1}:=\Big\{(\lambda_1,\ldots,\lambda_n)\in\R^n:\ \lambda_1\geq0,\ldots,\lambda_n\geq0,\ \sum_{k=1}^n\lambda_k=1\Big\}.
\end{equation}
The boundary $\partial\Delta_{n-1}=\bigcup_{i=1}^n(\Delta_{n-1}\cap\{\lambda_i=0\})$.

\begin{lem}\label{simplex}
Consider an $(n-1)$-dimensional simplex $\sigma\subset\R^n$ of vertices $v_1,\ldots,v_n$. Pick a point $p\in\R^n\setminus\sigma$ and consider the $n$-dimensional simplex $\widehat{\sigma}$ of vertices $\{p,v_1,\ldots,v_n\}$. Let $F:\Delta_{n-1}\times[0,1]\to\R^n$ be a continuous semialgebraic map such that $F|_{\Delta_{n-1}\times\{0\}}:\Delta_{n-1}\times\{0\}\to\sigma$ is a homeomorphism, $F(\partial\Delta_{n-1}\times(0,1))\subset\{h_0\geq0\}\setminus\widehat{\sigma}$ and $F(\Delta_{n-1}\times\{1\})=\{p\}$. Then $\Int(\widehat{\sigma})\subset F(\Int(\Delta_{n-1})\times(0,1))$ and $\widehat{\sigma}\subset F(\Delta_{n-1}\times[0,1])$.
\end{lem}
\begin{proof}
As $\Delta_{n-1}\times[0,1]$ is compact and $\widehat{\sigma}=\cl(\Int(\widehat{\sigma}))$, it is enough to check: $\Int(\widehat{\sigma})\subset F(\Int(\Delta_{n-1})\times(0,1))$. 

Suppose there exists $z\in\Int(\widehat{\sigma})\setminus F(\Delta_{n-1}\times[0,1])$. Let us construct a (continuous) semialgebraic retraction $\rho:\R^n\setminus\{z\}\to\partial\widehat{\sigma}$. For each $x\in\R^n\setminus\{z\}$ let $\ell_x$ be the ray $\{z+t(x-z):\ t\in[0,+\infty)\}$. By \cite[11.1.2.3, 11.1.2.7]{ber1} $\ell_x\cap\partial\widehat{\sigma}=\{\rho(x)\}$ is a singleton and if $x\in\partial\widehat{\sigma}$, then $\rho(x)=x$. Define $\rho:\R^n\setminus\{z\}\to\partial\widehat{\sigma},\ x\mapsto\rho(x)$. Let $h_0,\ldots,h_n\in \R[\x]$ be polynomials of degree $1$ such that the hyperplanes $H_i:=\{h_i=0\}$ contain the facets of $\widehat{\sigma}$. Assume $\widehat{\sigma}\subset \{h_i\geq 0\}$ for $i=0,\ldots,n$ and $\sigma\subset H_0$. Note that $\rho(x)=z+\lambda(x-z)$, where $\lambda$ is the smallest value $\mu>0$ such that $h_i(z+\mu(x-z))=0$ for some $i=0,\ldots,n$. As $z\in\Int(\widehat{\sigma})$, we have $h_i(z)>0$ for $i=0,\ldots,n$. Thus, 
$$
\frac{1}{\lambda}=\max\Big\{\frac{h_i(z)-h_i(x)}{h_i(z)}:\ i=0,\ldots,n\Big\}>0.
$$
Consequently,
$$
\rho(x)=z+\frac{1}{\max\Big\{\frac{h_i(z)-h_i(x)}{h_i(z)}:\ i=0,\ldots,n\Big\}}(x-z),
$$
so $\rho:\R^n\setminus\{z\}\to\partial\widehat{\sigma}$ is a continuous semialgebraic map such that $\rho|_{\partial\widehat{\sigma}}=\id_{\partial\widehat{\sigma}}$, that is, $\rho$ is a semialgebraic retraction. 

Observe that
$$
\max\Big\{\frac{h_i(z)-h_i(x)}{h_i(z)}:\ i=0,\ldots,n\Big\}=\frac{h_j(z)-h_j(x)}{h_j(z)}
$$
for a given $j=0,\ldots,n$ if and only if $\rho(x)\in\{h_j=0\}$. In addition, if $x\in\{h_0\geq0\}\setminus(\Int(\widehat{\sigma})\cup\Int(\sigma))$, then $h_{i_0}(x)\leq 0$ for some $i_0=1,\ldots,n$. Thus, 
$$
\frac{h_0(z)-h_0(x)}{h_0(z)}\leq1\leq\frac{h_{i_0}(z)-h_{i_0}(x)}{h_{i_0}(z)}\leq\max\Big\{\frac{h_i(z)-h_i(x)}{h_i(z)}:\ i=0,\ldots,n\Big\},
$$
so $\rho(x)\not\in\{h_0=0,h_1>0,\ldots,h_n>0\}=\Int(\sigma)$. Consequently, $\rho^{-1}(\Int(\sigma))\subset\Int(\widehat{\sigma})\cup\Int(\sigma)$.

Consider the continuous semialgebraic map $F^*:=\rho\circ F:\Delta_{n-1}\times[0,1]\to\partial\widehat{\sigma}$. Let us prove: \em the restriction map $F^*|_{\partial(\Delta_{n-1}\times[0,1])}:\partial(\Delta_{n-1}\times[0,1])\to\partial\widehat{\sigma}$ has degree $1$ \em (as a continuous map between spheres of dimension $n-1$).

Pick a point $x\in\Int(\sigma)$. Then $(F^*)^{-1}(x)=F^{-1}(\rho^{-1}(x))\subset F^{-1}(\Int(\widehat{\sigma}))\cup F^{-1}(\Int(\sigma))$. As 
$$
\partial(\Delta_{n-1}\times[0,1])=(\partial\Delta_{n-1}\times(0,1))\cup(\Delta_{n-1}\times\{0\})\cup(\Delta_{n-1}\times\{1\}),
$$
$F|_{\Delta_{n-1}\times\{0\}}:\Delta_{n-1}\times\{0\}\to\sigma$ is a homeomorphism, $F(\partial\Delta_{n-1}\times(0,1))\cap\widehat{\sigma}=\varnothing$ and $F(\Delta_{n-1}\times\{1\})=\{p\}$, we deduce
\begin{align*}
&F^{-1}(\Int(\widehat{\sigma}))\cap\partial(\Delta_{n-1}\times[0,1])=\varnothing,\\
&F^{-1}(\Int(\sigma))\cap\partial(\Delta_{n-1}\times[0,1])\subset\Delta_{n-1}\times\{0\}.
\end{align*}
Consequently, the preimage 
$$
(F^*)^{-1}(x)\cap\partial(\Delta_{n-1}\times[0,1])=(F^*)^{-1}(x)\cap(\Delta_{n-1}\times\{0\}).
$$ 
As $F|_{\Delta_{n-1}\times\{0\}}:\Delta_{n-1}\times\{0\}\to\sigma$ is a homeomorphism and $\rho|_\sigma=\id_\sigma$, also $F^*|_{\Delta_{n-1}\times\{0\}}=F|_{\Delta_{n-1}\times\{0\}}:\Delta_{n-1}\times\{0\}\to\sigma$ is a homeomorphism, so the preimage $(F^*)^{-1}(x)\cap\partial(\Delta_{n-1}\times[0,1])=(F|_{\Delta_{n-1}\times\{0\}})^{-1}(x)$ is a singleton. As this happens for each $x\in\Int(\sigma)$, the restriction map $F^*|_{\partial(\Delta_{n-1}\times[0,1])}$ has degree $1$. As $F^*|_{\partial(\Delta_{n-1}\times[0,1])}$ extends continuously to $\Delta_{n-1}\times[0,1]$, we deduce by \cite[Thm.5.1.6(b)]{hir3} that $F^*|_{\partial(\Delta_{n-1}\times[0,1])}$ has degree $0$, which is a contradiction. 

Consequently, $\Int(\widehat{\sigma})\subset F(\Delta_{n-1}\times[0,1])$, as required.
\end{proof}

Given a polynomial $h\in\R[\x]$ of degree $1$ and the hyperplane $H:=\{h=0\}$ of $\R^n$, denote the two subspaces determined by $H$ with $H^+:=\{h\geq0\}$ and $H^-:=\{h\leq0\}$. Denote also $\vec{h}:=h-h(0)$. If $\pol:=\{g_1\geq0,\ldots,g_m\geq0\}\subset\R^n$ is an $n$-dimensional convex polyhedron, where each $g_i\in\R[\x]$ is a polynomial of degree $1$, then $\Int(\pol)=\{g_1>0,\ldots,g_m>0\}$. Thus, if $\pol_1,\ldots,\pol_s\subset\R^n$ are $n$-dimensional convex polyhedra, $\Int(\pol_1\cap\cdots\cap\pol_s)=\Int(\pol_1)\cap\cdots\cap\Int(\pol_s)$. The following construction will be useful for the proof of Lemma \ref{simplex2}, which is an application of Lemma \ref{simplex} and one of the keys to prove Theorem \ref{riduzione}.

\subsubsection{Covering of the exterior of a simplex}
Let $\pol\subset\R^n$ be an $n$-dimensional convex polyhedron and $\sigma\subset\partial\pol$ an $(n-1)$-dimensional simplex of vertices $v_1,\ldots,v_n$. Let $p\in\Int(\pol)$ and $\widehat{\sigma}$ be the $n$-simplex of vertices $\{p,v_1,\ldots,v_n\}$. Let $H_1,\ldots,H_n$ be the hyperplanes of $\R^n$ generated by the facets of $\widehat{\sigma}$ that contain $p$, which are those facets of $\widehat{\sigma}$ different from $\sigma$ and suppose $v_i\not\in H_i$. Assume that $\widehat{\sigma}\subset\bigcap_{j=1}^n H_j^+$ and consider the convex polyhedra $\pol_j:=\pol\cap\bigcap_{\ell\neq j}H_\ell^-$, (see Figure \ref{simplexfig}). Observe that $p\in\pol_j$ and $\dim(\Int(\pol_j))=n$, because $p\in\Int(\pol)$ and the hyperplanes $H_1,\ldots,H_n$ are affinely independent.

\begin{figure}[ht]
\centering
\begin{tikzpicture}[scale=0.85]
\centering
%\draw[style=help lines,step=0.5cm] (0,0) grid (8,6);
\draw[fill=gray!50,opacity=0.5,draw=none] (2,0) -- (4.5,0) -- (6,2) -- (5,4) -- (2.25,4.25) -- (0,2) -- (2,0);
\draw[fill=gray!90,opacity=0.5,draw=none] (2,0) -- (4.5,0) -- (3,2) -- (2,0);

\draw[fill=blue!25,opacity=0.5,draw=none] (2,0) -- (3,2) -- (4.05,4.1) -- (2.25,4.25) -- (0,2) -- (2,0);
\draw[fill=red!25,opacity=0.5,draw=none] (4.5,0) -- (6,2) -- (5,4) -- (2.25,4.25) -- (1.7,3.7) -- (3,2) -- (4.5,0);

\draw[thick] (2,0) -- (4.5,0) -- (6,2) -- (5,4) -- (2.25,4.25) -- (0,2) -- (2,0);

\draw[densely dotted, thick] (4.5,0) -- (1.7142,3.7142);
\draw[densely dotted, thick] (2,0) -- (4.05,4.1);

\draw (3,2) node{\tiny{\textbullet}};
\draw (2,-0.375) node{{\small $v_1$}};
\draw (4.5,-0.375) node{{\small $v_2$}};
\draw (3.25,-0.375) node{{\small $\sigma$}};
\draw (3.25,0.75) node{{\small $\widehat{\sigma}$}};
\draw (2.95,2.5) node{{\small $p$}};
\draw (1,4) node{{\small $\pol$}};
\draw (1.5,2) node{{\small $\pol_1$}};
\draw (4.5,2.3) node{{\small $\pol_2$}};
\end{tikzpicture}
\caption{\small{The polyhedra $\pol_j$ (figure inspired by \cite[Fig.4.2]{fu5}).\label{simplexfig}}}
\end{figure}
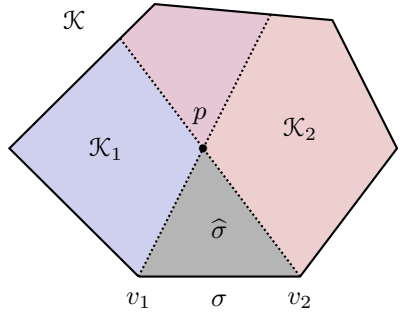

\subsubsection{$\Cont^k_{I'}$-topology}\label{ckt}
Let $I'\subset J\subset\R$ be compact intervals and define $\Cont^k_{I'}(\Delta_{n-1}\times J,\R^m)$ as the space of continuous functions on $\Delta_{n-1}\times J$ that are $\Cont^k$ differentiable functions with respect to $\t$ on $\Delta_{n-1}\times I'$. Observe that both $\Delta_{n-1}\times I'$ and $\Delta_{n-1}\times J$ are compact. For each $f\in\Cont^k_{I'}(\Delta_{n-1}\times J,\R)$ and $\veps>0$ define
$$
{\mathcal U}_{f,\veps}=\Big\{g\in\Cont^k_{I'}(\Delta_{n-1}\times J,\R):\ \|f-g\|<\veps,\ \Big\|\frac{\partial^\ell f}{\partial\t^\ell}-\frac{\partial^\ell g}{\partial\t^\ell}\Big\|_{\Delta_{n-1}\times I'}<\veps, \ell=1,\ldots,k\Big\}.
$$
The previous open sets are the basis of the {\em $\Cont^k_{I'}$-topology} of $\Cont^k_{I'}(\Delta_{n-1}\times J,\R^m)$.

\begin{lem}\label{simplex2}
Let $\pol:=\{g_1\geq0,\ldots,g_s\geq0\}\subset\R^n$ be an $n$-dimensional convex polyhedron and $\sigma\subset\pol$ an $(n-1)$-dimensional simplex of vertices $v_1,\ldots,v_n$. Fix $p\in\Int(\pol)$ and consider the simplex $\widehat{\sigma}$ of vertices $\{p,v_1,\ldots,v_n\}$. Let $H_i:=\{h_i=0\}$ be the hyperplanes of $\R^n$ generated by the facets of $\widehat{\sigma}$ that contain $p$ and assume $v_i\not\in H_i$ and $\widehat{\sigma}\subset\bigcap_{i=1}^nH_i^+$. Let $h_0\in\R[\t]$ be a polynomial of degree 1 such that $\sigma\subset\{h_0=0\}$ and $\widehat{\sigma}\subset\{h_0\geq0\}$. There exist continuous semialgebraic paths $\alpha_i:[-\delta,1+\delta]\to\pol$ (for some $\delta>0$) that are Nash on the compact neighborhood $I:=[-\delta,\delta]\cup[1-\delta,1+\delta]$ of $\{0,1\}$ and satisfy $\alpha_i(0)=v_i$ and $\alpha_i(1)=p$ for $i=1,\ldots,n$ and $\veps>0$ such that the continuous semialgebraic map 
$$
F:\Delta_{n-1}\times[-\delta,1+\delta]\to\pol,\ (\lambda_1,\ldots,\lambda_n,t)\mapsto\sum_{i=1}^n\lambda_i\alpha_i(t),
$$
(which is Nash on $\Delta_{n-1}\times I$), has the following property: 

If $G:\Delta_{n-1}\times[-\delta,1+\delta]\to\R^n$ is another continuous semialgebraic map that is a Nash map on $\Delta_{n-1}\times I'$ for a neighborhood $I'\subset I$ of $\{0,1\}$ and satisfies
$$
\frac{\partial^\ell G}{\partial\t^\ell}(\lambda,0)=\frac{\partial^\ell F}{\partial\t^\ell}(\lambda,0),\, \frac{\partial^\ell G}{\partial\t^\ell}(\lambda,1)=\frac{\partial^\ell F}{\partial\t^\ell}(\lambda,1)
$$ 
for each $\lambda\in\Delta_{n-1}$ and $\ell=0,1,2,3$, $\|G-F\|<\veps$ and $\|\frac{\partial^\ell G}{\partial\t^\ell}-\frac{\partial^\ell F}{\partial\t^\ell}\|_{\Delta_{n-1}\times I'}<\veps$ for $\ell=1,2,3$, then $\widehat{\sigma}\subset G(\Delta_{n-1}\times[0,1])$, $G(\Delta_{n-1}\times\{0\})=\sigma$, $G(\Delta_{n-1}\times\{1\})=\{p\}$, $G(\Delta_{n-1}\times(0,1))\subset\Int(\pol)$ and $G(\Delta_{n-1}\times([-\delta',0]\cup[1,1+\delta']))\subset\widehat{\sigma}$ for some $0<\delta'<\delta$ small enough.
\end{lem}
\begin{proof}
We assume $\sigma\subset\partial\pol$ after changing $\pol$ by $\pol':=\pol\cap\{h_0\geq0\}$, which is a convex polyhedron of dimension $n$ (that contains $\sigma$ in its boundary), because $\widehat{\sigma}\subset\pol'$ is an $n$-dimensional simplex. We keep the notation $\pol:=\{g_1\geq0,\ldots,g_s\geq0\}$ and we assume that $h_0$ coincides with some $g_i$. The proof is now conducted in several steps:

\vspace{2mm}
\noindent{\sc Initial preparation.} Let us construct the continuous semialgebraic paths $\alpha_i:[-\delta,1+\delta]\to\pol$. We claim: \em There exist continuous semialgebraic paths $\alpha_i:\R\to\pol$ such that: 
\begin{itemize}
\item[{\rm(i)}] $\alpha_i$ is Nash on $I$,
\item[{\rm(ii)}] $\alpha_i(\t)=v_i+\t^2u_i+\t^3w+\cdots$ and $\alpha_i(1+\t)=p-\t^3w+\cdots$, 
\item[{\rm(iii)}] $\alpha_i([-\delta,0)\cup(1,1+\delta])\subset\Int(\widehat{\sigma})$, 
\item[{\rm(iv)}] $\alpha_i((0,1))\subset\Ss_i:=\Int(\pol\cap\bigcap_{j\neq i}H_j^-)$,
\item[{\rm(v)}] $h_i\circ\alpha_i(\t)=h_i(v_i)+\vec{h_i}(u_i)\t^2-a_i\t^3+\cdots$ where $h_i(v_i)>0$, $\vec{h_i}(u_i)<0$ and $a_i>0$,
\item[{\rm(vi)}] $h_j\circ\alpha_i(\t)=-a_j\t^3+\cdots$ if $i\neq j$ and $h_j\circ\alpha_i(1+\t)=a_j\t^3+\cdots$, where $a_j>0$ and $1\leq i,j\leq n$,
\item[{\rm(vii)}] $h_0\circ\alpha_i(\t)=b_{i0}\t^2+\cdots$ where $b_{i0}>0$ and $1\leq i\leq n$,
\item[{\rm(viii)}] $g_k\circ\alpha_i(\t)=c_{ik}+d_{ik}\t^2+\cdots$ where either $c_{ik}>0$ or $c_{ik}=0$ and $d_{ik}>0$,
\item[{\rm(ix)}] $g_k\circ\alpha_i(1+\t)=e_{ik}+\cdots$ where $e_{ik}>0$.
\end{itemize}\em

We construct each continuous semialgebraic path $\alpha_i$ piecewise. The open semialgebraic set $\Ss_i$ defined in (iv) can be described as 
$$
\Ss_i=\{g_1>0,\ldots,g_s>0\}\cap\bigcap_{\substack{j\neq i \\ j\neq 0}}\{h_j<0\}.
$$
Define $u_i:=\vec{v_ip}$ and observe that $\vec{h}_j(u_i)=0$ if $1\leq i,j\leq n$ and $i\neq j$. This is because $h_j(p)=0$ and $h_j(v_i)=0$ if $1\leq i,j\leq n$ and $i\neq j$. Recall that $h_i(v_i)>0$ and $h_i(p)=0$, so $\vec{h_i}(u_i)<0$ for $1\leq i\leq n$. In addition, $b_{i0}:=\vec{h}_0(u_i)>0$, because $h_0(p)>0$ and $h_0(v_i)=0$ for $1\leq i\leq n$. As $g_k(v_i)\geq 0$ (because $\widehat{\sigma}\subset\pol$) and $g_k(v_i)+\vec{g}_k(u_i)=g_k(v_i+u_i)=g_k(p)>0$ (because $p\in \Int(\pol)$), we deduce that either $c_{ik}:=g_k(v_i)>0$ or $c_{ik}=0$ and $d_{ik}:=\vec{g}_k(u_i)>0$.

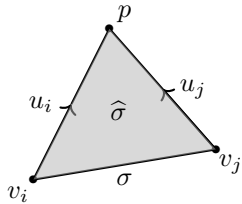
\begin{figure}[!ht]
\begin{center}
\begin{tikzpicture}[scale=0.4]
%\draw[style=help lines,step=1cm] (0,0) grid (12,8);
\draw (3,1) node{\tiny$\bullet$};
\draw (2.5,0.5) node{{\small $v_i$}};
\draw (9,2) node{\tiny$\bullet$};
\draw (9.5,1.5) node{\small $v_j$};
\draw (5.5,6) node{\tiny$\bullet$};
\draw (6,6.5) node{\small $p$};
\draw (3.25,3.5) node{\small $u_i$};
\draw (8.25,4) node{\small $u_j$};
\draw[thick] (9,2) -- (3,1);
\draw[->, thick] (3,1) -- (4.25,3.5);
\draw[thick] (4.25,3.5)--(5.5,6);
\draw[thick] (9,2)--(5.5,6);
\draw[->, thick] (9,2)--(7.25,4);
\draw[fill=gray!60,opacity=0.5,draw=none] (3,1)--(9,2)--(5.5,6);
\draw (6,1) node{\small $\sigma$};
\draw (5.8,3.3) node{\small $\widehat{\sigma}$};
\end{tikzpicture}
\end{center}
\vspace{-0.5cm}
\caption{\small{A picture of the situation.}}
\end{figure}

As $\{\vec{h}_1,\ldots,\vec{h}_n\}$ are independent linear forms, the open semialgebraic set $\bigcap_{j=1}^n\{\vec{h}_j<0\}\neq\varnothing$. Pick a non-zero vector $w\in\bigcap_{j=1}^n\{\vec{h}_j<0\}$ and write $a_j:=-\vec{h}_j(w)>0$ for $j=1,\ldots,n$. Consider the polynomial path
$$
\alpha_{i0}:\R\to\R^n,\ t\mapsto v_i+t^2u_i+t^3w.
$$
As $h_0(v_i)=0$ and $h_j(v_i)=0$ for $1\leq i,j\leq n$ if $i\neq j$, we deduce: 
\begin{align*}
(h_0\circ\alpha_{i0})(\t)&=h_0(v_i)+\vec{h}_0(u_i)\t^2+\vec{h}_0(w)\t^3=b_{i0}\t^2+\vec{h}_0(w)\t^3.\\
(h_j\circ\alpha_{i0})(\t)&=h_j(v_i)+\vec{h}_j(u_i)\t^2+\vec{h}_j(w)\t^3=\vec{h}_j(w)\t^3=-a_{j}\t^3\quad\text{if $i\neq j$},\\
(h_i\circ\alpha_{i0})(\t)&=h_i(v_i)+\vec{h_i}(u_i)\t^2+\vec{h_i}(w)\t^3=h_i(v_i)+\vec{h_i}(u_i)\t^2-a_i\t^3.
\end{align*}
In addition,
$$
(g_k\circ\alpha_{i0})(\t)=g_k(v_i)+\vec{g}_k(u_i)\t^2+\vec{g}_k(w)\t^3=c_{ik}+d_{ik}\t^2+\vec{g}_k(w)\t^3,
$$
where either $c_{ik}>0$ or both $c_{ik}=0$ and $d_{ik}>0$.

Consider the polynomial path
$$
\alpha_{i1}:\R\to\R^n,\ t\mapsto p-(t-1)^3w.
$$
Observe that
\begin{align*}
h_0\circ\alpha_{i1}(1+\t)&=h_0(p)-\vec{h}_0(w)\t^3,\\
h_j\circ\alpha_{i1}(1+\t)&=h_j(p)-\vec{h}_j(w)\t^3=a_j\t^3
\end{align*}
for $1\leq i,j\leq n$. As $a_j>0$, we have
\begin{equation}\label{123}
h_j\circ\alpha_{i0}(t)
\begin{cases}
<0&\text{if $t>0$,}\\
>0&\text{if $t<0$,}
\end{cases}\text{\, if\, } i\neq j
\text{\, and\, }
h_j\circ\alpha_{i1}(1+t)
\begin{cases}
>0&\text{if $t>0$,}\\
<0&\text{if $t<0$.}
\end{cases} 
\end{equation}
Denote $e_{ik}:=g_k(p)>0$ and observe that
$$
g_k\circ\alpha_{i1}(1+\t)=g_k(p)-\vec{g}_k(w)\t^3=e_{ik}-\vec{g}_k(w)\t^3.
$$
Let $0<\delta<\frac{1}{2}$ be such that $(h_0\circ\alpha_{i0})(t)>0$, $h_0\circ\alpha_{i1}(1+t)>0$, $(g_k\circ\alpha_{i0})(t)>0$ and $g_k\circ\alpha_{i1}(1+t)>0$ for $t\in[-\delta,\delta]\setminus\{0\}$ (recall that $h_0(p)>0$). Thus, by \eqref{123}, $\alpha_{i0}([-\delta,0)), \alpha_{i1}((1,1+\delta])\subset\Int(\widehat{\sigma})$ and $\alpha_{i0}((0,\delta]),\alpha_{i1}([1-\delta,1))\subset\Ss_i$. 

As $\Ss_i$ is a convex set and $\alpha_{i0}(\delta),\alpha_{i1}(1-\delta)\in\Ss_i$, the segment that connects both points is contained in $\Ss_i$. Let 
$$
\alpha_{i2}:[\delta,1-\delta]\to\Ss_i,\ t\mapsto\frac{(1-\delta)-t}{1-2\delta}\alpha_{i0}(\delta)+\frac{t-\delta}{1-2\delta}\alpha_{i1}(1-\delta)
$$ 
be a parameterization of such segment. Define the continuous semialgebraic path 
$$
\alpha_i:=\alpha_{i0}|_{[-\delta,\delta]}*\alpha_{i2}*\alpha_{i1}|_{[1-\delta,1+\delta]}:[-\delta,1+\delta]\to\pol,
$$ 
which satisfies $\alpha_i([-\delta,0)\cup(1,1+\delta])\subset\Int(\widehat{\sigma})$, $\alpha_i((0,1))\subset\Ss_i$ and in fact all the required conditions (i)-(ix).

\vspace{2mm}
\noindent{\sc Step 1.} {\em We have the following inclusions: $\widehat{\sigma}\subset F(\Delta_{n-1}\times[0,1])\subset\pol$ and $F(\Delta_{n-1}\times([-\delta,0)\cup(1,1+\delta]))\subset\Int(\widehat{\sigma})\subset\pol$}. 

Observe that $F(\Delta_{n-1}\times(0,1))\subset\pol$, because $\pol$ is convex and $\alpha_i((0,1))\subset\pol$. In addition, $\alpha_i(1)=p$ for each $i$, so $F(\Delta_{n-1}\times\{1\})=p$, and $F(\lambda_1,\ldots,\lambda_n,0)=\sum_{i=1}^n\lambda_iv_i$, so $F(\Delta_{n-1}\times\{0\})=\sigma\subset\pol$ and $F|_{\Delta_{n-1}\times\{0\}}$ is a homeomorphism. Thus, $F(\Delta_{n-1}\times[0,1])\subset\pol$. 

Let us analyze the restriction map $F|_{\partial\Delta_{n-1}\times(0,1)}:\partial\Delta_{n-1}\times(0,1)\to\pol$. Recall that $\partial\Delta_{n-1}=\bigcup_{i=1}^n(\Delta_{n-1}\cap\{\lambda_i=0\})$. 

Fix an index $i=1,\ldots,n$ and write $\lambda^{(i)}:=(\lambda_1,\ldots,\lambda_{i-1},0,\lambda_{i-1},\ldots,\lambda_n)$ where $\sum_{j\neq i}\lambda_j=1$ and each $\lambda_j\geq0$ if $j\neq i$. We have $\Int(H_i^-)\subset\R^n\setminus\widehat{\sigma}$ and 
\begin{equation}\label{fuori}
F(\lambda^{(i)},t)=\sum_{j\neq i}\lambda_j\alpha_j(t)\in\Int(\pol\cap H_i^-)=\Int(\pol)\cap\Int(H_i^-)\subset\pol\setminus\widehat{\sigma}
\end{equation}
for $t\in(0,1)$, because if $j\neq i$ each $\alpha_j(t)\in\Int(\pol)\cap\Int(H_i^-)$ and the latter is convex. Thus, $F(\partial\Delta_{n-1}\times(0,1))\subset\pol\setminus\widehat{\sigma}$. By Lemma \ref{simplex} $\widehat{\sigma}\subset F(\Delta_{n-1}\times[0,1])$. 

As $\alpha_i([-\delta,0)\cup(1,1+\delta])\subset\Int(\widehat{\sigma})$ and $\Int(\widehat{\sigma})$ is convex, one concludes that $F(\Delta_{n-1}\times([-\delta,0)\cup(1,1+\delta]))\subset\Int(\widehat{\sigma})$. 

\vspace{2mm}
Let us construct in the following steps $\veps>0$ such that: {\em if $G$ is under the hypothesis of the statement, then $\widehat{\sigma}\subset G(\Delta_{n-1}\times[0,1])\subset\pol$ and $G(\Delta_{n-1}\times([-\delta',0]\cup[1,1+\delta']))\subset\widehat{\sigma}$ for some $0<\delta'<\delta$ small enough}.

\vspace{2mm}
\noindent{\sc Step 2.} {\em Choice of $\veps>0$.} By \eqref{fuori} $F(\lambda,t)\in\pol\setminus\widehat{\sigma}\subset\R^n\setminus\widehat{\sigma}$ for each $(\lambda,t)\in\partial\Delta_{n-1}\times(0,1)$. For each $0<\rho<\frac{1}{2}$ define
$$
\veps_\rho:=\tfrac{1}{2}\min\Big\{\dist(F(\lambda,t),\widehat{\sigma}):\ (\lambda,t)\in\partial\Delta_{n-1}\times[\rho,1-\rho]\Big\}>0.
$$
If $G:\Delta_{n-1}\times[\rho,1-\rho]\to\R^n$ satisfies $\|F-G\|<\veps_\rho$, then $G((\partial\Delta_{n-1}\times[\rho,1-\rho])\subset\R^n\setminus\widehat{\sigma}$. 

See assertions (i)-(ix) above for the definition of $a_j,b_{i0}, c_{ik}, d_{ik}$ and $e_{ik}$. Consider
$$
c_{ik}^*:=\begin{cases}
c_{ik}&\text{if $c_{ik}>0$,}\\
d_{ik}&\text{if $c_{ik}=0$}
\end{cases}
$$
and define
\begin{equation}\label{vepsd}
\epsilon_0:=\tfrac{1}{2}\min\{a_j,b_{i0},c_{ik}^*,e_{ik}:\ \forall\,i,j,k\}>0.
\end{equation}
By hypothesis (v) and (vi): 
\begin{align*}
&h_i\circ\alpha_i(\t)=h_i(v_i)+\vec{h_i}(u_i)\t^2-a_i\t^3+\cdots,\\
&h_j\circ\alpha_i(\t)=-a_j\t^3+\cdots \text{ if } i\neq j,\\
&h_j\circ\alpha_i(1+\t)=a_j\t^3+\cdots.
\end{align*} 
As $a_j>0$ for each $j=1,\ldots,n$, there exists $0<\rho_0<\delta$ such that 
\begin{equation}\label{rhod}
-(h_j\circ\alpha_i)'''|_{[-\rho_0,\rho_0]}\geq\epsilon_0\quad\text{and}\quad (h_j\circ\alpha_i)'''|_{[1-\rho_0,1+\rho_0]}\geq\epsilon_0.
\end{equation}
As $h_i$ is a polynomial of degree $1$ for $i=1,\ldots,n$ and $\sum_{j=1}^n\lambda_j=1$,
\begin{align*}
\frac{\partial^\ell}{\partial\t^\ell}((h_i\circ F)(\lambda,\t))&=\sum_{j=1}^n\lambda_j(h_i\circ\alpha_j)^{(\ell)}(\t)
\end{align*}
for each $\ell\geq 0$. Consequently,
\begin{align}
&-\frac{\partial^3}{\partial\t^3}(h_i\circ F)|_{\Delta_{n-1}\times[-\rho_0,\rho_0]}\geq\epsilon_0,\label{form1}\\
&\frac{\partial^3}{\partial\t^3}(h_i\circ F)|_{\Delta_{n-1}\times[1-\rho_0,1+\rho_0]}\geq\epsilon_0\label{form2}
\end{align}
for $i=1,\ldots,n$.

For each $k=1,\ldots,s$ define ${\mathfrak F}_k:=\{i=1,\ldots,n:\ c_{ik}\neq0\}$. As $g_k$ is a polynomial of degree $1$ and $\sum_{j=1}^n\lambda_j=1$,
$$
(g_k\circ F)(\lambda,\t)=\sum_{i=1}^n\lambda_i(g_k\circ\alpha_i)(t)=\sum_{i\in{\mathfrak F}_k}\lambda_ic_{ik}+\sum_{i\in{\mathfrak F}_k}\lambda_id_{ik}\t^2+\sum_{i\not\in{\mathfrak F}_k}\lambda_id_{ik}\t^2+\cdots
$$
Define $\Gamma_k:=\{\lambda\in\Delta_{n-1}:\ \lambda_i=0,\ i\in{\mathfrak F}_k\}$. If $\mathfrak{F}_k\neq\{1,\ldots,n\}$, then $\Gamma_k\neq\varnothing$ and $\mu_{0k}:=\min\{d_{ik}\, :\, i\not\in\mathfrak{F}_k\}>0$. Otherwise, $\mathfrak{F}_k=\{1,\ldots,n\}$, so $\Gamma_k=\varnothing$, and define $\mu_{0k}:=1$. If $\mathfrak{F}_k\neq\{1,\ldots,n\}$ and $\lambda:=(\lambda_1,\ldots,\lambda_n)\in\Gamma_k$, then $\sum_{i\not\in{\mathfrak F}_k}\lambda_i=1$ and each $\lambda_i\geq0$, so $\sum_{i\not\in{\mathfrak F}_k}\lambda_id_{ik}\geq\mu_{0k}$. Define
$$
V_k:=\begin{cases}
\{\lambda\in\Delta_{n-1}:\ \sum_{i\in{\mathfrak F}_k}\lambda_i<\frac{1}{4},\ |\sum_{i\in{\mathfrak F}_k}\lambda_id_{ik}|<\frac{1}{4}\mu_{0k}\}&\text{if $\mathfrak{F}_k\neq\varnothing,\{1,\ldots,n\}$},\\
\Delta_{n-1}&\text{if $\mathfrak{F}_k=\varnothing$},\\
\varnothing&\text{if $\mathfrak{F}_k=\{1,\ldots,n\}$}.
\end{cases}
$$
If $\mathfrak{F}_k\neq\varnothing,\{1,\ldots,n\}$, then $V_k\neq\varnothing$ and if $\lambda\in V_k$, we have $\sum_{i\not\in{\mathfrak F}_k}\lambda_i>\frac{3}{4}$, so $\sum_{i\not\in{\mathfrak F}_k}\lambda_id_{ik}>\frac{3}{4}\mu_{0k}$ and
\begin{equation}\label{form4}
\frac{1}{2}\frac{\partial^2(g_k\circ F)}{\partial\t^2}(\lambda,0)=\sum_{i=1}^n\lambda_id_{ik}=\sum_{i\in{\mathfrak F}_k}\lambda_id_{ik}+\sum_{i\not\in{\mathfrak F}_k}\lambda_id_{ik}>-\frac{\mu_{0k}}{4}+\frac{3\mu_{0k}}{4}=\frac{1}{2}\mu_{0k}.
\end{equation}
If $\mathfrak{F}_k=\varnothing$, then $\sum_{i=1}^n\lambda_i=1$ and
\begin{equation}\label{nuova}
\frac{1}{2}\frac{\partial^2(g_k\circ F)}{\partial\t^2}(\lambda,0)=\sum_{i=1}^n\lambda_id_{ik}\geq \mu_{0k}>\frac{1}{2}\mu_{0k}
\end{equation}
for each $\lambda\in\Delta_{n-1}=V_k$. Consequently, if $\mathfrak{F}_k\neq\{1,\ldots,n\}$, we have
\begin{equation*}
\frac{\partial^2(g_k\circ F)}{\partial\t^2}(\lambda,0)>\mu_{0k}\quad\text{for each $\lambda\in V_k$.}
\end{equation*}

Define in addition
\begin{equation}\label{mu1}
\mu_{1k}:=\begin{cases}
\min\{(g_k\circ F)(\lambda,0)=\sum_{i\in{\mathfrak F}_k}\lambda_ic_{ik}:\ \lambda\in\Delta_{n-1}\setminus V_k\}>0&\text{if $\mathfrak{F}_k\neq\varnothing$},\\
1&\text{if $\mathfrak{F}_k=\varnothing$}.
\end{cases}
\end{equation}

Recall that $\rho_0>0$ was chosen in \eqref{rhod}. Let $0<\rho<\rho_0$ be such that 
\begin{equation}\label{vk}
\begin{cases}
|g_k(F(\lambda,t))-g_k(F(\lambda,0))|<\frac{\mu_{1k}}{2}&\text{if $(\lambda,t)\in(\Delta_{n-1}\setminus V_k)\times[-\rho,\rho]$,}\\
\frac{\partial^2(g_k\circ F)}{\partial\t^2}(\lambda,t)>\frac{\mu_{0k}}{2}&\text{if $(\lambda,t)\in\cl(V_k)\times[-\rho,\rho]$.}
\end{cases}
\end{equation}
Observe that $F(\Delta_{n-1}\times[\rho,1+\rho])\subset\Int(\pol)$ because $\Int(\pol)$ is convex and $\alpha_i([\rho,1+\rho])\subset\Int(\pol)$ for $i=1,\ldots,n$, so 
\begin{equation}\label{form5}
\veps'_\rho:=\tfrac{1}{2}\min\{\dist(F(\lambda,t),\R^n\setminus\Int(\pol)):\ (\lambda,t)\in\Delta_{n-1}\times[\rho,1+\rho]\}>0.
\end{equation}
Thus, if $G:\Delta_{n-1}\times [\rho,1-\rho]\to \R^n$ and $\|F-G\|<\veps_\rho'$, then $
G((\Delta_{n-1}\times[\rho,1+\rho])\subset\Int(\pol)$. In fact, if $\|F-G\|<\min\{\veps_\rho,\veps_\rho'\}$, then 
$$
G((\Delta_{n-1}\times[\rho,1+\rho])\subset\Int(\pol)\setminus\widehat{\sigma}.
$$

Denote $\epsilon_0':=\min\{\epsilon_0,\mu_{1k},\frac{\mu_{0k}}{2}:\, k=1,\ldots,s\}$ (recall that $\epsilon_0$ was defined in \eqref{vepsd}) and $I_\rho:=[-\rho,\rho]\cup[1-\rho,1+\rho]\subset I$. The maps
\begin{align*}
\Psi_k:\Cont^3_{I_\rho}(\Delta_{n-1}\times[-\rho,1+\rho],\R^n)\to\Cont^3_{I_\rho}(\Delta_{n-1}\times[-\rho,1+\rho],\R),&\ H\mapsto g_k\circ H,\\
\Theta_i:\Cont^3_{I_\rho}(\Delta_{n-1}\times[-\rho,1+\rho],\R^n)\to\Cont^3_{I_\rho}(\Delta_{n-1}\times[-\rho,1+\rho],\R),&\ H\mapsto h_i\circ H
\end{align*}
are continuous with respect to the $\mathcal{C}^3_{I_\rho}$-topology of the involved spaces (see \S\ref{ckt}). Let $0<\veps<\min\{\veps_\rho,\veps_\rho'\}$ be such that if $\|F-G\|<\veps$ and $\|\frac{\partial^\ell}{\partial\t^\ell}F-\frac{\partial^\ell}{\partial\t^\ell}G\|_{\Delta_{n-1}\times I_\rho}<\veps$ for $\ell=1,2,3$, then $|g_k\circ F-g_k\circ F|<\frac{\epsilon_0'}{2}$, $|h_i\circ F-h_i\circ F|<\frac{\epsilon_0'}{2}$ and
\begin{align}
\Big|\frac{\partial^\ell}{\partial\t^\ell}(g_k\circ F)-\frac{\partial^\ell}{\partial\t^\ell}(g_k\circ G)\Big|_{\Delta_{n-1}\times I_\rho}&=\Big|\frac{\partial^\ell}{\partial\t^\ell}(\Psi_k(F))-\frac{\partial^\ell}{\partial\t^\ell}(\Psi_k(G))\Big|_{\Delta_{n-1}\times I_\rho}<\frac{\epsilon_0'}{2},\label{form7}\\
\Big|\frac{\partial^\ell}{\partial\t^\ell}(h_i\circ F)-\frac{\partial^\ell}{\partial\t^\ell}(h_i\circ G)\Big|_{\Delta_{n-1}\times I_\rho}&=\Big|\frac{\partial^\ell}{\partial\t^\ell}(\Theta_i(F))-\frac{\partial^\ell}{\partial\t^\ell}(\Theta_i(G))\Big|_{\Delta_{n-1}\times I_\rho}<\frac{\epsilon_0'}{2}\label{form6}
\end{align}
for $\ell=1,2,3$ and $G\in \Cont^3_{I_\rho}(\Delta_{n-1}\times[-\rho,1+\rho],\R^n)$. The chosen value $\veps>0$ depends only on $\pol$, $\widehat{\sigma}$, $F$ and $\rho>0$.

\vspace{2mm}
Let us check in the following steps: \em $\veps>0$ satisfies the conditions in the statement\em. Let $G:\Delta_{n-1}\times[-\rho,1+\rho]\to\R^n$ be a continuous semialgebraic map satisfying the conditions in the statement. We have to prove: {\em $\widehat{\sigma}\subset G(\Delta_{n-1}\times[0,1])\subset\pol$ and $G(\Delta_{n-1}\times([-\delta',0]\cup[1,1+\delta']))\subset\widehat{\sigma}$ for some $0<\delta'<\delta$.}

\vspace{2mm}
\noindent{\sc Step 3.} We prove next: $G(\Delta_{n-1}\times([-\rho,1+\rho]\setminus\{0\}))\subset\Int(\pol)$. 

Since $\|F-G\|<\veps\leq\veps_\rho'$ (see \eqref{form5} for the definition of $\veps'_{\rho}$), we have 
$$
G(\Delta_{n-1}\times[\rho,1+\rho])\subset\Int(\pol).
$$
Fix $k=1,\ldots,s$ and let $\lambda\in\Delta_{n-1}$. Let us check: \em $G(\lambda,t)\in\Int(\pol)=\{g_1>0,\ldots,g_s>0\}$ for each $t\in[-\rho,\rho]\setminus\{0\}$\em. 

We distinguish two cases:

\noindent{\sc Case 1}. $\lambda\in\Delta_{n-1}\setminus V_k$. Observe that if $\mathfrak{F}_k=\varnothing$, then $\Delta_{n-1}\setminus V_k=\varnothing$. By \eqref{vk} and \eqref{form7}
$$
|g_k(F(\lambda,t))-g_k(F(\lambda,0))|<\frac{\mu_{1k}}{2} \text{ \, and \, } |g_k\circ G-g_k\circ F|<\frac{\mu_{1k}}{2},
$$
if $t\in[-\rho,\rho]$. By \eqref{mu1} we deduce
\begin{multline*}
(g_k\circ G)(\lambda,t)=(g_k\circ F)(\lambda,0)+(g_k\circ G)(\lambda,t)-(g_k\circ F)(\lambda,t)\\
+(g_k\circ F)(\lambda,t)-(g_k\circ F)(\lambda,0)
>\mu_{1k}-\frac{\mu_{1k}}{2}-\frac{\mu_{1k}}{2}=0
\end{multline*}
for each $t\in[-\rho,\rho]$. Thus, $g_k(G(\lambda,t))>0$ for $t\in[-\rho,\rho]$ and $k=1,\ldots,s$, that is, $G(\lambda,t)\in\Int(\pol)$ for $t\in[-\rho,\rho]$.

\vspace{1mm}
\noindent{\sc Case 2}. $V_k\neq\varnothing$ and $\lambda\in V_k$. Then $0\leq\sum_{i\in{\mathfrak F}_k}\lambda_ic_{ik}$ and $\sum_{i=1}^n\lambda_id_{ik}\geq\frac{1}{2}\mu_{0k}>0$ (see \eqref{form4} and \eqref{nuova}). As 
$$
g_k(F(\lambda,\t))=\sum_{i=1}^n\lambda_i(g_k\circ\alpha_i)(\t)=\sum_{i\in{\mathfrak F}_k}\lambda_ic_{ik}+\sum_{i=1}^n\lambda_i(d_{ik}\t^2+\cdots)
$$
and $F(\lambda,\t)-G(\lambda,\t)\in(\t)^4\R[[\t]]$, we have
$$
g_k(G(\lambda,\t))=\sum_{i\in{\mathfrak F}_k}\lambda_ic_{ik}+\sum_{i=1}^n\lambda_i(d_{ik}\t^2+\cdots).
$$

Define $\theta_k:=g_k(G(\lambda,\cdot))-\sum_{i\in{\mathfrak F}_k}\lambda_ic_{ik}$. Suppose there exists $t_0\in[-\rho,\rho]\setminus\{0\}$ such that $\theta_k(t_0)\leq0$. As $\theta_k(t)>0$ for $t$ close to $0$, we may assume $\theta_k(t_0)=0$. By Rolle's theorem there exists $t_1\in(0,t_0)$ (or $t_1\in(t_0,0)$) such that $\theta_k'(t_1)=0$. As $\theta_k'(0)=0$, there exists $t_2\in(0,t_1)$ (or $t_2\in(t_1,0)$) satisfying $\theta_k''(t_2)=0$. We have by \eqref{vk} and \eqref{form7}
\begin{multline*}
\frac{\mu_{0k}}{2}\leq\Big|\frac{\partial^2}{\partial\t^2}(g_k\circ F)(\lambda,t_2)\Big|=\Big|\frac{\partial^2}{\partial\t^2}(g_k\circ F)(\lambda,t_2)-\theta_k''(t_2)\Big|\\
=\Big|\frac{\partial^2}{\partial\t^2}(g_k\circ F)(\lambda,t_2)-\frac{\partial^2}{\partial\t^2}(g_k\circ G)(\lambda,t_2)\Big|<\frac{\epsilon_0'}{2}\leq\frac{\mu_{0k}}{4},
\end{multline*}
which is a contradiction. Consequently, $\theta_k(t)>0$ for each $t\in[-\rho,\rho]\setminus\{0\}$ and $k=1,\ldots,s$. Thus, $G(\lambda,t)\in \Int(\pol)$ for $t\in [-\rho,\rho]\setminus\{0\}$.

\vspace{2mm}
\noindent{\sc Step 4.} We prove $\widehat{\sigma}\subset G(\Delta_{n-1}\times[0,1])$ as an application of Lemma \ref{simplex}. Observe that $G|_{\Delta_{n-1}\times\{0\}}=F|_{\Delta_{n-1}\times\{0\}}$ is a homeomorphism and $G|_{\Delta_{n-1}\times\{1\}}=F|_{\Delta_{n-1}\times\{1\}}=p$. Let us show: $G(\partial\Delta_{n-1}\times(0,1))\subset\Int(\pol)\setminus\widehat{\sigma}\subset\{h_0\geq0\}\setminus\widehat{\sigma}$. 

As $\|F-G\|<\veps\leq\veps_\rho$, we have $G(\partial\Delta_{n-1}\times[\rho,1-\rho])\subset\R^n\setminus\widehat{\sigma}$. We fix $\lambda^{(i)}\in\Delta_{n-1}\cap\{\lambda_i=0\}$ and claim: \em $G(\lambda^{(i)},t)\in\Int(H_i^-)=\{h_i<0\}$ for each $t\in(0,\rho]\cup[1-\rho,1)$\em.

Denote $\varphi_i:=h_i(G(\lambda^{(i)},\cdot))$. Suppose there exists $t_0\in(0,\rho]$ such that $\varphi_i(t_0)\geq0$. As 
$$
h_i(F(\lambda^{(i)},\t))=\sum_{j\neq i}\lambda_j(h_i\circ\alpha_j)(\t)=\sum_{j\neq i}\lambda_j(-a_i\t^3+\cdots)
$$
and $F(\lambda^{(i)},\t)-G(\lambda^{(i)},\t)\in(\t)^4\R[[\t]]$, we have
$$
\varphi_i(\t)=h_i(G(\lambda^{(i)},\t))=\sum_{j\neq i}\lambda_j(-a_i\t^3+\cdots).
$$
Thus, $\varphi_i(t)<0$ for $t>0$ close to $0$, so we may assume $\varphi_i(t_0)=0$. Consequently, as $\varphi_i(0)=0$, there exists by Rolle's theorem $t_1\in(0,t_0)$ such that $\varphi_i'(t_1)=0$. As $\varphi_i'(0)=0$, there exists $t_2\in(0,t_1)$ satisfying $\varphi_i''(t_2)=0$. As $\varphi_i''(0)=0$, there exists $t_3\in(0,t_2)$ such that $\varphi_i'''(t_3)=0$. We have by \eqref{form1} and \eqref{form6}
\begin{multline*}
\epsilon_0\leq\Big|\frac{\partial^3}{\partial\t^3}(h_i\circ F)(\lambda^{(i)},t_3)\Big|=\Big|\frac{\partial^3}{\partial\t^3}(h_i\circ F)(\lambda^{(i)},t_3)-\varphi_i'''(t_3)\Big|\\
=\Big|\frac{\partial^3}{\partial\t^3}(h_i\circ F)(\lambda^{(i)},t_3)-\frac{\partial^3}{\partial\t^3}(h_i\circ G)(\lambda^{(i)},t_3)\Big|\leq\frac{\epsilon_0}{2},
\end{multline*}
which is a contradiction. Consequently, $\varphi_i(t)<0$ for each $t\in(0,\rho]$.

Analogously, one shows $\varphi_i(t)<0$ for each $t\in[1-\rho,1)$ and $i=1,\ldots,n$. Thus, 
$$
G(\partial\Delta_{n-1}\times(0,1))=\bigcup_{i=1}^nG((\Delta_{n-1}\cap\{\lambda_i=0\})\times(0,1))\subset\bigcup_{i=1}^n\{h_i<0\}=\R^n\setminus\widehat{\sigma}.
$$
In addition, by {\sc Step 3} we know $G(\Delta_{n-1}\times(0,1))\subset\Int(\pol)$, so 
$$
G(\partial\Delta_{n-1}\times(0,1))\subset\Int(\pol)\setminus\widehat{\sigma}\subset\{h_0\geq0\}\setminus\widehat{\sigma}.
$$ 
By Lemma \ref{simplex} we have $\widehat{\sigma}\subset G(\Delta_{n-1}\times[0,1])$.

\vspace{2mm}
\noindent{\sc Step 5.} Observe that 
\begin{align*}
&G(\Delta_{n-1}\times\{0\})=F(\Delta_{n-1}\times\{0\})=\sigma\subset\widehat{\sigma},\\
&G(\Delta_{n-1}\times\{1\})=F(\Delta_{n-1}\times\{1\})=\{p\}\subset\widehat{\sigma}.
\end{align*}
Finally, we show: \em $G(\Delta_{n-1}\times([-\delta',0)\cup(1,1+\delta']))\subset\widehat{\sigma}$ for some $0<\delta'\leq \rho<\delta$. \em 

For each $\lambda:=(\lambda_1,\ldots,\lambda_n)\in\Delta_{n-1}$ consider
$$
h_i(F(\lambda,\t))=\sum_{j=1}^n\lambda_j(h_i\circ\alpha_j)(\t)=(h_i\circ\alpha_i)(\t)+\sum_{j\neq i}\lambda_j(-a_{j}\t^3+\cdots).
$$
As $F(\lambda,\t)-G(\lambda,\t)\in(\t)^4\R[[\t]]$, we have
\begin{align*}
h_i(G(\lambda,\t))&=\lambda_i(h_i(v_i)+\vec{h}_i(u_i)\t^2-a_i\t^3+\cdots)+\sum_{j\neq i}\lambda_j(-a_{j}\t^3+\cdots)\\
&=\lambda_i(h_i(v_i)+\vec{h}_i(u_i)\t^2)-\sum_{j=1}^n \lambda_ja_{j}\t^3+\cdots.
\end{align*}
Pick $\lambda\in \Delta_{n-1}$ and define 
$$
\psi_i(\t):=h_i(G(\lambda,\t))-\lambda_i(h_i(v_i)+\vec{h}_i(u_i)\t^2)=-\sum_{j=1}^n \lambda_ja_{j}\t^3+\cdots.
$$
As $a_j>0$, $\lambda_j\geq0$ and $\sum_{j=1}^n\lambda_j=1$, we have $\psi_i(t)>0$ for $t<0$ close enough to $0$. Suppose that there exists $t_0\in [-\rho,0)$ such that $\psi_i(t_0)\leq 0$. As $\psi_i(0)=0$ and $\psi_i(t)>0$ for $t<0$ close enough to 0, we may assume $\psi_i(t_0)=0$. As $\psi_i(0)=0$, by Rolle's theorem there exists $t_1\in(t_0,0)$ such that $\psi_i'(t_1)=0$. As $\psi_i'(0)=0$, there exists $t_2\in(t_1,0)$ satisfying $\psi_i''(t_2)=0$. As $\psi_i''(0)=0$, there exists $t_3\in(t_2,0)$ such that $\psi_i'''(t_3)=0$. 
Recall that by \eqref{form1}
$$
\frac{\partial^3}{\partial\t^3}(h_i\circ F)(\lambda,\t)=-6\sum_{j=1}^n\lambda_ja_{j}+\cdots\leq-\epsilon_0.
$$
Thus, we have by \eqref{form6}
\begin{multline*}
\epsilon_0\leq\Big|\frac{\partial^3}{\partial\t^3}(h_i\circ F)(\lambda,t_3)\Big|=\Big|\frac{\partial^3}{\partial\t^3}(h_i\circ F)(\lambda,t_3)-\psi_i'''(t_3)\Big|\\
=\Big|\frac{\partial^3}{\partial\t^3}(h_i\circ F)(\lambda,t_3)-\frac{\partial^3}{\partial\t^3}(h_i\circ G)(\lambda,t_3)\Big|\leq\frac{\epsilon_0}{2},
\end{multline*}
which is a contradiction. Consequently, $\psi_i(t)>0$ for each $t\in[-\rho,0)$. As $h_i(v_i)>0$, there exists $0<\delta'\leq\rho<\delta$ such that $h_i(v_i)+\t^2\vec{h}_i(u_i)>0$ on $[-\delta',0)$ for $i=1,\ldots,n$. Thus, $h_i(G(\lambda,t))>0$ on $\Delta_{n-1}\times[-\delta',0)$ for $i=1,\ldots,n$. 

Let us show $h_i(G(\lambda,t))>0$ for $(\lambda,t)\in (1,1+\rho]$ and $i=1,\ldots,n$. We have
$$
h_i(F(\lambda,1+\t))=\sum_{j=1}^n\lambda_j(h_i\circ\alpha_j)=\sum_{j=1}^n\lambda_j a_j\t^3+\cdots
$$
and we can repeat the previous argument taking $\phi_i(\t):=h_i(G(\lambda,1+\t)$. As $F(\lambda,1+\t)-G(\lambda,1+\t)\in(\t)^4\R[[\t]]$,
$$
\phi_i(\t)=h_i(G(\lambda,1+\t)=\sum_{j=1}^n\lambda_j(h_i\circ\alpha_j)=\sum_{j=1}^n\lambda_j a_j\t^3+\cdots.
$$
As $a_j>0$, $\lambda_j\geq0$ and $\sum_{j=1}^n\lambda_j=1$, we have $\phi_i(t)>0$ for $t>0$ close enough to $0$. Suppose that there exists $t_0\in (0,\rho]$ such that $\phi_i(t_0)\leq 0$. As $\phi_i(0)=0$ and $\phi_i(t)>0$ for $t>0$ close enough to 0, we may assume $\phi_i(t_0)=0$. As $\phi_i(0)=0$, there exists by Rolle's theorem $t_1\in(0,t_0)$ such that $\phi_i'(t_1)=0$. As $\phi_i'(0)=0$, there exists $t_2\in(0,t_1)$ satisfying $\phi_i''(t_2)=0$. As $\phi_i''(0)=0$, there exists $t_3\in(0,t_2)$ such that $\phi_i'''(t_3)=0$. We have by \eqref{form2} and \eqref{form6}
\begin{multline*}
\epsilon_0\leq\Big|\frac{\partial^3}{\partial\t^3}(h_i\circ F)(\lambda,1+t_3)\Big|=\Big|\frac{\partial^3}{\partial\t^3}(h_i\circ F)(\lambda,1+t_3)-\psi_i'''(t_3)\Big|\\
=\Big|\frac{\partial^3}{\partial\t^3}(h_i\circ F)(\lambda,1+t_3)-\frac{\partial^3}{\partial\t^3}(h_i\circ G)(\lambda,1+t_3)\Big|\leq\frac{\epsilon_0}{2},
\end{multline*}
which is a contradiction. Thus $\phi_i(t)>0$ for each $t\in (0,\rho]$, so $h_i(G(\lambda,t))>0$ on $\Delta_{n-1}\times (1,1+\rho]$.

We conclude $G(\Delta_{n-1}\times([-\delta',0)\cup(1,1+\delta']))\subset\widehat{\sigma}$, as required. 
\end{proof}

We will use the technical Lemma \ref{simplex2} to `cover' simplices with Nash maps. For technical reasons, we will first approximate the continuous semialgebraic paths $\alpha_i$ by Nash paths. Let us check that for (close enough) approximations we obtain the desired result.

\begin{remark}\label{simplex2r}
For each $i=1,\ldots,n$ let $\alpha_i^*:[-\delta,1+\delta]\to\pol$ be a continuous semialgebraic path such that $\alpha_i^*|_I$ is a Nash map, $\alpha_i^*$ is close to $\alpha_i$, $(\alpha_i^*|_I)^{(\ell)}$ is close to $(\alpha_i|_I)^{(\ell)}$ for $\ell=1,2,3$, $(\alpha_i^*)^{(\ell)}(0)=(\alpha_i)^{(\ell)}(0)$ and $(\alpha_i^*)^{(\ell)}(1)=(\alpha_i)^{(\ell)}(1)$ for $\ell=0,1,2,3$ (recall that $I:=[-\delta,\delta]\cup[1-\delta,1+\delta]$).

(i) Then there exists $\veps^*>0$ and
$$
F^*:\Delta_{n-1}\times[-\delta,1+\delta]\to\pol,\ (\lambda,t)\mapsto\sum_{i=1}^n\lambda_i\alpha_i^*(t)
$$
that satisfy the same conditions as $\veps$ and $F$ in the statement of Theorem \ref{simplex2}.

Observe that
\begin{align*}
(F-F^*)(\lambda,t)&=\sum_{i=1}^n\lambda_i(\alpha_i(t)-\alpha_i^*(t)),\\
\Big(\frac{\partial^{\ell}F}{\partial\t^\ell}-\frac{\partial^{\ell}F^*}{\partial\t^\ell}\Big)(\lambda,t)&=\sum_{i=1}^n\lambda_i(\alpha_i^{(\ell)}(t)-(\alpha_i^*)^{(\ell)}(t))
\end{align*}
for $\ell=1,2,3$. In addition, for $\ell=0,1,2,3$,
\begin{align*}
\frac{\partial^{\ell}F}{\partial\t^\ell}(\lambda,0)&=\sum_{i=1}^n\lambda_i\alpha_i^{(\ell)}(0)=\sum_{i=1}^n\lambda_i(\alpha_i^*)^{(\ell)}(0)=\frac{\partial^{\ell}F^*}{\partial\t^\ell}(\lambda,0),\\
\frac{\partial^{\ell}F}{\partial\t^\ell}(\lambda,1)&=\sum_{i=1}^n\lambda_i\alpha_i^{(\ell)}(1)=\sum_{i=1}^n\lambda_i(\alpha_i^*)^{(\ell)}(1)=\frac{\partial^{\ell}F^*}{\partial\t^\ell}(\lambda,1).
\end{align*}
Take $\veps^*:=\frac{\veps}{2}>0$ and assume that $\|\alpha_i-\alpha_i^*\|<\veps^*$ and $\|\alpha_i^{(\ell)}-(\alpha_i^*)^{(\ell)}\|_I<\veps^*$ for $\ell=1,2,3$. Let $G:\Delta_{n-1}\times[-\delta,1+\delta]\to\R^n$ be a continuous semialgebraic map that is Nash on a neighborhood $\Delta_{n-1}\times I'\subset\Delta_{n-1}\times I$ of $\Delta_{n-1}\times\{0,1\}$ and satisfies $\frac{\partial^\ell G}{\partial\t^\ell}(\lambda,0)=\frac{\partial^\ell F^*}{\partial\t^\ell}(\lambda,0)$, $\frac{\partial^\ell G}{\partial\t^\ell}(\lambda,1)=\frac{\partial^\ell F^*}{\partial\t^\ell}(\lambda,1)$ for each $\lambda\in\Delta_{n-1}$ and $\ell=0,1,2,3$, $\|G-F^*\|<\veps^*$ and $\|\frac{\partial^\ell G}{\partial\t^\ell}-\frac{\partial^\ell F^*}{\partial\t^\ell}\|_{\Delta_{n-1}\times I'}<\veps^*$ for $\ell=1,2,3$. Then
\begin{align*}
\frac{\partial^\ell G}{\partial\t^\ell}(\lambda,0)&=\frac{\partial^\ell F^*}{\partial\t^\ell}(\lambda,0)=\frac{\partial^\ell F}{\partial\t^\ell}(\lambda,0),\\
\frac{\partial^\ell G}{\partial\t^\ell}(\lambda,1)&=\frac{\partial^\ell F^*}{\partial\t^\ell}(\lambda,1)=\frac{\partial^\ell F}{\partial\t^\ell}(\lambda,1)
\end{align*}
for each $\lambda\in\Delta_{n-1}$ and $\ell=0,1,2,3$, and
\begin{align*}
\|G-F\|&\leq\|G-F^*\|+\|F^*-F\|<\veps^*+\veps^*=\veps,\\
\Big\|\frac{\partial^\ell G}{\partial\t^\ell}-\frac{\partial^\ell F}{\partial\t^\ell}\Big\|_{\Delta_{n-1}\times I'}&\leq\Big\|\frac{\partial^\ell G}{\partial\t^\ell}-\frac{\partial^\ell F^*}{\partial\t^\ell}\Big\|_{\Delta_{n-1}\times I'}+\Big\|\frac{\partial^\ell F^*}{\partial\t^\ell}-\frac{\partial^\ell F}{\partial\t^\ell}\Big\|_{\Delta_{n-1}\times I'}<\veps^*+\veps^*=\veps
\end{align*}
for $\ell=1,2,3$. By Theorem \ref{simplex2} we have that $\widehat{\sigma}\subset G(\Delta_{n-1}\times[0,1])\subset\pol$ and $G(\Delta_{n-1}\times([-\rho,0]\cup[1,1+\rho]))\subset\widehat{\sigma}$ for some $0<\rho<\delta$ small enough, as required.

(ii) By (i) and Lemma \ref{icsl} we may assume that each path $\alpha_i:[-\delta,1+\delta]\to\pol$ in the statement of Theorem \ref{simplex2} is Nash on $[-\delta,1+\delta]$.
\end{remark}

The following result provides sufficient conditions to guarantee that the high order derivatives of two continuous semialgebraic maps on $\R^d\times[-1,1]$ that are Nash on a neighborhood of a semialgebraic set $\Ss\times\{0\}$ are equal at the points of $\Ss\times\{0\}$. This provides a sufficient condition to decide when the approximating maps satisfy the hypothesis of Lemma \ref{simplex2}.

\begin{lem}\label{order}
Let $\Ss\subset\R^d$ be a non-empty semialgebraic set. Let $F,G:\R^d\times[-1,1]\to\R^m$ be two continuous semialgebraic maps that are Nash on a neighborhood of $\Ss\times\{0\}$ and suppose that there exists a Nash function $\lambda:[-1,1]\to\R$ such that $\|F-G\|_{\Ss\times[-1,1]}<|\lambda|_{[-1,1]}$ and that $\lambda(\t)=a_{k+1}\t^{k+1}u^2(\t)$ where $a_{k+1}\neq0$ and $u\in\R[[\t]]_{\rm alg}$ is a Nash series such that $u(0)=1$. Then, for each $x\in\Ss$ we have
$$
\frac{\partial^\ell F}{\partial\t^\ell}(x,0)=\frac{\partial^\ell G}{\partial\t^\ell}(x,0)
$$
for $\ell=0,\ldots,k$.
\end{lem}
\begin{proof}
Pick $x\in\Ss$ and write 
\begin{align*}
F(x,\t)&:=\sum_{\ell\geq0}\frac{1}{\ell!}\frac{\partial^\ell F}{\partial\t^\ell}(x,0)\t^\ell,\\
G(x,\t)&:=\sum_{\ell\geq0}\frac{1}{\ell!}\frac{\partial^\ell G}{\partial\t^\ell}(x,0)\t^\ell.
\end{align*}
Thus, we have the following inequalities in the ring $\R[[\t]]_{\rm alg}$ of algebraic power series with respect to any of its two orders (one characterized by $\t>0$ and the other one by $\t<0$): 
$$
\Big\|\sum_{\ell\geq0}\frac{1}{\ell!}\Big(\frac{\partial F}{\partial\t^\ell}(x,0)-\frac{\partial G}{\partial\t^\ell}(x,0)\Big)\t^\ell\Big\|\leq\|F(x,\t)-G(x,\t)\|\leq|a_{k+1}||\t^{k+1}|u^2.
$$
Consequently, the series
$$
\sum_{\ell\geq0}\frac{1}{\ell!}\Big(\frac{\partial^\ell F}{\partial\t^\ell}(x,0)-\frac{\partial^\ell G}{\partial\t^\ell}(x,0)\Big)\t^\ell
$$
is a series of order $\geq k+1$, so
$$
\frac{\partial^\ell F}{\partial\t^\ell}(x,0)-\frac{\partial^\ell G}{\partial\t^\ell}(x,0)=0
$$
for $\ell=0,\ldots,k$, as required.
\end{proof}

\subsection{Local charts and tubular neighborhoods.}\label{localcharts}
Let $\Tt\subset\R^n$ be a compact checkerboard set of dimension $d\geq 2$ and $\partial\Tt:=\Tt\setminus\Reg(\Tt)$. The algebraic set $M:=\ol{\Tt}^{\zar}\subset\R^n$ is a Nash manifold and $\ol{\partial\Tt}^{\zar}\subset M$ is a Nash normal-crossings divisor of $M$. By \cite[Thm.1.6]{fgr} $M$ can be covered by finitely many open semialgebraic subsets $U\subset M$ endowed with Nash diffeomorphisms $u:=(u_1,\ldots,u_d):U\to\R^d$ such that $U\cap\ol{\partial\Tt}^{\zar}=\{u_1\cdots u_k=0\}$ for some $k$ depending on $U$. Denote $\Lambda_k:=\{\x_1\geq0,\ldots,\x_k\geq0\}\subset\R^d$ for $k=1,\ldots,d$ and $\Lambda_0=\R^d$. As $\Tt$ is compact, there exist finitely many Nash diffeomorphisms $\phi_i:\R^d\to U_i\subset M$ (with inverse maps $u_i:=(u_{i1},\ldots,u_{id}):U_i\to\R^d$) for $i=1,\ldots,r$ such that: 
\begin{itemize}
\item If $U_i\cap\ol{\partial\Tt}^{\zar}\neq\varnothing$, then $U_i\cap\ol{\partial\Tt}^{\zar}=\{u_{i1}\cdots u_{ik_i}=0\}$.
\item $\Tt=\bigcup_{i=1}^r\phi_i(\Lambda_{k_i})$ for some $0\leq k_i\leq d$.
\item $\Reg(\Tt)=\bigcup_{i=1}^r\phi_i(\Int(\Lambda_{k_i}))$.
\end{itemize}

Let $(\Omega,\nu)$ be a Nash tubular neighborhood for the Nash manifold $M:=\ol{\Tt}^{\zar}$ endowed with a Nash retraction $\nu$ such that $\dist(z,M)=\|\nu(z)-z\|$ for each $z\in\Omega$ (see \cite[Cor.8.9.5]{bcr}). When $\Tt$ is compact, shrinking $\Omega$ if necessary, we may assume $\cl(\nu^{-1}(\Tt))$ is compact and $\nu$ admits a Nash extension to $\cl(\nu^{-1}(\Tt))$.

\subsection{Some preliminary estimations.} We want to provide some estimations in order to apply Lemma \ref{order} later in our construction. Let $x\in M$ and $y\in\R^n$ be such that $x+y\in\Omega$, then 
\begin{multline*}
\|\nu(x+y)-x\|\leq\|\nu(x+y)-(x+y)\|+\|y\|\\
=\dist(x+y,M)+\|y\|\leq\|x+y-x\|+\|y\|=2\|y\|.
\end{multline*}
Let $\Ff:=\{\psi:\R^d\to\R^d\ \text{linear map}\}\equiv(\R^{d,*})^d$ and $\psi_1,\ldots,\psi_r\in\Ff$. If $w\in\R^d$ and $\lambda_1,\ldots,\lambda_r\in\R$ are such that $(\phi_1\circ\psi_1)(w)+\sum_{i=1}^r\lambda_i(\phi_i\circ\psi_i)(w)\in\Omega$, then
\begin{align}
\begin{split}\label{stima}
\Big\|\nu\Big((\phi_1\circ\psi_1)(w)&+\sum_{i=1}^r\lambda_i(\phi_i\circ\psi_i)(w)\Big)-(\phi_1\circ\psi_1)(w)\Big\|\\
&\leq 2\Big\|\sum_{i=1}^r\lambda_i(\phi_i\circ\psi_i)(w)\Big\|\leq2\sum_{i=1}^r|\lambda_i|\|(\phi_i\circ\psi_i)(w)\|.
\end{split}
\end{align}

Recall that $\ol{\Bb}_d(0,\veps)$ (resp. $\Bb_d(0,\veps)$) denotes the closed ball (resp. open ball) of $\R^d$ of centre the origin and radius $\veps>0$.

\begin{lem}[Lipschitz Nash charts]\label{bound}
Let $M\subset\R^n$ be a Nash manifold and consider a Nash chart $\theta:=(\theta_1,\ldots,\theta_n):\R^d\to M$. Let $\pi:\R^n\to\R^d$ be the projection onto the first $d$ coordinates and denote $W:=(\pi\circ\theta)(\R^d)$. Assume that $W$ is open and that the map $\theta':=\pi\circ\theta:\R^d\to W$ is a Nash diffeomorphism. For each $t>0$ there exists a constant $L_t>0$ such that $\|\theta^{-1}(x)-\theta^{-1}(y)\|\leq L_t\|x-y\|$ for each $x,y\in\theta(\ol{\Bb}_d(0,t))$.
\end{lem}
\begin{proof}
Define $f:=(\theta_{d+1},\ldots,\theta_n)\circ\theta'^{-1}:W\to\R^{n-d}$ and observe that we have $\theta\circ\theta'^{-1}:W\to\theta(\R^d),\ z\mapsto(z,f(z))$. Thus, 
$$
\|z-w\|\leq\|(z,f(z))-(w,f(w))\|=\|(\theta\circ\theta'^{-1})(z)-(\theta\circ\theta'^{-1})(w)\|
$$
for each $z,w\in W$. Consequently, writing $z=(\theta'\circ\theta^{-1})(x)$ and $w=(\theta'\circ\theta^{-1})(y)$, we deduce
$$
\|(\theta'\circ\theta^{-1})(x)-\theta'\circ\theta^{-1}(y)\|\leq\|x-y\| 
$$
for each $x,y\in\theta(\R^d)=(\theta\circ\theta'^{-1})(W)$. 

By the mean value theorem there exists a constant $L_t>0$ such that
$$
\|\theta'^{-1}(z)-\theta'^{-1}(w)\|\leq L_t\|z-w\|
$$
for each $z,w\in\theta'(\ol{\Bb}_d(0,t))$. Thus,
\begin{multline*}
\|\theta^{-1}(x)-\theta^{-1}(y)\|=\|\theta'^{-1}((\theta'\circ\theta^{-1})(x))-\theta'^{-1}((\theta'\circ\theta^{-1})(y))\|\\
\leq L_t\|(\theta'\circ\theta^{-1})(x)-\theta'\circ\theta^{-1}(y)\|\leq L_t\|x-y\| 
\end{multline*}
for each $x,y\in\theta(\ol{\Bb}_d(0,t))$, as required.
\end{proof}
\begin{remark}\label{boundr}
The previous result works the same if $\pi$ is any projection $\pi:\R^n\to\R^d, x:=(x_1,\ldots,x_n)\mapsto(x_{i_1},\ldots,x_{i_d})$ for some $1\leq i_1<\ldots<i_d\leq n$.
\end{remark}

As $\Tt$ is compact, we may assume $\Tt\subset\bigcup_{i=1}^r\phi_i(\Bb_d(0,1))=M$. We may also assume, using the compactness of $\Tt$, that each $\phi_i$ is under the hypothesis of Lemma \ref{bound} (see Remark \ref{boundr}). Define $K:=\max\{\|x\|:\ x\in\bigcup_{i=1}^r(\phi_i\circ\psi_i)(\ol{\Bb}_d(0,1))\}>0$. If $w\in\Bb_d(0,1)$ and $\lambda_1,\ldots,\lambda_r\in\R$ are such that $\nu((\phi_1\circ\psi_1)(w)+\sum_{i=1}^r\lambda_i(\phi_i\circ\psi_i)(w))\in\phi_1(\Bb_d(0,1))$, then by \eqref{stima} and Lemma \ref{bound} there exists $L>0$ such that
\begin{align}\label{bound0}
\Big\|\phi_1^{-1}\Big(\nu\Big((\phi_1\circ\psi_1)(w)&+\sum_{i=1}^r\lambda_i(\phi_i\circ\psi_i)(w)\Big)\Big)-\phi_1^{-1}(\phi_1\circ\psi_1)(w)\Big\|\nonumber\\
&\leq L\Big\|\nu\Big((\phi_1\circ\psi_1)(w)+\sum_{i=1}^r\lambda_i(\phi_i\circ\psi_i)(w)\Big)-(\phi_1\circ\psi_1)(w)\Big\|\\
&\leq2L\sum_{i=1}^r|\lambda_i|\|(\phi_i\circ\psi_i)(w)\|\leq 2LK\sum_{i=1}^r|\lambda_i|.\nonumber
\end{align}

\subsection{Decomposition as a finite union of `simplices'.}\label{dus}
Consider the vectors of the standard basis ${\tt e}_i:=(0,\ldots,0,1,0,\ldots,0)$ of $\R^d$ for $i=1,\ldots,d$. Fix $k=1,\ldots,d$ and consider the convex polyhedron $\pol_k$ that is the convex hull of the origin ${\bf 0}$ and the points 
$$
{\tt e}_1,\ldots,{\tt e}_k,{\tt e}_{k+1},-{\tt e}_{k+1},\ldots,{\tt e}_d,-{\tt e}_d.
$$
We have that the polyhedron $\pol_k$ is a compact neighborhood of the origin in $\Lambda_k=\{\x_1\geq 0,\ldots,\x_k\geq 0\}$, $\pol_k\cap\partial\Lambda_k=\partial\pol_k\cap\partial\Lambda_k$ and $\Int(\pol_k)=\Int(\pol_k\cap\Int(\Lambda_k))$. Observe that $\pol_k$ is the union of the simplices $\Delta(\veps_{k+1},\ldots,\veps_d)$ of vertices the origin ${\bf 0}$ and the points
$$
{\tt e}_1,\ldots,{\tt e}_k,\veps_{k+1}{\tt e}_{k+1},\ldots,\veps_d{\tt e}_d,
$$ 
where $\veps_{k+1},\ldots,\veps_d=\pm1$. Let ${\mathfrak T}_k$ be the collection of the proper faces of all the simplices $\Delta(\veps_{k+1},\ldots,\veps_d)$ that are contained in $\partial\pol_k$. Observe that ${\mathfrak T}_k$ provides a triangulation of $\partial\pol_k$. Define 
\begin{align*}
p_k:\!&=\sum_{j=1}^k\frac{1}{2d-k+1}{\tt e}_j\\
&=\sum_{j=1}^k\frac{1}{2d-k+1}{\tt e}_j+\sum_{j=k+1}^d\frac{1}{2d-k+1}{\tt e}_j+\sum_{j=k+1}^d\frac{1}{2d-k+1}(-{\tt e}_j)+\frac{1}{2d-k+1}{\bf 0}{,}\end{align*}
which belongs to $\Int(\pol_k)$. Observe that if $k=0$, then $p_0$ is the origin. For each $\sigma\in{\mathfrak T}_k$ define $\widehat{\sigma}$ as the convex hull of $\sigma\cup\{p_k\}$, which is a simplex. Observe that $\widehat{\mathfrak T}_k:=\{\widehat{\sigma}:\ \sigma\in{\mathfrak T}_k\}$ is a triangulation of $\pol_k$ such that 
$$
\widehat{\sigma}\cap\partial\Lambda_k=\widehat{\sigma}\cap\pol_k\cap\partial\Lambda_k=\widehat{\sigma}\cap\partial\pol_k\cap\partial\Lambda_k=\sigma\cap\partial\Lambda_k,
$$ 
which is either the empty set or a face of $\sigma$ (see Figure \ref{Triangolazione}). Let $\Ff_k$ be the collection of the simplices $\widehat{\sigma}\in\widehat{\mathfrak T}_k$ of dimension $d$. 

We retake here the Nash atlas $\{\phi_i\}_{i=1}^r$ of $M=\ol{\Tt}^{\zar}$ introduced in \ref{localcharts} and we keep all the hypothesis concerning $\{\phi_i\}_{i=1}^r$ already introduced there. We may assume that $\{\phi_i(\pol_{k_i})\}_{i=1}^r$ is a covering of the compact checkerboard set $\Tt$ introduced in \ref{localcharts}. For each $i$ consider the finite family $\Ff_{k_i}$ of simplices of dimension $d$. Note that we are considering the families $\Ff_{k_{i_1}}$ and $\Ff_{k_{i_2}}$ as different families of simplices when $i_1\neq i_2$, even if $\pol_{i_1}=\pol_{i_2}$ as subsets of $\R^d$. We consider all the pairs $(\phi_i,\tau)$ where $\tau\in \Ff_{k_i}$. Observe that $\tau$ is the convex hull of $\sigma\cup\{p_{k_i}\}$ for some $\sigma\in {\mathfrak T}_{k_i}$.

Repeating the diffeomorphisms $\phi_i$ as many times as needed and reordering the diffeomorphisms $\phi_i$, we may assume that $\{\phi_i(\tau_i)\}_{i=1}^r$ is a covering of $\Tt$ such that $\tau_i$ is a $d$-dimensional simplex of $\R^d$ and $\tau_i\cap\phi_i^{-1}(\partial\Tt)$ is either the empty set or a proper face of $\tau_i$. Let $\sigma_i$ be the $(d-1)$-dimensional face of $\tau_i$ that does not contain $p_i:=p_{k_i}$. Note that $p_i$ belongs to $\phi_i^{-1}(\Reg(\Tt))$ because $p_i\in \Int(\Lambda_{k_i})\subset\phi_i^{-1}(\Reg(\Tt))$. Consequently, $\tau_i\cap\phi_i^{-1}(\partial\Tt)\subset\sigma_i$.
 
\begin{figure}[htp]
\centering
\begin{subfigure}[htp]{0.32\textwidth}
\centering
\begin{tikzpicture}[scale=0.42]
\coordinate (O) at (0,0);
\coordinate (C) at (-3,0);
\coordinate [label=below left:{\small{$-e_2$}}](D) at (0,-3);
\coordinate (A) at (3,0);
\coordinate (C) at (-3,0);
\coordinate[label=below right:{\small{$e_1$}}] (H) at (2.9,-0.1);
\coordinate[label=below left:{\small{$-e_1$}}] (G) at (-2.9,0);
\coordinate[label=above left:{\small{$e_2$}}] (B) at (0,3);
\coordinate[label=below right:{\small{$e_1$}}] (H) at (2.9,-0.1);
\coordinate[label=below left:{\small{$-e_1$}}] (G) at (-2.9,0);
\draw[->] (-4,0) -- (4,0);
\draw[->] (0,-4) -- (0,4);

\coordinate[label=below right:{\small{$k=0$}}] (G) at (2,3.2);

\filldraw[draw=black,fill=brown,fill opacity=0.1,thick] (A) -- (B) -- (C) -- (D) -- (A);
\filldraw[draw=black,fill=brown,fill opacity=0.1,thick] (A) -- (O);
\filldraw[draw=black,fill=brown,fill opacity=0.1,thick] (O) -- (B);
\filldraw[draw=black,fill=brown,fill opacity=0.1,thick] (C) -- (O);
\filldraw[draw=black,fill=brown,fill opacity=0.1,thick] (D) -- (O);

\fill[fill=blue,fill opacity=0.1,thick] (A) -- (B) -- (O)-- cycle;
\fill[fill=blue,fill opacity=0.2,thick] (C) -- (B) -- (O)-- cycle;
\fill[fill=blue,fill opacity=0.3,thick] (C) -- (O) -- (D)-- cycle;
\fill[fill=blue,fill opacity=0.4,thick] (D) -- (O) -- (A)-- cycle;
\end{tikzpicture}
\end{subfigure} 
\begin{subfigure}[htp]{0.32\textwidth}
\centering
\begin{tikzpicture}[scale=0.42]
\coordinate (O) at (0,0);
\coordinate (C) at (0.75,0);
\coordinate [label=below left:{\small{$-e_2$}}](D) at (0,-3);

\coordinate (A) at (3,0);
\coordinate (E) at (-3,0);
\coordinate[label=below right:{\small{$e_1$}}] (H) at (2.9,-0.1);
\coordinate[label=below left:{\small{$-e_1$}}] (G) at (-2.9,0);

\coordinate[label=above left:{\small{$e_2$}}] (B) at (0,3);
\draw[->,densely dotted] (-4,0) -- (4,0);
\draw[->] (0,-4) -- (0,4);

\coordinate[label=below right:{\small{$k=1$}}] (G) at (2,3.2);

\filldraw[draw=black,fill=brown,fill opacity=0.1,thick] (A) -- (B) -- (D) -- (A);
\filldraw[draw=black,fill=brown,fill opacity=0.1,thick] (B) -- (C);
\filldraw[draw=black,fill=brown,fill opacity=0.1,thick] (D) -- (C);
\filldraw[draw=black,fill=brown,fill opacity=0.1,thick] (A) -- (C);
\filldraw[draw=black,fill=brown,fill opacity=0.1] (-4,0) -- (O);
\filldraw[draw=black,fill=brown,fill opacity=0.1] (C) -- (4,0);
 
\fill[fill=blue,fill opacity=0.2,thick] (C) -- (B) -- (D)-- cycle;
\fill[fill=blue,fill opacity=0.1,thick] (C) -- (B) -- (A)-- cycle;
\fill[fill=blue,fill opacity=0.3,thick] (C) -- (A) -- (D)-- cycle;
\end{tikzpicture}
\end{subfigure} 
\begin{subfigure}[htp]{0.32\textwidth}
\centering
\begin{tikzpicture}[scale=0.42]
\coordinate (O) at (0,0);
\coordinate [label=below left:{\small{$-e_2$}}](D) at (0,-3);
\coordinate (A) at (3,0);
\coordinate (E) at (-3,0);
\coordinate[label=below right:{\small{$e_1$}}] (H) at (2.9,-0.1);
\coordinate[label=below left:{\small{$-e_1$}}] (G) at (-2.9,0);
\coordinate[label=above left:{\small{$e_2$}}] (B) at (0,3);

\coordinate (C) at (1,1);

\draw[->] (-4,0) -- (4,0);
\draw[->] (0,-4) -- (0,4);

\coordinate[label=below right:{\small{$k=2$}}] (G) at (2,3.2);

\filldraw[draw=black,fill=brown,fill opacity=0.1,thick] (A) -- (B) ;
\filldraw[draw=black,fill=brown,fill opacity=0.1,thick] (A) -- (O);
\filldraw[draw=black,fill=brown,fill opacity=0.1,thick] (B) -- (O);
\filldraw[draw=black,fill=brown,fill opacity=0.1,thick] (O) -- (C) ;
\filldraw[draw=black,fill=brown,fill opacity=0.1,thick] (A) -- (C) ;
\filldraw[draw=black,fill=brown,fill opacity=0.1,thick] (B) -- (C) ;
 
\fill[fill=blue,fill opacity=0.3,thick] (C) -- (A) -- (O)-- cycle;
\fill[fill=blue,fill opacity=0.1,thick] (C) -- (A) -- (B)-- cycle;
\fill[fill=blue,fill opacity=0.2,thick] (C) -- (B) -- (O)-- cycle;
\end{tikzpicture}
\end{subfigure}
\caption{\small{Triangulations $\widehat{\mathfrak T}_k$ of the polyhedra $\pol_k$ for $d=2$.}}
\label{Triangolazione}
\end{figure}

\subsection{Smart set of maps.}\label{spaziomappe} 
Let $\Ff:=\{\psi:\R^d\to\R^d\ \text{linear}\}\equiv(\R^{d,*})^d$ and denote $\mu:=(\mu_1,\ldots,\mu_r)\in\R^r$ and $\psi:=(\psi_1,\ldots,\psi_r)\in\Ff^r$. Let
$$
\Gamma:\R^r\times\Ff^r\to\mathcal{N}(\R^d,\R^n),\ (\mu;\psi)\mapsto\sum_{i=1}^r\mu_i(\phi_i\circ\psi_i),
$$ 
which is a continuous map if both spaces $\Ff$ and $\mathcal{N}(\R^d,\R^n)$ are endowed with the compact-open topology. The compact-open topology of $\Ff$ coincides with the topology of $\Ff$ induced by the Euclidean topology of $(\R^{d,*})^d\equiv\R^{d^2}$. Recall that
$$
\Delta_{d-1}:=\{\lambda_1\geq0,\ldots,\lambda_d\geq0,\, \lambda_1+\cdots+\lambda_d=1\}\subset\R^d.
$$
Note that an element of $\R^{d,*}$ is determined by the images of the vertices of $\Delta_{d-1}$, which is a (compact) finite set. 

Define $\Theta_0:=\{(\mu;\psi)\in\R^r\times\Ff^r:\ \Gamma(\mu,\psi)(\Delta_{d-1})\subset\nu^{-1}(\Reg(\Tt))\}$ and let us prove that it is an open semialgebraic set. The objects $\Tt$ and $\nu$ were already introduced in \ref{localcharts}.

\begin{prop}
The set $\Theta_0\subset \R^r\times\Ff^r$ is open and semialgebraic. 
\end{prop}
\begin{proof}
The fact that $\Theta_0$ is semialgebraic follows by the Tarski-Seidenberg principle \cite[Thm.2.6]{C}, because it can be described as $\Theta_0=\{x\in \R^r\times \Ff^r\, :\, \Psi(x)\}$, where $\Psi(\x)$ is a first order formula in the language of ordered fields.

Let us show now that $\Theta_0$ is open. Recall that $\Reg(\Tt)$ is an open semialgebraic subset of $\Tt$. As $\Tt$ is pure dimensional, $\Reg(\Tt)$ is open in the Nash manifold $M:=\ol{\Tt}^{\zar}$. Thus, $\nu^{-1}(\Reg(\Tt))$ is an open subset of $\R^n$. Consider now the set $\{F\in\mathcal{N}(\R^d,\R^n):\ F(\Delta_{d-1})\subset\nu^{-1}(\Reg(\Tt))\}$, which is an open subset of the open-compact topology of $\mathcal{N}(\R^d,\R^n)$. Thus, the set 
$$
\Theta_0=\Gamma^{-1}(\{F\in\mathcal{N}(\R^d,\R^n):\ F(\Delta_{n-1})\subset\nu^{-1}(\Reg(\Tt))\})
$$
is an open subset of $\R^r\times \Ff ^r$, because the map $\Gamma$ is continuous. 
\end{proof}

\subsection{Properties of $\Theta_0$.}\label{Theta0}

For each $w\in\R^d$, define the linear map 
$$
\psi_w:\R^d\to\R^d,\ (x_1,\ldots,x_d)\mapsto (x_1+\cdots+x_d)w.
$$
The restriction $\psi_w|_{\{\x_1+\cdots+\x_d=1\}}$ is the constant map $w$ and the simplex $\Delta_{d-1}\subset\{\x_1+\cdots+\x_d=1\}$. The vectors $\e_i$ were introduced in \S\ref{dus}. Let us analyze some properties of $\Theta_0$:

\noindent(1) {\em If $\phi_i(w_i)\in\Reg(\Tt)$, then $(\e_i;0,\ldots,0,\overset{(i)}{\psi_{w_i}},0,\ldots,0)\in\Theta_0$}. 

As $\Gamma(\e_i;0,\ldots,0,\overset{(i)}{\psi_{w_i}},0,\ldots,0)(\Delta_{d-1})=\{\phi_i(w_i)\}\subset\Reg(\Tt)\subset\nu^{-1}(\Reg(\Tt))$, we have 
$$
(\e_i;0,\ldots,0,\overset{(i)}{\psi_{w_i}},0,\ldots,0)\in\Theta_0.
$$

\noindent(2) {\em Given $1\leq i,j\leq r$, let $w_i\in\Int(\Lambda_{k_i})$ and $z_j\in\Int(\Lambda_{k_j})$ be such that $\phi_i(w_i)=\phi_j(z_j)$. Then 
$$
(\e_i;0,\ldots,0,\overset{(i)}{\psi_{w_i}},0,\ldots,0) \text{ and } (\e_j;0,\ldots,0,\overset{(j)}{\psi_{w_j}},0,\ldots,0)
$$
belong to the same connected component of $\Theta_0$}.

Observe that
$$
\phi_i(w_i)=(1-t)\phi_i(w_i)+t\phi_j(z_j)=\phi_j(z_j)\in\Reg(\Tt)\subset\nu^{-1}(\Reg(\Tt))
$$
for $t\in[0,1]$. Thus,
\begin{multline*}
\Gamma(((1-t)\e_i+t\e_j;0,\ldots,0,\overset{(i)}{\psi_{w_i}},0,\ldots,0,\overset{(j)}{\psi_{z_j}},0,\ldots,0))(\Delta_{d-1})\\
=((1-t)(\phi_i\circ\psi_{w_i})+t(\phi_j\circ\psi_{z_j}))(\Delta_{d-1})=\{(1-t)\phi_i(w_i)+t\phi_j(z_j)\}\subset\nu^{-1}(\Reg(\Tt)),
\end{multline*}
so $((1-t)\e_i+t\e_j;0,\ldots,0,\overset{(i)}{\psi_{w_i}},0,\ldots,0,\overset{(j)}{\psi_{z_j}},0,\ldots,0)\in\Theta_0$ for $t\in[0,1]$. Consequently, the connected set 
$$
\Cc_1:=\{((1-t)\e_i+t\e_j;0,\ldots,0,\overset{(i)}{\psi_{w_i}},0,\ldots,0,\overset{(j)}{\psi_{z_j}},0,\ldots,0)\, :\, t\in [0,1]\},
$$
is contained in one of the connected components of $\Theta_0$. In addition, for $t\in[0,1]$
\begin{align*}
\Gamma(\e_i;0,\ldots,0,\overset{(i)}{\psi_{w_i}},0,\ldots,0,\overset{(j)}{\psi_{tz_j}},0,\ldots,0)(\Delta_{d-1})=\{\phi_i(w_i)\}\subset\nu^{-1}(\Reg(\Tt))\\
\Gamma(\e_j;0,\ldots,0,\overset{(i)}{\psi_{tw_i}},0,\ldots,0,\overset{(j)}{\psi_{z_j}},0,\ldots,0)(\Delta_{d-1})=\{\phi_j(z_j)\}\subset\nu^{-1}(\Reg(\Tt))
\end{align*}.
As $\psi_{tz_j}=t\psi_{z_j}$ and $\psi_{tw_i}=t\psi_{w_i}$ for $t\in[0,1]$, we deduce
\begin{align*}
(\e_i;0,\ldots,0,\overset{(i)}{\psi_{w_i}},0,\ldots,0,\overset{(j)}{t\psi_{z_j}},0,\ldots,0)\in\Theta_0\quad\text{for $t\in[0,1]$},\\
(\e_j;0,\ldots,0,\overset{(i)}{t\psi_{w_i}},0,\ldots,0,\overset{(j)}{\psi_{z_j}},0,\ldots,0)\in\Theta_0\quad\text{for $t\in[0,1]$.}
\end{align*}
Thus, the connected sets
\begin{align*}
&\Cc_2:=\{(\e_i;0,\ldots,0,\overset{(i)}{\psi_{w_i}},0,\ldots,0,\overset{(j)}{t\psi_{z_j}},0,\ldots,0)\, :\, t\in [0,1]\},\\
&\Cc_3:=\{(\e_j;0,\ldots,0,\overset{(i)}{t\psi_{w_i}},0,\ldots,0,\overset{(j)}{\psi_{z_j}},0,\ldots,0)\, :\, t\in [0,1]\}
\end{align*}
are contained in a connected component of $\Theta_0$. As
\begin{align*}
&\Cc_1\cap \Cc_2=\{(\e_i;0,\ldots,0,\overset{(i)}{\psi_{w_i}},0,\ldots,0,\overset{(j)}{\psi_{z_j}},0,\ldots,0)\},\\
&\Cc_1\cap \Cc_3=\{(\e_j;0,\ldots,0,\overset{(i)}{\psi_{w_i}},0,\ldots,0,\overset{(j)}{\psi_{z_j}},0,\ldots,0)\},
\end{align*}
we deduce $\Cc_1\cup\Cc_2\cup\Cc_3$ is a connected subset of $\Theta_0$ contained in one of its connected components.

We conclude that $(\e_i;0,\ldots,0,\overset{(i)}{\psi_{w_i}},0,\ldots,0)\in\Cc_2$ and $
(\e_j;0,\ldots,0,\overset{(j)}{\psi_{w_j}},0,\ldots,0)\in\Cc_3$ belong to the same connected component of $\Theta_0$.

\noindent(3) {\em If $w_i\in\Int(\Lambda_{k_i})$ and $z_j\in\Int(\Lambda_{k_j})$, then 
$$
(\e_i;0,\ldots,0,\overset{(i)}{\psi_{w_i}},0,\ldots,0)\quad\text{and}\quad(\e_j;0,\ldots,0,\overset{(j)}{\psi_{z_j}},0,\ldots,0)
$$ 
belong to the same connected component of $\Theta_0$}.

As $\Reg(\Tt)$ is connected and $\{\Int(\phi_i(\Lambda_{k_i}))\}_{i=1}^r$ is an open semialgebraic covering of $\Reg(\Tt)$, given $1\leq i,j\leq r$, there exists a chain $\{\phi_{i_\ell}(\Lambda_{k_{i_\ell}})\}_{\ell=1}^s$ such that $i=i_1$, $j=i_s$ and $\phi_{i_\ell}(\Int(\Lambda_{k_{i_\ell}}))\cap\phi_{i_{\ell+1}}(\Int(\Lambda_{k_{i_{\ell+1}}}))\neq\varnothing$ for each $\ell$. Thus, we may assume $\phi_i(\Int(\Lambda_{k_i}))\cap\phi_j(\Int(\Lambda_{k_j}))\neq\varnothing$. Now, by (2) it is enough to prove the statement for $i=j$. Observe that 
$$
t\psi_{w_i}+(1-t)\psi_{z_i}=\psi_{tw_i+(1-t)z_i}
$$ 
for each $t\in[0,1]$. As $w_i,z_i\in\Int(\Lambda_{k_i})$ and the latter is convex, we have that $tw_i+(1-t)z_i\in\Int(\Lambda_{k_i})$ for each $t\in[0,1]$, so
$$
\Gamma((\e_i;0,\ldots,0,\overset{(i)}{t\psi_{w_i}+(1-t)\psi_{z_i}},0,\ldots,0))
=\phi_i(tw_i+(1-t)z_i)\in\Reg(\Tt)\subset\nu^{-1}(\Reg(\Tt)).
$$
Thus, 
$$
\{(\e_i;0,\ldots,0,\overset{(i)}{t\psi_{w_i}+(1-t)\psi_{z_i}},0,\ldots,0):\ t\in[0,1]\}\subset\Theta_0
$$
is connected, so $(\e_i;0,\ldots,0,\overset{(i)}{\psi_{w_i}},0,\ldots,0)$ and $(\e_i;0,\ldots,0,\overset{(i)}{\psi_{z_i}},0,\ldots,0)$ belong to the same connected component of $\Theta_0$. 

\noindent(4) {\em There exists a connected component $\Theta$ of $\Theta_0$ that contains all the connected sets $\Xi_i:=\{\e_i\}\times\Ff^{i-1}\times\{\psi_{w_i}:\ w_i\in\Int(\Lambda_{k_i})\}\times\Ff^{r-i}$ for $i=1,\ldots,r$.}

Observe that
\begin{multline*}
\Gamma((\e_i; \psi_1,\ldots,\psi_{i-1},\psi_{w_i},\psi_{i+1},\ldots,\psi_r))(\Delta_{d-1})\\=\phi_i(\psi_{w_i})(\Delta_{d-1})=\phi_i(w_i)\in \Reg(\Tt)\subset\nu^{-1}(\Reg(\Tt)),
\end{multline*}
so $\Xi_i\subset\Theta_0$ and it is a connected set, because it is a finite product of connected sets. Thus, there exists a connected component $\Theta_i$ of $\Theta_0$ that contains $\Xi_i$ for $i=1,\ldots,r$. By (3) we deduce $\Theta_i=\Theta_j$ if $i\neq j$, so there exists a connected component $\Theta$ of $\Theta_0$ that contains all the connected sets $\Xi_i$ for $i=1,\ldots,r$.

\noindent(5) Let $\Theta$ be the connected component of $\Theta_0$ introduced in (4). {\em If $\psi_i\in\Ff$ satisfies $\psi_i(\Delta_{d-1})\subset\Lambda_{k_i}$, then $(\e_i;0,\ldots,0,\overset{(i)}{\psi_i},0,\ldots,0)\in\cl(\Theta)$. If in addition $\psi_i(\Delta_{d-1})\subset\Int(\Lambda_{k_i})$, then 
$$
(\e_i;0,\ldots,0,\overset{(i)}{\psi_i},0,\ldots,0)\in\Theta.
$$
}
\indent Let $w_i\in\Int(\Lambda_{k_i})$. Recall that by (4) $(\e_i;0,\ldots,0,\overset{(i)}{\psi_{tw_i}},0,\ldots,0)\in\Theta$ for each $t\in(0,1]$, because $\Lambda_{k_i}$ is an open cone (so $tw_i\in\Int(\Lambda_{k_i})$ for each $t\in(0,1]$). Consider the Nash path 
$$
\alpha:(0,1]\to\R^r\times\Ff^r,\, t\mapsto(\e_i;0,\ldots,0,\overset{(i)}{(1-t)\psi_i+t\psi_{w_i}},0,\ldots,0).
$$ 
We claim: {\em$\alpha(t)\in\Theta_0$ for $t\in(0,1]$}.

This is because $w_i\in\Int(\Lambda_{k_i})$ and $\psi_i(\Delta_{d-1})\subset\Lambda_{k_i}$ for $t\in(0,1]$, so $((1-t)\psi_i+t\psi_{w_i})(\Delta_{d-1})\subset\Int(\Lambda_{k_i})$ for $t\in(0,1]$ (see \cite[Lemma 11.2.4]{ber1}). Thus, 
$$
\Gamma(\alpha(t))(\Delta_{d-1})\subset\phi_i((1-t)\psi_i+t\psi_{w_i})(\Delta_{d-1})\subset\phi_i(\Int(\Lambda_{k_i}))\subset\Reg(\Tt)
$$
for each $t\in(0,1]$, so $\alpha(t)\in\Theta_0$ for each $t\in(0,1]$. 

As $\alpha(1)=(\e_i;0,\ldots,0,\overset{(i)}{\psi_{w_i}},0,\ldots,0)\in\Theta$, we have $\alpha(0)=(\e_i;0,\ldots,0,\overset{(i)}{\psi_i},0,\ldots,0)\in\cl(\Theta)$.
 
If in addition $\psi_i(\Delta_{d-1})\subset\Int(\Lambda_{k_i})$, then $(\e_i;0,\ldots,0,\overset{(i)}{\psi_i},0,\ldots,0)\in\Theta_0$, so 
$$
(\e_i;0,\ldots,0,\overset{(i)}{\psi_i},0,\ldots,0)\in\Theta_0\cap\cl(\Theta)=\Theta.
$$ 

\noindent(6){\em If $(\mu;\psi)\in\cl(\Theta_0)$, then $\Gamma(\mu;\psi)(\Delta_{d-1})\subset\cl(\nu^{-1}(\Tt))$ and $\nu(\Gamma(\mu;\psi)(\Delta_{d-1}))\subset\Tt$.}

By the curve selection lemma \cite[Thm.2.5.5]{bcr} there exists a continuous semialgebraic path $\alpha:[0,1]\to\cl(\Theta_0)$ such that $\alpha(0)=(\mu;\psi)$ and $\alpha((0,1])\subset\Theta_0$. This means that for each $t\in(0,1]$ one has $
\Gamma(\alpha(t))(\Delta_{d-1})\subset\nu^{-1}(\Reg(\Tt))$. If $x\in\Delta_{d-1}$, then $\Gamma(\alpha(t))(x):[0,1]\to\R^n$ is a continuous semialgebraic path such that $\Gamma(\alpha(t))(x)\subset\nu^{-1}(\Reg(\Tt))$ for each $t\in(0,1]$, so 
$$
\Gamma(\alpha(0))(x)\in\cl(\nu^{-1}(\Reg(\Tt)))\subset\cl(\nu^{-1}(\Tt)).
$$ 
Consequently, $\Gamma(\mu;\psi)(\Delta_{d-1})\subset\cl(\nu^{-1}(\Tt))$.

As $\cl(\nu^{-1}(\Tt))$ is a compact set and $\nu$ admits a Nash extension to $\cl(\nu^{-1}(\Tt))$ (see Subsection \ref{localcharts}), we deduce $\nu:\cl(\nu^{-1}(\Tt))\to M=\ol{\Tt}^{\zar}$ is proper and
$$
\nu(\Gamma(\mu;\psi)(\Delta_{d-1}))\subset\nu(\cl(\nu^{-1}(\Tt)))=\cl(\Tt)=\Tt.
$$

\noindent(7) Let ${\tt j}:\mathcal{N}(\R^d,\R^n)\to\mathcal{N}(\Delta_{d-1},\R^n),\ f\mapsto f|_{\Delta_{n-1}}$, which is continuous if we endow both spaces with the compact-open topology. {\em Then the composition
$$
\nu_*\circ{\tt j}\circ\Gamma:\cl(\Theta)\to\mathcal{N}(\Delta_{d-1},\Tt),\ (\mu;\psi)\mapsto\nu\circ(\Gamma((\mu;\psi))|_{\Delta_{d-1}})
$$
is continuous.}

\noindent(8) {\em If $\beta:[0,1]\to\cl(\Theta_0)$ is a Nash path, then $B:[0,1]\times\Delta_{d-1}\to\Tt,\ (t,x)\mapsto\nu\circ(\Gamma(\beta(t)))(x)$ is a Nash map.}

\subsection{Nash images of the closed ball.} 
We are finally ready to finish the proof of Theorem \ref{main1}. As we have seen in Paragraph \ref{ridchekcer7}, it is enough to prove Theorem \ref{riduzione}. Let us prove: \em Given a compact checkerboard set $\Tt\subset\R^n$ of dimension $d\geq 2$, there exists a Nash map $F:\Delta_{d-1}\times[0,1]\to\R^n$ such that $F(\Delta_{d-1}\times[0,1])=\Tt$.\em

\begin{proof}[Proof of Theorem \em\ref{riduzione}]
We keep all the notations introduced in Subsections \ref{localcharts}, \ref{dus} and \ref{spaziomappe}. We also keep all the assumptions done along these subsections. Recall that $\Tt=\bigcup_{i=1}^r\phi_i(\tau_i)\subset\bigcup_{i=1}^r\phi_i(\Bb_d(0,1))$, where $\tau_i\subset\Lambda_{k_i}$ is a $d$-dimensional simplex such that $\tau_i\cap\phi_i^{-1}(\partial\Tt)$ is either empty or a proper face of $\tau_i$ contained in a $(d-1)$-dimensional face $\sigma_i$ of $\tau_i$ that does not contain a vertex $p_i\in\phi_i^{-1}(\Reg(\Tt))$ of $\tau_i$. If $\tau_i\cap\phi_i^{-1}(\partial\Tt)=\varnothing$, we denote with $\sigma_i$ the facet of $\tau_i$ that does not contain the origin of $\R^d$ (this situation corresponds to the case $k=0$ in Subsection \ref{dus} and the origin is the point $p_0$ introduced there). In both cases the remaining vertex $p_i$ of $\tau_i$ belongs to $\phi_i^{-1}(\Reg(\Tt))$ and $\tau_i$ is the convex hull of $\sigma_i\cup\{p_i\}$, see Subsection \ref{dus}. In addition, $\Tt=\bigcup_{i=1}^r\phi_i(\Lambda_{k_i})$ (see Subsection \ref{localcharts}).

Denote the vertices of $\sigma_i$ with $v_{ij}$ for $j=1,\ldots,d$. Let $H_{ik}$ be the hyperplanes generated by the facets of $\tau_i$ that contain the vertex $p_i$ and assume $v_{ij}\not\in H_{ij}$ and $\tau_i\subset\bigcap_{j=1}^dH_{ij}^+$.

Let $\alpha_{ij}:[-\delta,1+\delta]\to\R^d$ be Nash paths satisfying the conditions of Lemma \ref{simplex2} (see also Remark \ref{simplex2r}(ii)) applied to the pair $\tau_i\subset\Lambda_{k_i}$. Consider the Nash path
$$
A_i:[-\delta,1+\delta]\to\{\e_i\}\times\Ff^r,\ t\mapsto\bigg(\e_i;0,\ldots,0,\sum_{j=1}^d\alpha_{ij}(t)\x_j,0,\ldots,0\bigg),
$$
which satisfies in particular 
$$
\phi_i(\tau_i)\subset\bigcup_{t\in[0,1]}\Gamma(A_i(t))(\Delta_{d-1})\subset\phi_i(\Lambda_{k_i})\quad\text{ and }\quad\bigcup_{t\in[-\delta,0)\cup(1,1+\delta]}\Gamma(A_i(t))(\Delta_{d-1})\subset\phi_i(\Int(\tau_i)). 
$$
By Lemma \ref{simplex2} (and Remark \ref{simplex2r}(ii)) we have in addition
$$
\Gamma(A_i(t))(\Delta_{d-1})\begin{cases}
=\phi_i(\sigma_i)\subset\phi_i(\Lambda_{k_i})\subset\Tt&\text{if $t=0$,}\\
\subset\phi_i(\Int(\Lambda_{k_i}))\subset\Reg(\Tt)&\text{if $t\in[-\delta,1+\delta]\setminus\{0\}$.}
\end{cases}
$$
Thus, $A_i(t)\in\Theta$ if $t\in[-\delta,1+\delta]\setminus\{0\}$ and $\zeta_i:=A_i(0)\in\cl(\Theta)$, see Property \ref{Theta0}(5). Define the linear maps $\eta_i:=\sum_{j=1}^d\alpha_{ij}(-\delta)\x_j\in \Ff$ and $\xi_i:=\sum_{j=1}^d\alpha_{ij}(1+\delta)\x_j\in\Ff$. By Lemma \ref{simplex2} (and Remark \ref{simplex2r}(ii) $\eta_i(\Delta_{d-1}),\xi_i(\Delta_{d-1})\subset\Int(\tau_i)\subset\Int(\Lambda_{k_i})$ and we deduce by Property \ref{Theta0}(5)
$$
(\e_i;0,\ldots,0,\overset{(i)}{\eta_i},0,\ldots,0),\ (\e_i;0,\ldots,0,\overset{(i)}{\xi_i},0,\ldots,0)\in\Theta.
$$ 
Up to repeating the charts $\phi_i$ as many times as needed, we may assume 
$$
\phi_i(\Int(\Lambda_{k_i}))\cap\phi_{i+1}(\Int(\Lambda_{k_{i+1}}))\neq\varnothing
$$
and let $\phi_i(w_i)=\phi_{i+1}(z_{i+1})\in\phi_i(\Int(\Lambda_{k_i}))\cap\phi_{i+1}(\Int(\Lambda_{k_{i+1}}))$. Observe that
\begin{align}
&\phi_i((1-t)\xi_i+t\psi_{w_i})(\Delta_{d-1})\subset\phi_i(\Int(\Lambda_{k_i}))\subset\Reg(\Tt)\subset\nu^{-1}(\Reg(\Tt)),\label{bi}\\
&\phi_{i+1}((1-t)\psi_{z_{i+1}}+t\eta_{i+1})(\Delta_{d-1})\subset\phi_{i+1}(\Int(\Lambda_{k_{i+1}}))\subset\Reg(\Tt)\subset\nu^{-1}(\Reg(\Tt))\label{di}
\end{align}
for $t\in[0,1]$. Consider the Nash paths
\begin{align*}
&B_i:[0,1]\to\Theta_0,\ t\mapsto(\e_i;0,\ldots,0,(1-t)\xi_i+t\psi_{w_i},t\psi_{z_{i+1}},0,\ldots,0),\\
&C_i:[0,1]\to\Theta_0,\ t\mapsto((1-t)\e_i+t\e_{i+1};0,\ldots,0,\psi_{w_i},\psi_{z_{i+1}},0,\ldots,0),\\
&D_i:[0,1]\to\Theta_0,\ t\mapsto(\e_{i+1};0,\ldots,0,(1-t)\psi_{w_i},(1-t)\psi_{z_{i+1}}+t\eta_{i+1},0,\ldots,0).
\end{align*}
We have $B_i([0,1])\subset\Theta_0$ because
$$
\Gamma(B_i(t))(\Delta_{d-1})=\phi_i((1-t)\xi_i+t\psi_{w_i})(\Delta_{d-1})\subset\nu^{-1}(\Reg(\Tt))
$$
for $t\in[0,1]$ (see \eqref{bi}). In addition, $C_i([0,1])\subset\Theta_0$ by the proof of Property \ref{Theta0}(2) and $D_i([0,1])\subset\Theta_0$, because
$$
\Gamma(D_i(t))(\Delta_{d-1})=\phi_{i+1}((1-t)\psi_{z_{i+1}}+t\eta_{i+1})(\Delta_{d-1})\subset\nu^{-1}(\Reg(\Tt))
$$
for $t\in[0,1]$ (see \eqref{di}). By Property \ref{Theta0}(4) $,
B_i(1)=(\e_i;0,\ldots,0,\overset{(i)}{\psi_{w_i}},\psi_{z_{i+1}},0,\ldots,0)\in\Theta$ 
and, as $B_i([0,1])\subset\Theta_0$ is connected, we deduce $B_i([0,1])\subset\Theta$. Analogously, $C_i(0)=B_i(1)\in\Theta$ and $C_i([0,1])\subset\Theta_0$ is connected, so $C_i([0,1])\subset\Theta$. In addition, 
$$
D_i(0)=C_i(1)=(\e_{i+1};0,\ldots,0,\psi_{w_i},\overset{(i+1)}{\psi_{z_{i+1}}},0,\ldots,0)\in\Theta
$$ 
and $D_i([0,1])\subset\Theta_0$ is connected, so $D_i([0,1])\subset\Theta$.

Fix times $0<t_1<s_1<\cdots<t_r<s_r<1$ and denote 
$$
\chi_i:=A_i(1)=(\e_i,0,\ldots,0,\overset{(i)}{(\x_1+\cdots+\x_d)p_i},0,\ldots,0)\in\Theta.
$$
Observe that $(\x_1+\cdots+\x_d)p_i$ is the linear map $\R^d\to\R^d$ that takes the constant value $p_i\in\Int(\Lambda_{k_i})$ on the hyperplane $\{\x_1+\cdots+\x_d=1\}$. Consider the continuous semialgebraic path obtained concatenating the previous paths:
$$
E:=\Bigast_{i=1}^r(A_i*B_i*C_i*D_i):[0,1]\to\Theta\cup\{\zeta_1,\ldots,\zeta_r\}
$$
and assume (after reparameterizing the paths) $E(t_i)=\zeta_i$, $E_j(s_i)=\chi_i$ and $E|_{[t_i,s_i]}$ is an affine reparameterization of $A_i|_{[0,1]}$. Let $\rho>0$ be such that $E$ is Nash on 
$$
I:=\bigcup_{i=1}^r\Big([t_i-\rho,t_i+\rho]\cup[s_i-\rho,s_i+\rho]\Big).
$$
By Lemma \ref{icsl} we can approximate the continuous semialgebraic path $E$ by a polynomial path $\gamma:[0,1]\to\Theta\cup\{\zeta_1,\ldots,\zeta_r\}$, such that: 
\begin{itemize}
\item[{\rm (i)}] $\gamma(t_i)=E(t_i)=\zeta_i$, $\gamma'(t_i)=E'(t_i)$, $\gamma''(t_i)=E''(t_i)$ and $\gamma'''(t_i)=E'''(t_i)$ for each $i=1,\ldots,r$.
\item[{\rm (ii)}] $\gamma(s_i)=E(s_i)=\chi_i$, $\gamma'(s_i)=E'(s_i)$, $\gamma''(s_i)=E''(s_i)$ and $\gamma'''(s_i)=E'''(s_i)$ for each $i=1,\ldots,r$.
\item[{\rm (iii)}] $\|\gamma-E\|$, $\|\gamma'-E'\|_I$, $\|\gamma''-E''\|_I$ and $\|\gamma'''-E'''\|_I$ are small enough.
\end{itemize}

Write $\gamma:=(\mu;\psi_1,\ldots,\psi_r)$ and $\mu:=(\mu_1,\ldots,\mu_r)$. As 
$$
A_i(t)=\bigg(\e_i;0,\ldots,0,\sum_{j=1}^d\alpha_{ij}(t)\x_j,0,\ldots,0\bigg),
$$
we deduce by (i) and (ii) above that
$$
\mu_i(s_j)=\mu_i(t_j)=
\begin{cases}
1 \text{ if } i=j\\
0 \text{ if } i\neq j,
\end{cases}
$$
$\mu_i'(t_i)=\mu_i'(s_i)=0$, $\mu''_i(t_i)=\mu''_i(s_i)=0$ and $\mu'''_i(t_i)=\mu'''_i(s_i)=0$. Observe that 
$$
\psi_i:[t_i,s_i]\times\Delta_{d-1}\to\Lambda_{k_i},\ (t,\lambda)\mapsto\sum_{j=1}^d\alpha_{ij}(t)\lambda_j.
$$ 
By Lemma \ref{simplex2} $\tau_i\subset\psi_i([t_i,s_i]\times\Delta_{d-1})\subset\Lambda_{k_i}$. Consider
$$
\Gamma(\gamma)=\sum_{i=1}^r\mu_i(t)(\phi_i\circ\psi_i)(t,\x):[0,1]\times\R^d\to\R^n.
$$
We have $\Gamma(\gamma)(\{t\}\times\Delta_{d-1})\subset\nu^{-1}(\Reg(\Tt))$ if $t\in[0,1]\setminus\{t_1,\ldots,t_r\}$, whereas $\Gamma(\gamma)(\{t_i\}\times\Delta_{d-1})=\phi_i(\sigma_i)\in\Tt$ for $i=1,\ldots,r$. This means that 
\begin{equation}\label{234}
\nu(\Gamma(\gamma)([0,1]\times\Delta_{d-1}))\subset\Tt.
\end{equation}

Fix $i=1,\ldots,r$ and denote
$$
\lambda_{ij}:=\begin{cases}
\mu_i-1&\text{if $j=i$,}\\
\mu_j&\text{if $i\neq j$.}
\end{cases}
$$
Observe that $\lambda_{ij}^{(\ell)}(t_i)=\lambda_{ij}^{(\ell)}(s_i)=0$ for $1\leq i,j\leq r$, $\ell=0,1,2,3$ and each $\lambda_{ij}$ is as close to zero as needed. By \eqref{bound0} there exist $L,K>0$ (depending only on $\phi_j,\psi_j$) such that
$$
\Big\|\phi_i^{-1}\Big(\nu\Big((\phi_i\circ\psi_i)(t,x)+\sum_{j=1}^r\lambda_{ij}(t)(\phi_j\circ\psi_j)(t,x)\Big)\Big)-\psi_i(t,x)\Big\|\leq2LK\sum_{j=1}^r|\lambda_{ij}(t)|.
$$
As $\lambda_{ij}^{(\ell)}(t_i)=\lambda_{ij}^{(\ell)}(s_i)=0$ for $1\leq i,j\leq r$, $\ell=0,1,2,3$, we deduce by Lemma \ref{order} that the $\ell$th partial derivatives with respect to $\t$ of the Nash maps 
\begin{multline*}
\phi_i^{-1}(\nu(\Gamma(\gamma)))(t,x)=\phi_i^{-1}\Big(\nu\Big(\sum_{j=1}^r\mu_j(t)(\phi_j\circ\psi_j)(t,x)\Big)\Big)\\
=\phi_i^{-1}\Big(\nu\Big((\phi_i\circ\psi_i)(t,x)+\sum_{j=1}^r\lambda_{ij}(t)(\phi_j\circ\psi_j)(t,x)\Big)\Big)
\end{multline*}
and $\psi_i(t,x)$ coincide for $\ell=0,1,2,3$ at the points $(t_i,x)$ (resp. at the points $(s_i,x)$) for $x\in\Delta_{d-1}$ and $1\leq i\leq r$. By Lemma \ref{simplex2} we deduce $\tau_i\subset\phi_i^{-1}(\nu(\Gamma(\gamma)([t_i,s_i]\times\Delta_{d-1})))\subset\Lambda_{k_i}$ because $\tau_i\subset\psi_i([t_i,s_i]\times\Delta_{d-1})\subset\Lambda_{k_i}$. Thus,
$$
\phi_i(\tau_i)\subset\nu(\Gamma(\gamma))([t_i,s_i]\times\Delta_{d-1})\subset\phi_i(\Lambda_{k_i})\subset\Tt
$$ 
for $i=1,\ldots,r$, so by \eqref{234}
$$
\Tt=\bigcup_{i=1}^r\phi_i(\tau_i)\subset\nu(\Gamma(\gamma))\Big(\Big(\bigcup_{i=1}^r[t_i,s_i]\Big)\times\Delta_{d-1}\Big)\subset\nu(\Gamma(\gamma)([0,1]\times\Delta_{d-1}))\subset\Tt.
$$
Consequently, $\nu(\Gamma(\gamma)([0,1]\times\Delta_{d-1}))=\Tt$, as required.
\end{proof}

\section{Proofs of Theorems \ref{nice0} and \ref{nice2}}\label{s4}

In this section we prove Theorems \ref{nice0} and \ref{nice2}. To prove Theorem \ref{nice0}, by Theorem \ref{main1} it is enough to show the following.

\begin{lem}\label{nice01}
Let $\Ss\subset\R^m$ be any semialgebraic set of dimension $d$. Then there exists a regular map $f:\R^m\to\R^d$ such that $f(\Ss)=\ol{\Bb}_d$.
\end{lem}
\begin{proof}
Let $p\in\Ss$ be a regular point of $\Ss$ such that $\dim(\Ss_p)=d$, let $p+T_p\Ss$ be the affine tangent space to $\Reg(\Ss)$ at $p$ and $\pi:\R^m\to p+T_p\Ss$ the orthogonal projection of $\R^m$ onto $p+T_p\Ss$. There exist $\veps>0$ and a compact neighborhood $W^p\subset\Reg(\Ss)$ of $p$ such that $\pi|_{W^p}:W^p\to\ol{\Bb}_m(p,\veps)\cap(p+T_p\Ss)$ is a Nash diffeomorphism. For simplicity we assume that $p$ is the origin and $\veps=1$, so there exists an isometry $h:p+T_p\Ss\to\R^d$ that maps $\ol{\Bb}_m(p,\veps)\cap(p+T_p\Ss)$ onto the closed unit ball $\ol{\Bb}_d$. Consider the inverse of the stereographic projection
\begin{align*}
\varphi:\R^d&\to\sph^d\setminus\{(0,\ldots,1)\},\\
x:=(x_1,\ldots,x_d)&\mapsto\Big(\frac{2x_1}{1+\|x\|^2},\ldots,\frac{2x_d}{1+\|x\|^2},\frac{-1+\|x\|^2}{1+\|x\|^2}\Big).
\end{align*}
Let $\pi':\R^{d+1}\to\R^d,\ (x_1,\ldots,x_{d+1})\mapsto(x_1,\ldots,x_d)$ be the projection onto the first $d$ coordinates and observe that $\pi'\circ\varphi:\R^d\to\R^d$ satisfies $(\pi'\circ\varphi)(\R^d)=(\pi'\circ\varphi)(\ol{\Bb}_d)=\ol{\Bb}_d$. The surjective regular map $g:=\pi'\circ\varphi\circ h:p+T_p\Ss\to\ol{\Bb}_d$ fulfills 
$$
g(\ol{\Bb}_m(p,\veps)\cap(p+T_p\Tt))=g(p+T_p\Tt)=\ol{\Bb}_d\subset\R^d.
$$
In particular, $g(A)=\ol{\Bb}_d$ for each $A$ such that $\ol{\Bb}_m(p,\veps)\cap(p+T_p\Tt)\subset A\subset p+T_p\Tt$. Thus, the composition $g\circ\pi:\R^m\to\R^d$ is a regular map satisfying 
$$
(g\circ\pi)(\Tt)=g(\pi(\Tt))=g(\pi(W^p))=g(\ol{\Bb}_m(p,\veps)\cap(p+T_p\Tt))=\ol{\Bb}_d,
$$
as required.
\end{proof}

We prove next Theorem \ref{nice2}. By Theorem \ref{main0} it is enough to consider the case $\Tt=\R^d$ for $d\geq2$.

\begin{lem}\label{nice1}
Let $\Ss\subset\R^m$ be a semialgebraic set and let $d\geq 2$. Assume that $\cl(\Ss^{(d)})\cap\Ss$ is not compact. Then there exists a Nash map $f:\R^m\to\R^d$ such that $f(\Ss)=\R^d$.
\end{lem}
\begin{proof}
If $m=d$ and $\Ss=\R^d$, there is nothing to prove. Thus, let us assume $\R^m\setminus\Ss\neq\varnothing$. We may assume that $\Ss^{(d)}$ is unbounded. Otherwise, $\Ss^{(d)}$ is bounded and not closed (because $\cl(\Ss^{(d)})\cap\Ss$ is not compact), so there exists $p\in\cl(\Ss^{(d)})\setminus\Ss$. Consider the Nash map 
$$
h:\R^m\setminus\{p\}\to\R^{m+1},\ x\mapsto\Big(x,\frac{1}{\|x-p\|}\Big),
$$
which is a Nash diffeomorphism onto its image. Observe that $h(\Ss^{(d)})\subset h(\Ss)\subset\R^{m+1}$ is unbounded. We identify $h(\Ss)$ with $\Ss$ and $h(\Ss^{(d)})$ with $\Ss^{(d)}$.

Consider the immersion $\psi_1:\R^m\to\R\PP^m$, $x\mapsto[1:x]$. As $\Ss^{(d)}$ is unbounded, we may assume 
$$
[0:\cdots:0:\overset{(d+1)}1:0:\cdots:0]\in \cl_{\R\PP^m}(\Ss^{(d)}).
$$
Consider the projection 
$$
\widehat{\pi}:\R\PP^m\to\R\PP^d,\, [x_0:x_1:\cdots:x_m]\mapsto[x_0:x_1:\cdots:x_d]
$$ 
whose restriction to $\R^m$ is the projection $\pi:\R^m\to\R^d$, $(x_1,\ldots,x_m)\mapsto(x_1,\ldots,x_d)$. As 
$$
\widehat{\pi}([0:\cdots:0:1:0:\cdots:0])=[0:\cdots:0:1]\in \cl_{\R\PP^d}(\pi(\Ss^{(d)})),
$$
we deduce $\pi(\Ss)$ is not bounded. Thus, taking $\pi(\Ss)$ instead of $\Ss$, we may assume $m=d$, $\dim(\Ss)=d$ and $\Ss^{(d)}$ unbounded. 

If $\Ss=\R^d$, we are done, otherwise, we may assume after a translation that $0\not\in\Ss$. Consider the inversion ${\tt i}:\R^d\setminus\{0\}\to\R^d\setminus\{0\},\ x\mapsto\frac{x}{\|x\|}$, which is a Nash involution of $\R^d\setminus\{0\}$. Thus, at this point $\Ss\subset\R^d$ is a semialgebraic set of dimension $d$ such that $0\in \cl(\Ss^{(d)})\setminus\Ss$. Observe that $\Reg(\Ss^{(d)})$ is an open subset of $\R^d$ adherent to the origin.

By the Nash curve selection lemma \cite[Prop.8.1.13]{bcr} there exists a Nash arc 
$$
\alpha:=(\alpha_1,\ldots,\alpha_d):[0,1]\to \Reg(\Ss^{(d)})\cup\{0\}
$$ 
such that $\alpha((0,1])\subset \Reg(\Ss^{(d)})$ and $\alpha(0)=0$. After a linear change of coordinates we may assume that $\alpha((0,1])\cap\{\x_1=0\}=\varnothing$. Consider now the Nash map 
$$
g:\R^d\setminus\{0\}\to\{\x_1>0\},\ (x_1,\ldots,x_d)\mapsto(\|x\|,x_2,\ldots,x_d).
$$ 
As $g|_{\R^d\setminus\{\x_1=0\}}$ is a local diffeomorphism and in particular is open, $g(\Reg(\Ss^{(d)}))$ contains an open semialgebraic set $U\subset\{\x_1>0\}$ adherent to the origin such that $g(\alpha((0,1]))\subset U$. After substituting $\alpha$ by $g\circ\alpha$ and reparameterizing, we may assume $\alpha_1=\t^p$, each $\alpha_i$ is an algebraic series in the variable $\t$ and the order of $\alpha_1$ is smaller than or equal to the order of $\alpha_i$ for $i=2,\ldots,d$. The previous conditions hold because the $\alpha_i$ are algebraic Puiseux series at the origin and the first component of $g\circ\alpha$ is $\sqrt{\alpha_1^2+\cdots+\alpha_d^2}$ . By Lemma \ref{icsl} we may assume that $\alpha_i\in\R[\t]$ for $i=2,\ldots,d$. We substitute $\Ss\subset\R^d\setminus\{0\}$ by $g(\Ss)\subset\{\x_1>0\}$, which is a semialgebraic subset of $\R^d$ of dimension $d$ such that $0\in \cl(g(\Ss)^{(d)})\setminus g(\Ss)$.

Consider the Nash diffeomorphism 
\begin{align*} 
h:\{\x_1>0\}&\to\{\x_1>0\},\\
(x_1,x_2,\ldots,x_d)&\mapsto(\sqrt[p]{x_1},x_2-\alpha_2(\sqrt[p]{x_1}),\ldots,x_d-\alpha_d(\sqrt[p]{x_1}))
\end{align*} 
and observe that $h\circ\alpha=(t,0,\ldots,0)$. 

Let us fix an $\veps>0$ such that $(0,\veps]\times\{(0,\ldots,0)\}\subset h(U)$. Observe that $h(\Ss)\subset\{\x_1>0\}$ because $\Ss\subset\{\x_1>0\}$. Let 
$$
\delta:(0,\veps]\to(0,+\infty),\ t\mapsto\dist((t,0,\ldots,0),\R^d\setminus h(U))
$$ 
and let $\xi:(0,\veps]\to(0,+\infty)$ be a Nash function such that $|\frac{\delta}{2}-\xi|<\frac{\delta}{4}$, so $\frac{\delta}{4}<\xi<\frac{3\delta}{4}$. Write $x':=(x_2,\ldots,x_d)$ and consider the open semialgebraic set
$$
\Uu:=\{(x_1,x')\in(0,\veps)\times\R^{d-1}:\ \|x'\|^2<\xi^2(x_1)\}\subset h(\Ss).
$$
Observe that $\xi$ is a Puiseux series at the origin. The map 
$$
f_\ell:\{\x_1>0\}\to\{\x_1>0\},\ (x_1,x')\mapsto\Big(\frac{1}{x_1},\frac{x'}{x_1^\ell}\Big)
$$ 
is a Nash involution for each $\ell\geq1$. Fix two positive numbers $N_1,N_2>0$ and consider the semialgebraic set $\Ff:=\{N_1+N_2\|\x'\|^2\leq\x_1\}$. Observe that $(y_1,y')\in f_{\ell}(\Ff)$ if and only if $f_{\ell}(y_1,y')\in \Ff$, so 
\begin{equation*}
f_\ell(\Ff)=\Big\{\|\x'\|^2\leq\frac{1}{N_2}\x_1^{2\ell-1}(1-N_1\x_1)\Big\}\subset\Big\{(x_1,x')\in \Big[0,\frac{1}{N_1}\Big)\times\R^{d-1}\, :\, \|x'\|^2\leq \frac{1}{N_2}x_1^{2\ell-1}\Big\}.
\end{equation*}
If $N_1,N_2,\ell$ are large enough, $f_\ell(\Ff)\setminus\{0\}\subset\Uu\subset h(\Ss)$, so $\Ff\subset(f_\ell\circ h)(\Ss)$ (recall that $f_{\ell}$ is an involution). Consider the polynomial map
$$
P_1:\R^d\to\R^d,\ (x_1,x')\mapsto(x_1-(N_1+N_2\|x'\|^2),x'),
$$
which maps $\Ff$ onto $\{\x_1\geq0\}$. Let $P_2:\R^d\to\{\x_1\geq0\},\ (x_1,x')\mapsto(x_1^2,x')$. Observe that 
$$
\{\x_1\geq0\}=P_2(\{\x_1\geq0\})=(P_2\circ P_1)(\Ff)\subset(P_2\circ P_1\circ f_\ell\circ h)(\Ss)\subset P_2(\R^d)=\{\x_1\geq0\},
$$
so $(P_2\circ P_1\circ f_\ell\circ h)(\Ss)=\{\x_1\geq0\}$. Denote $x'':=(x_3,\ldots,x_d)$ and consider the polynomial map
$$
P_3:\R^d\to\R^d,\ (x_1,x_2,x'')\mapsto(x_1^2-x_2^2,2x_1x_2,x'')
$$
that maps $\{\x_1\geq0\}$ to $\R^d$ (here is the exact point where we use $d\geq 2$). Consequently, $(P_3\circ P_2\circ P_1\circ f_\ell\circ h)(\Ss)=\R^d,$ as required.
\end{proof}

The following example shows that Theorem \ref{nice2} is no longer true if $d=1$.

\begin{example}
Let $f:[0,+\infty)\to\R$ be a non-constant Nash function and consider its derivative $f':[0,+\infty)\to\R$. Observe that $\{f'=0\}$ is a finite set. Define $a:=\max\{f'=0\}$ and assume that $f'$ is strictly positive on $(a,+\infty)$, so $f$ is strictly increasing on $(a,+\infty)$. This means that $f([a,+\infty))=[f(a),b)$ for some $b\in\R\cup\{+\infty\}$. As $f([0,a])$ is a connected compact set, $f([0,a])=[c,d]$, so $f([0,+\infty))=[c,d]\cup[f(a),b)$, which is not an open interval. Thus, $f([0,+\infty))$ is a proper subset of $\R$.
\end{example}

\section{Proof of Theorem \ref{snmbss}}\label{s5}

As commented in the introduction, some obstructions to construct a surjective Nash map $f:\Ss\to\Tt$ between a semialgebraic set $\Ss\subset\R^m$ and $\Tt\subset\R^n$ concentrate on the configuration of the intersections of pairwise different analytic path-connected components $\{\Ss_i\}_{i=1}^r$ (resp. irreducible components $\{\Ss_j^*\}_{j=1}^\ell$) of $\Ss$ and the configuration of their images, which are semialgebraic subsets $\Tt_i:=f(\Ss_i)$ of $\Tt$ connected by analytic paths (resp. irreducible semialgebraic subsets $\Tt_j^*:=f(\Ss_j^*)$ of $\Tt$). Namely, if the intersection $\Ss_{i_1}\cap\cdots\cap\Ss_{i_k}$ (for $1\leq i_1<\cdots<i_k\leq r$) is non-empty, then
$$
f(\Ss_{i_1}\cap\cdots\cap\Ss_{i_k})\subset f(\Ss_{i_1})\cap\cdots\cap f(\Ss_{i_k})\subset\Tt_{i_1}\cap\cdots\cap\Tt_{i_k}
$$
and the analytic path-connected components of $\Ss_{i_1}\cap\cdots\cap\Ss_{i_k}$ are mapped into analytic path-connected components of $\Tt_{i_1}\cap\cdots\cap\Tt_{i_k}$. Analogously, if the intersection $\Ss_{j_1}^*\cap\cdots\cap\Ss_{j_p}^*$ (for $1\leq j_1<\cdots<j_p\leq \ell$) is non-empty, then
$$
f(\Ss_{j_1}^*\cap\cdots\cap\Ss_{j_p}^*)\subset f(\Ss_{j_1}^*)\cap\cdots\cap f(\Ss_{j_p}^*)\subset\Tt_{j_1}^*\cap\cdots\cap\Tt_{j_p}^*
$$
and the irreducible components of $\Ss_{j_1}^*\cap\cdots\cap\Ss_{j_p}^*$ are mapped into irreducible components of $\Tt_{j_1}^*\cap\cdots\cap\Tt_{j_p}^*$.

Let us see with some examples how the previous obstructions influence.

\begin{examples}
(i) Let $\Ss:=\{\z=0\}\cup\{\x=0,\y=0\}\cup\{\x-\z=0,\y=0\}\subset\R^3$ and $\Tt:=\{\z=0\}\cup\{\x=0,\y=0\}\cup\{\x=1,\y=0\}\subset\R^3$ (Figure \ref{stn}). We claim: {\em There exists no surjective Nash map $f:\Ss\to\Tt$.} 

The analytic path-connected components of $\Ss$ are 
$$
\Ss_1:=\{\z=0\},\ \Ss_2:=\{\x=0,\y=0\} \text{\ and\ } \Ss_3:=\{\x-\z=0,\y=0\},
$$ 
whereas the analytic path-connected components of $\Tt$ are $\Tt_1:=\{\z=0\}$, $\Tt_2:=\{\x=0,\y=0\}$ and $\Tt_3:=\{\x=1,\y=0\}$. Suppose there exists a surjective Nash map $f:\Ss\to\Tt$. Using straightforward dimensional arguments $f(\Ss_1)\subset\Tt_1$ and either $f(\Ss_2)\subset\Tt_2$ and $f(\Ss_3)\subset\Tt_3$ or $f(\Ss_2)\subset\Tt_3$ and $f(\Ss_3)\subset\Tt_2$. However, this is not possible because $\Ss_1\cap\Ss_2\cap\Ss_3=\{(0,0,0)\}$, whereas $\Tt_1\cap\Tt_2\cap\Tt_3=\varnothing$ and $f(\Ss_1\cap\Ss_2\cap\Ss_3)\subset\Tt_1\cap\Tt_2\cap\Tt_3$.

\begin{figure}[ht]
\begin{center}
\begin{tikzpicture}[scale=0.7]
%\draw[style=help lines,step=0.5cm] (0,0) grid (15,8);

\draw[fill=gray!60,opacity=0.75] (0,1.5) -- (1.5,3) -- (6.5,3) -- (5,1.5) -- (0,1.5);

\draw[fill=gray!60,opacity=0.75] (8,1.5) -- (9.5,3) -- (14.5,3) -- (13,1.5) -- (8,1.5);

\draw[line width=0.8pt,draw] (3.25,2.25) -- (3.25, 4.5);
\draw[line width=0.8pt,draw, densely dotted] (3.25,2.25) -- (3.25, 1.5);
\draw[line width=0.8pt,draw] (3.25,0) -- (3.25, 1.5);
\draw[line width=0.8pt,draw] (3.25,2.25) -- (5,4);
\draw[line width=0.8pt,draw, densely dotted] (3.25,2.25) -- (2.5, 1.5);
\draw[line width=0.8pt,draw] (1.5,0.5) -- (2.5, 1.5);

\draw[line width=0.8pt,draw] (11.25,2.25) -- (11.25, 4.5);
\draw[line width=0.8pt,draw, densely dotted] (11.25,2.25) -- (11.25, 1.5);
\draw[line width=0.8pt,draw] (11.25,0) -- (11.25, 1.5);

\draw[line width=0.8pt,draw] (12.25,2.25) -- (12.25, 4.5);
\draw[line width=0.8pt,draw, densely dotted] (12.25,2.25) -- (12.25, 1.5);
\draw[line width=0.8pt,draw] (12.25,0) -- (12.25, 1.5);

\draw (10.5,4) node{\small$\Tt_2$};
\draw (13,4) node{\small$\Tt_3$};
\draw (9.5,2) node{\small$\Tt_1$};
\draw (8.5,1) node{\small$\Tt$};

\draw (2.5,4) node{\small$\Ss_2$};
\draw (4,3.75) node{\small$\Ss_3$};
\draw (1.5,2) node{\small$\Ss_1$};
\draw (0.5,1) node{\small$\Ss$};

\end{tikzpicture}
\end{center}
\caption{\small{Semialgebraic sets $\Ss$ and $\Tt$.}\label{stn}}
\end{figure}

(ii) Let $\Ss:=\{\y\geq0\}\cup\{\x=0\}\subset\R^2$ and $\Tt:=\{\x^2-\z\y^2=0\}\subset\R^3$, which are both irreducible. We claim: {\em There exists no surjective Nash map $f:\Ss\to\Tt$.} 

The analytic path-connected components of $\Ss$ are $\Ss_1:=\{\y\geq0\}$ and $\Ss_2:=\{\x=0\}$, whereas the analytic path-connected components of $\Tt$ are 
$$
\Tt_1:=\{\x^2-\z\y^2=0,\z\geq0\} \text{\ and\ } \Tt_2:=\{\x=0,\y=0\}.
$$
Suppose there exists a surjective Nash map $f:\Ss\to\Tt$. Using straightforward arguments, $\Tt_1\setminus\Tt_2\subset f(\Ss_1)\subset\Tt_1$ and $\{\x=0,\y=0,\z<0\}\subset f(\Ss_2)\subset\Tt_2$. As $f$ is Nash, there exist a connected open semialgebraic neighborhood $U\subset\R^2$ and a Nash extension $F:U\to\R^3$. As $U$ is an open connected semialgebraic subset of $\R^2$, it is an irreducible semialgebraic set of dimension $2$. Thus, $F(U)$ is an irreducible semialgebraic subset of $\R^3$ of dimension $\leq 2$. In particular, its Zariski closure is an irreducible algebraic set of dimension $\leq 2$. As $f(\Ss_1)\subset\Tt_1$ has dimension $2$ and the Zariski closure of $\Tt_1$ is $\Tt$ (which is irreducible), we conclude that the Zariski closure of $F(U)$ is $\Tt$. As connected open semialgebraic sets are connected by analytic paths (because they are connected Nash manifolds), we deduce that $F(U)\subset\Tt_1$, so 
$$
\{\x=0,\y=0,\z<0\}\subset f(\Ss_2)\subset F(U)=\Tt_1,
$$
which is a contradiction.

(iii) However, there exists a surjective Nash map $f:\Tt\to\Ss$ where $\Tt:=\{\x^2-\z\y^2=0\}\subset\R^3$ and $\Ss:=\{\y\geq0\}\cup\{\x=0\}\subset\R^2$. It is enough to take $f(\x,\y,\z)=(\y,\z)$.
\end{examples}

Before proving Theorem \ref{snmbss}, we need the following preliminary result.

\begin{lem}\label{immersion2}
Let $\Ss\subset\R^m$ be a semialgebraic set and $\{\Ss_i^*\}_{i=1}^s$ the family of the irreducible components of $\Ss$ that are non-compact. Denote $d_i:=\dim(\Ss_i^*)$ and with $\Ss_i^{*,(d_i)}$ the set of points of $\Ss_i^*$ of dimension $d_i$, which we assume to be non-compact for each $i=1,\ldots,s$. Let $U$ be an open semialgebraic subset of $\R^m$ that contains $\Ss$ and let $X_1,\ldots,X_s$ be Nash subsets of $U$ such that $\Ss_i^*\setminus X_i\neq\varnothing$ for each $i$. Up to shrinking $U$ if necessary, there exist: 
\begin{itemize}
\item a Nash manifold $M\subset\R^p$, 
\item a Nash diffeomorphism $\varphi:M\to U$ and 
\item a Nash function $g_i:M\to\R$ whose zero set contains $\varphi^{-1}(X_i)$
\end{itemize}
such that the corresponding Nash map $G_i:M\to\R^{p+1},\ x\mapsto(x\cdot g_i(x),g_i(x))$ satisfies $0\in G_i(\varphi^{-1}(\Ss_i^*))=G_i(\varphi^{-1}(\Ss_i^{*,(d_i)}))$ and $G_i(\varphi^{-1}(\Ss_i^{*,(d_i)}))$ is pure dimensional of dimension $d_i$ and non-compact for $i=1,\ldots,s$.
\end{lem}
\begin{proof}
We may apply the Nash diffeomorphism 
$$
\psi_0:\R^m\to\Bb_m(0,1),\ x\mapsto\frac{x}{\sqrt{1+\|x\|^2}}
$$
to $\Ss$ and assume that $\Ss$ is bounded. As $\Ss_i^*\setminus X_i\neq\varnothing$, $\Ss_i^*$ is irreducible and $X_i$ is the zero-set of a Nash function on $U$, we deduce by \cite[Lem.3.6]{fg3} that $\dim(\Ss_i^*\cap X_i)<\dim(\Ss_i^*)=\dim(\Ss_i^{*,(d_i)})$ for each $i=1,\ldots,r$. Pick a point $q_i\in\Ss_i^{*,(d_i)}$. Let $Z_i$ be the Zariski closure of $(\cl(\Ss_i^*)\cap\cl(X_i))\cup\cl(\Ss_i^*\setminus\Ss_i^{*,(d_i)})\cup\{q_i\}$. We claim: {\em $Z_i$ has dimension strictly smaller than $\dim(\Ss_i^*)$.} 

As $\Ss_i^*$ is closed in $\Ss$, $X_i$ is closed in $U$ and $\Ss\subset U$, we deduce 
\begin{align*}
&\cl(\Ss_i^*)\cap\cl(X_i)\cap\Ss=\Ss_i^*\cap X_i,\\
&\cl(\Ss_i^*)\cap\cl(X_i)\cap(\cl(\Ss)\setminus\Ss)=\cl(\Ss_i^*)\cap\cl(X_i)\setminus\Ss_i^*\subset\cl(\Ss_i^*)\setminus\Ss_i^*.
\end{align*}
Both semialgebraic sets have dimensions strictly smaller than $\dim(\Ss_i^*)$, so $\cl(\Ss_i^*)\cap\cl(X_i^*)$ has dimension strictly smaller than $\dim(\Ss_i^*)$. In addition, $\cl(\Ss_i^*\setminus\Ss_i^{*,(d_i)})$ has dimension strictly smaller than $\dim(\Ss^*_i)$, because $\dim(\Ss_i^*\setminus\Ss_i^{*,(d_i)})<d_i=\dim(\Ss_i^*)$. Thus, $Z_i$ is a real algebraic set of dimension strictly smaller than $\dim(\Ss_i^*)$.

As $\Ss_i^{*,(d_i)}$ is bounded and non-compact and $\Ss_i^{*,(d_i)}$ is closed in $\Ss$ (because it is a closed subset of $\Ss_i^*$, which is a closed subset of $\Ss$), there exists $p_i\in\cl(\Ss_i^{*,(d_i)})\setminus\Ss$ (because otherwise $\cl(\Ss_i^{*,(d_i)})\subset\Ss$ and $\Ss_i^{*,(d_i)}=\cl(\Ss_i^{*,(d_i)})\cap\Ss=\cl(\Ss_i^{*,(d_i)})$ would be compact). As $p_i\not\in\Ss$, up to replace $U$ by $U':=U\setminus\{p_1,\ldots,p_r\}$ and $X_i$ by $U'\cap X_i$ if necessary, we may assume $p_i\not\in X_i$. As $\Ss_i^{*,(d_i)}\setminus Z_i$ is dense in $\Ss_i^{*,(d_i)}$ (because $\Ss_i^{*,(d_i)}$ is pure dimensional and $Z_i$ has strictly smaller dimension), there exists by the Nash curve selection lemma \cite[8.1.13]{bcr} a Nash curve $\alpha_i:(-1,1)\to\R^m$ such that $\alpha_i((0,1))\subset\Ss_i^{*,(d_i)}\setminus Z_i$ and $\alpha_i(0)=p_i$. Let $Q_i\in\R[\x_1,\ldots,\x_n]$ be a polynomial whose zero set is $Z_i$. 

\noindent{\sc Case 1.} If $Q_i(p_i)\neq0$, we take a bounded Nash function $g_i$ on $U$ whose zero set is the union of $X_i$ and the smallest Nash subset of $U$ that contains $\cl(\Ss_i^*\setminus\Ss_i^{*,(d_i)})$. Observe that the limit $\lim_{t\to0^+}g_i\circ\alpha_i(t)$ exists and it is non-zero, because otherwise either $p_i$ belongs to the Zariski closure of $\cl(\Ss_i^*\setminus\Ss_i^{*,(d_i)})\subset Z_i=\{Q_i=0\}$ or $p_i\in\cl(\Ss_i^{*,(d_i)})\cap\cl(X_i)\subset Z_i=\{Q_i=0\}$, which is a contradiction.

Consider the Nash map $G_i:U\to\R^{m+1},\ x\mapsto(x\cdot g_i(x),g_i(x))$, whose restriction to $U\setminus\{g_i=0\}$ is a Nash diffeomorphism between $U\setminus\{g_i=0\}$ and $G_i(U)\setminus\{0\}$, whose inverse is $H_i:G_i(U)\setminus\{0\}\to U\setminus\{g_i=0\},\ (y,t)\mapsto\frac{y}{t}$. If $G_i(\Ss_i^{*,(d_i)})$ is compact, then $\lim_{t\to0^+}(\alpha_i(t)\cdot(g_i\circ\alpha_i)(t),g_i\circ\alpha_i(t))\in G_i(\Ss_i^{*,(d_i)})$. As
$$
G_i|_{\Ss_i^{*,(d_i)}\setminus\{g_i=0\}}:\Ss_i^{*,(d_i)}\setminus\{g_i=0\}\to G_i(\Ss_i^{*,(d_i)})\setminus\{0\}
$$
is a Nash diffeomorphism, we conclude that $\lim_{t\to0^+}\alpha_i(t)\cdot (g_i\circ\alpha_i)(t)=0$ (because $p_i\not\in\Ss_i^{*,(d_i)}$), which is a contradiction because $\lim_{t\to0^+}(g_i\circ\alpha_i)(t)$ exists and it is non-zero. Consequently, $G_i(\Ss_i^{*,(d_i)})$ is non-compact. Again, as the restriction $G_i|_{\Ss_i^{*,(d_i)}\setminus\{g_i=0\}}$ is a Nash diffeomorphism, 
$$
G_i(\Ss_i^{*,(d_i)})\setminus\{0\}=G_i(\Ss_i^{*,(d_i)}\setminus\{g_i=0\})
$$ 
is pure dimensional of dimension $d_i$. As $\Ss_i^{*,(d_i)}$ is pure dimensional of dimension $d_i$ and $\dim(\{g_i=0\})\leq d_i-1$, we deduce $q_i\in\cl(\Ss_i^{*,(d_i)}\setminus\{g_i=0\})$. As $q_i\in\Ss_i^{*,(d_i)}\cap\{g_i=0\}$, we conclude $0\in G_i(\cl(\Ss_i^{*,(d_i)}\setminus\{g_i=0\}))\subset\cl(G_i(\Ss_i^{*,(d_i)}\setminus\{g_i=0\}))$, so $G_i(\Ss_i^{*,(d_i)})$ is pure dimensional of dimension $d_i$. In addition, 
$$
0\in G_i(\Ss_i^{*,(d_i)})=G_i(\Ss_i^{*,(d_i)})\cup\{0\}=G_i(\Ss_i^{*,(d_i)}\cup(\Ss_i^*\cap\{g_i=0\}))=G_i(\Ss_i^*).
$$

\noindent{\sc Case 2.} If $Q_i(p_i)=0$, we have $Q_i\circ\alpha_i\in\R[[\t]]_{\rm alg}$ is a non-zero series. Let $(Y_i,\phi_i)$ be the blow-up of $\R^m$ at $p_i$. The restriction $\phi_i:Y_i\setminus\{\phi_i^{-1}(p_i)\}\to\R^m\setminus\{p_i\}$ is a Nash diffeomorphism and $p_i\not\in\Ss\cup X_i$, so $\phi_i^{-1}(\Ss)$ is Nash diffeomorphic to $\Ss$ and $\phi_i^{-1}(X_i)$ is Nash diffeomorphic to $X_i$. The series $Q_i\circ\alpha_i=(Q_i\circ\phi_i)\circ(\phi_i^{-1}\circ\alpha_i)$ and let $(Q_i\circ\phi_i)^*$ be the strict transform of $(Q_i\circ\phi_i)$. The order of the series $(Q_i\circ\phi_i)^*\circ(\phi_i^{-1}\circ\alpha_i)$ is strictly smaller than the order of $Q_i\circ\alpha_i$, because we have eliminated from $(Q_i\circ\phi_i)$ a power of an equation of the exceptional divisor. Let $p_i':=\lim_{t\to0^+}(\phi_i^{-1}\circ\alpha_i)(t)$. If $(Q_i\circ\phi_i)^*(p_i')\neq0$, we have finished with this index $i$. Otherwise, we repeat the previous process with the point $p_i'$. In each step the order of the strict transform of the corresponding polynomial substituted in the corresponding curve has strictly smaller order, so in finitely many steps we achieve order $0$ and the corresponding polynomial does not vanish at the limit point. This allows, after finitely many steps, to reduce {\sc Case 2} to {\sc Case 1}.

After composing all the involved blow-ups (corresponding to all the indices $i=1,\ldots,s$ that are under the assumptions of {\sc Case 2}) and taking the corresponding strict transforms, we find a Nash manifold $M\subset\R^p$, a Nash diffeomorphism $\varphi:M\to U$ and Nash functions $g_i:M\to\R$ such that $\varphi^{-1}(X_i)\subset\{g_i=0\}$ and the corresponding Nash map 
$$
G_i:M\to\R^{p+1},\ x\mapsto(x\cdot g_i(x),g_i(x))
$$ 
satisfies $0\in G_i(\varphi^{-1}(\Ss_i^*))=G_i(\varphi^{-1}(\Ss_i^{*,(d_i)}))$ and $G_i(\varphi^{-1}(\Ss_i^{*,(d_i)}))$ is pure dimensional of dimension $d_i$ and non-compact for $i=1,\ldots,s$, as required.
\end{proof}

We are ready to prove Theorem \ref{snmbss}.

\begin{proof}[Proof of Theorem \em\ref{snmbss}]
The only if conditions are obtained straightforwardly. The proof of the converse is conducted in several steps:

\noindent{\sc Step 1.} Suppose $\Ss_i^*$ is non-compact for $i=1,\ldots,s$ and $\Ss_i^*$ is compact for $i=s+1,\ldots,r$. For each $i=1,\ldots,r$ let $f_i:\Ss\to\R$ be a Nash function on $\Ss$ such that $\Ss_i^*=\{f_i=0\}$ (see \cite[Lem.2.4, Thm. 4.3]{fg3}). Let $U$ be an open semialgebraic neighborhood of $\Ss$ in $\R^m$ to which all the Nash functions $f_i$ extend as Nash functions $F_i:U\to\R$. Define $X_i:=\bigcup_{j\neq i}\{F_j=0\}$ and observe that $\Ss_i^*\cap X_i=\Ss_i^*\cap\bigcup_{j\neq i}\Ss_j^*$ is a semialgebraic subset of $\Ss_i^*$ of dimension strictly smaller than $d_i$. We distinguish two cases:

\vspace{1mm}
\noindent{\sc Case 1. Non-compact irreducible components.} By Lemma \ref{immersion2} we may assume (up to a suitable Nash diffeomorphism) that for each $i=1,\ldots,s$ there exist a Nash function $g_i:U\to\R$ whose respective zero set $\{g_i=0\}$ contains $X_i$ and the corresponding Nash map 
$$
G_i:U\to\R^{m+1},\ x\mapsto(x\cdot g_i(x),g_i(x))
$$ 
satisfies $0\in G_i(\Ss_i^*)=G_i(\Ss_i^{*,(d_i)})$ and $G_i(\Ss_i^{*,(d_i)})$ is non-compact and pure dimensional of dimension $d_i$ for $i=1,\ldots,s$. In addition, 
$$
G_i(\Ss)=G_i(\Ss_i^*\cup\bigcup_{j\neq i}\Ss_j^*)=G_i(\Ss_i^*)\cup\bigcup_{j\neq i}G_i(\Ss_j^*)=G_i(\Ss_i^*)\cup\{0\}=G_i(\Ss_i^{*,(d_i)}).
$$

\noindent{\sc Case 2. Compact irreducible components.} For each $i=s+1,\ldots,r$ let $q_i\in\Ss_i^{*,(d_i)}$ and $h_i$ be a polynomial whose zero set is the union of $\{q_i\}$ and the Zariski closure $Y_i$ of $\cl(\Ss_i^*\setminus\Ss_i^{*,(d_i)})$. Define $g_i:=h_i\prod_{j\neq i}F_j:U\to\R$ and observe that $\{g_i=0\}=\{q_i\}\cup(Y_i\cap\Ss_i^*)\cup\bigcup_{j\neq i}\Ss_j^*$. As $\Ss_i^*$ is irreducible and $g_i$ does not vanish identically on $\Ss_i^*$, the intersection $\{g_i=0\}\cap\Ss_i^*$ has dimension $<d_i:=\dim(\Ss_i^*)$. Consider the Nash map
$$
G_i:U\to\R^{m+1},\ x\mapsto (x\cdot g_i(x),g_i(x)),
$$
whose restriction to $U\setminus\{g_i=0\}$ is a Nash diffeomorphism between $U\setminus\{g_i=0\}$ and $G_i(U)\setminus\{0\}$. Observe that $G_i(\Ss_j^*)=\{0\}$ if $i\neq j$ and let us check: {\em$\Ss_i':=G_i(\Ss)=G_i(\Ss_i^{*,(d_i)})$ is pure dimensional of dimension $d_i$.}

As $G_i|_{\Ss_i^{*,(d_i)}\setminus\{g_i=0\}}:\Ss_i^{*,(d_i)}\setminus\{g_i=0\}\to G_i(\Ss_i^{*,(d_i)})\setminus\{0\}$ is a Nash diffeomorphism, $G_i(\Ss_i^{*,(d_i)})\setminus\{0\}=G_i(\Ss_i^{*,(d_i)}\setminus\{g_i=0\})$ is pure dimensional of dimension $d_i$. As $\Ss_i^{*,(d_i)}$ is pure dimensional of dimension $d_i$ and $\dim(\{g_i=0\})\leq d_i-1$, we deduce $q_i\in\cl(\Ss_i^{*,(d_i)}\setminus\{g_i=0\})$. As $q_i\in\Ss_i^{*,(d_i)}\cap\{g_i=0\}$, we conclude $0\in G_i(\cl(\Ss_i^{*,(d_i)}\setminus\{g_i=0\}))\subset\cl(G_i(\Ss_i^{*,(d_i)}\setminus\{g_i=0\}))$, so $G_i(\Ss_i^{*,(d_i)})$ is pure dimensional of dimension $d_i$. In addition, 
$$
0\in G(\Ss_i^{*,(d_i)})=G(\Ss_i^{*,(d_i)})\cup\{0\}=G(\Ss_i^{*,(d_i)}\cup(\Ss_i^*\cap\{g_i=0\}))=G(\Ss_i^*).
$$
Moreover, 
$$
G_i(\Ss)=G_i(\Ss_i^*\cup\bigcup_{j\neq i}\Ss_j^*)=G_i(\Ss_i^*)\cup\bigcup_{j\neq i}G_i(\Ss_j^*)=G_i(\Ss_i^*)\cup\{0\}=G_i(\Ss_i^{*,(d_i)})
$$
for $i=s+1,\ldots,r$.

\vspace{1mm}
\noindent{\sc Step 2.} Define $\Ss_i':=G_i(\Ss)$ for $i=1,\ldots,r$ and 
$$
G:\Ss\to\R^{(m+1)r},\ x\mapsto(G_1(x),\ldots,G_r(x)).
$$ 
Observe that
$$
G(\Ss_i^*)=\{0\}\times\cdots\times\{0\}\times\overset{(i)}{\Ss_i'}\times\{0\}\times\cdots\times\{0\}
$$
and $G(\Ss)=\bigcup_{i=1}^sG(\Ss_i^*)$. In addition, $G(\Ss_i^*)\cap G(\Ss_j^*)=\{(0,\ldots,0)\}$ if $i\neq j$.

We distinguish two cases:

\vspace{1mm}
\noindent{\sc Case 1.} If $\Ss_i^*$ is non-compact, $\Ss_i'$ is non-compact. By Theorem \ref{nice2} there exists a Nash map $H_i:\R^{m+1}\to\R^{d_i}$ such that $H_i(\Ss_i')=\R^{d_i}$. We may assume in addition $H_i(0)=0$.

\vspace{1mm}
\noindent{\sc Case 2.} If $\Ss_i$ is compact, also $\Ss_i'$ is compact and there exists by Theorem \ref{nice0} a Nash map $H_i:\R^{m+1}\to\R^{d_i}$ such that $H_i(\Ss_i')=\ol{\Bb}_{d_i}$. Following the proof of Theorem \ref{nice0}, the reader can check that we may assume $H_i(0)=0$.

\vspace{1mm}
\noindent{\sc Step 3.} Let $q\in\bigcap_{i=1}^r\Tt_i$ and assume $q$ is the origin of $\R^n$. Let $F_i:\R^{d_i}\to\R^n$ be a Nash map such that $F_i(\R^{d_i})=\Tt_i$ for $i=1,\ldots,s$ and $F_i(\ol{\Bb}_{d_i})=\Tt_i$ for $i=s+1,\ldots,r$. Define $E_i=\R^{d_i}$ if $\Ss_i'$ is non-compact ($i=1,\ldots,s$) and $E_i=\ol{\Bb}_{d_i}$ if $\Ss_i'$ is compact ($i=s+1,\ldots,r$). 

We may assume in addition $F_i(0)=0$ for $i=1,\ldots,r$. We have
$$
(F_i\circ H_i\circ G_i)(\Ss_j^*)=\begin{cases}
F_i(E_i)=\Tt_i&\text{if $j=i$},\\
F_i(\{0\})=\{0\}&\text{if $j\neq i$,}
\end{cases}
$$
so $(F_i\circ H_i\circ G_i)(\Ss)=\Tt_i$. Observe that
\begin{equation*}
\begin{split}
((F_1\circ H_1,\ldots,F_r\circ H_r)&\circ G)(\Ss_i^*)=(F_1\circ H_1\circ G_1,\ldots,F_r\circ H_r\circ G_r)(\Ss_i^*)\\
&=(F_1\circ H_1)(\{0\})\times\cdots\times(F_{i-1}\circ H_{i-1})(\{0\})\times(F_i\circ H_i)(\Ss_i')\\
&\hspace{3.842cm}\times(F_{i+1}\circ H_{i+1})(\{0\})\times\cdots\times(F_r\circ H_r)(\{0\})\\
&=F_1(\{0\})\times\cdots\times F_{i-1}(\{0\})\times F_i(E_i)\times F_{i+1}(\{0\})\times\cdots\times F_r(\{0\})\\
&=\{0\}\times\cdots\times\{0\}\times\overset{(i)}{\Tt_i}\times\{0\}\times\cdots\times\{0\}.
\end{split}
\end{equation*}
Thus, if
$$
F:=\sum_{i=1}^r(F_i\circ H_i\circ G_i):\Ss\to\Tt,
$$
we have $F(\Ss_i^*)=\Tt_i$ for $i=1,\ldots,r$, so 
$$
F(\Ss)=F\Big(\bigcup_{i=1}^r\Ss_i^*\Big)=\bigcup_{i=1}^rF(\Ss_i^*)=\bigcup_{i=1}^r\Tt_i=\Tt,
$$
as required.
\end{proof}

Recall that the analytic path-connected components of $\Ss$ are irreducible semialgebraic sets. Thus, each of them is contained in an irreducible component of $\Ss$. If $\Ss_i^*$ is the irreducible component of $\Ss$ that contains $\Ss_i$ for $i=1,\ldots,r$, it may happen that $\Ss_i^*=\Ss_j^*$ for some $i\neq j$ or $\Ss_i^*\neq\Ss_j^*$, whereas $\Ss_i\subsetneq\Ss_i^*$ and $\Ss_j\subsetneq\Ss_j^*$.

\begin{examples}
(i) Define $\Ss:=\Ss_1\cup\Ss_2\cup\Ss_3\subset\R^2$ where $\Ss_1:=\{\x\geq1\}$, $\Ss_2:=\{\y=0\}$ and $\Ss_3:=\{\x\leq-1\}$. Observe that $\Ss_1$, $\Ss_2$ and $\Ss_3$ are the analytic path-connected components of $\Ss$, whereas $\Ss$ is irreducible. Thus, $\Ss_1^*=\Ss_2^*=\Ss_3^*$.

(ii) Define $\Ss:=\Ss_1\cup\Ss_2\subset\R^3$ where $\Ss_1:=\{\x=0,\y\geq0\}$ and $\Ss_2:=\{\y\leq0,\z=0\}$. Observe that $\Ss_1$ and $\Ss_2$ are the analytic path-connected components of $\Ss$, whereas $\Ss_1^*=\Ss_1\cup\{\x=0,\z=0\}$ and $\Ss_2^*=\Ss_2\cup\{\x=0,\z=0\}$ are the irreducible components of $\Ss$.
\end{examples}

\begin{remarks}\label{snmbssr}
(i) Let $\Ss\subset\R^m$ be a semialgebraic set and $\{\Ss_i\}_{i=1}^r$ the analytic path-connected components of $\Ss$. Let $\Ss_i^*$ be the irreducible component of $\Ss$ that contains $\Ss_i$ for $i=1,\ldots,r$ and assume $\Ss_i^*\neq\Ss_j^*$ for $1\leq i<j\leq r$. Denote $d_i:=\dim(\Ss_i^*)$. We claim:
\begin{itemize}
\item[(1)] $\{\Ss_i^*\}_{i=1}^r$ is the collection of the irreducible components of $\Ss$.
\item[(2)] $\Ss_i^{*,(d_i)}=\Ss_i$ for $i=1,\ldots,r$.
\end{itemize}

As $\Ss=\bigcup_{i=1}^r\Ss_i\subset\bigcup_{i=1}^r\Ss_i^*\subset\Ss$, we deduce that $\{\Ss_i^*\}_{i=1}^r$ is the collection of the irreducible components of $\Ss$, because $\Ss_i^*\neq\Ss_j^*$ if $i\neq j$. Thus, (1) holds.

Let us check (2). To that end, we prove first: $\dim(\Ss_j\cap\Ss_i^*)<\dim(\Ss_i^*)$ if $j\neq i$.

Otherwise, there exists $\Ss_j$ with $j\neq i$ such that $\dim(\Ss_j\cap\Ss_i^*)=\dim(\Ss_i^*)$, so $\dim(\Ss_j^*\cap\Ss_i^*)=\dim(\Ss_i^*)$ and each Nash function that vanishes identically on $\Ss_j^*$ vanishes also identically on $\Ss_i^*$. Thus, $\Ss_i^*\subset\Ss_j^*$ and $i=j$, which is a contradiction. 

Consequently, 
$$
\Ss_i\setminus\bigcup_{j\neq i}\Ss_j\subset\Ss_i^*\setminus\bigcup_{j\neq i}\Ss_j\subset\Ss\setminus\bigcup_{j\neq i}\Ss_j=\Ss_i\setminus\bigcup_{j\neq i}\Ss_j,
$$ 
so $\Ss_i\setminus\bigcup_{j\neq i}\Ss_j=\Ss_i^*\setminus\bigcup_{j\neq i}\Ss_j$ is non-empty and has dimension $d_i$. As $\Ss_i$ is pure dimensional of dimension $d_i$ and $\bigcup_{j\neq i}\Ss_i^*\cap\Ss_j$ has dimension $<d_i$, we deduce that $\Ss_i\setminus\bigcup_{j\neq i}\Ss_j$ is dense in $\Ss_i$. In addition, $\Ss_i\subset\Ss_i^{*,(d_i)}\subset\Ss_i^*$ (because $\Ss_i$ is pure dimensional of dimension $d_i$), so $\Ss_i\setminus\bigcup_{j\neq i}\Ss_j=\Ss_i^{*,(d_i)}\setminus\bigcup_{j\neq i}\Ss_j$, which is dense in $\Ss_i^{*,(d_i)}$. As 
$$
\Ss_i^{*,(d_i)}\setminus\bigcup_{j\neq i}\Ss_j\subset\Ss_i^*\setminus\bigcup_{j\neq i}\Ss_j=\Ss_i\setminus\bigcup_{j\neq i}\Ss_j,
$$
we conclude taking closures in $\Ss$ that $\Ss_i^{*,(d_i)}=\Ss_i$ (because both $\Ss_i^{*,(d_i)}$ and $\Ss_i$ are closed in $\Ss$).

(ii) Observe that Theorems \ref{nice0} and \ref{nice2} are particular cases of Theorem \ref{snmbss} when $\Tt$ is connected by analytic paths.
\end{remarks}

As a straightforward consequence of Theorem \ref{snmbss} and Remark \ref{snmbssr}(i), we have the following:

\begin{cor}
Let $\Ss\subset\R^m$ and $\Tt\subset\R^n$ be semialgebraic sets, let $\{\Ss_i\}_{i=1}^r$ be the family of analytic path-connected components of $\Ss$ and let $\Ss_i^*$ be the irreducible component of $\Ss$ that contains $\Ss_i$ for $i=1,\ldots,r$. Assume $\Ss_i^*\neq\Ss_j^*$ for $1\leq i<j\leq r$. Let $\{\Tt_i\}_{i=1}^r$ be a family of (non-necessarily distinct) semialgebraic subsets of $\Tt$ connected by analytic paths and assume $\bigcap_{i=1}^r\Tt_i\neq\varnothing$ and $\bigcup_{i=1}^r\Tt_i=\Tt$. Then there exists a surjective Nash map $f:\Ss\to\Tt$ such that $f(\Ss_i)=\Tt_i$ for $i=1,\ldots,r$ if and only if $e_i:=\dim(\Tt_i)\leq\dim(\Ss_i)=:d_i$ and $\Tt_i$ is compact in case $\Ss_i$ is compact for $i=1,\ldots,r$.
\end{cor}

\section{Two applications of the main results}\label{s6}

In this section we present two remarkable consequences of Theorem \ref{main1}. The first one about representation of pure dimensional compact irreducible arc-symmetric semialgebraic sets as Nash images of closed balls. As a second consequence we show that a compact semialgebraic set is the projection of a non-singular compact algebraic set with the simplest possible topology (a disjoint union of spheres). 

\subsection{Representation of arc-symmetric compact semialgebraic sets.} \label{arcsym1}
It follows from Theorem \ref{main1} and \cite[Cor.2.8]{k} that a pure dimensional compact irreducible arc-symmetric semialgebraic set is a Nash image of $\ol{\Bb}_d$ where $d:=\dim(\Ss)$. 

\begin{proof}[Proof of Corollary \em \ref{consq1}]
Let $X$ be the Zariski closure of $\Ss$ and $\pi:\widetilde{X}\to X$ a resolution of the singularities of $X$ (see \cite{hi}). Assume $\widetilde{X}\subset\R^p$ and $\pi$ is the restriction to $\widetilde{X}$ of a polynomial map $\Pi:\R^p\to\R^n$. By \cite[Thm.2.6]{k} applied to the irreducible arc-symmetric set $\Ss$ there exists a connected component $E$ of $\widetilde{X}$ such that $\pi(E)=\cl(\Reg(\Ss))=\Ss$ (recall that $\Ss$ is pure dimensional and compact). As $\pi$ is proper and $\Ss$ is compact, also $E$ is compact (because it is a closed subset of the compact set $\pi^{-1}(\Ss)$). Thus, $E$ is a connected compact Nash manifold. By Theorem \ref{main1} there exists a Nash map $f_0:\R^{d}\to\R^p$ such that $f_0(\ol{\Bb}_d)=E$. Consequently, the Nash map $f:=\pi\circ f_0:\R^{d}\to\R^n$ satisfies $f(\ol{\Bb}_d)=\pi(f_0(\ol{\Bb}_d))=\pi(E)=\Ss$, as required.
\end{proof}

\subsection{Elimination of inequalities.}\label{Motzkin}
To prove Corollary \ref{consq2} we recall first the following well-known separation result, that we include here for the sake of completeness.

\begin{lem}[Separation]\label{separation}
Let $\Ss_1,\Ss_2\subset\R^n$ be semialgebraic sets such that $\Ss_1$ is compact, $\Ss_2$ is closed and $\Ss_1\cap\Ss_2=\varnothing$. Then there exists $f\in\R[\x]$ such that $\Ss_1\subset\{f<0\}$ and $\Ss_2\subset\{f>0\}$.
\end{lem}
\begin{proof}
We may assume $\Ss_1\subset\Bb_n(0,\frac{1}{2})$. Let $g:\R^n\to\R$ be a continuous function such that $\Ss_1\subset\{g<0\}$ and $\Ss_2\subset\{g>0\}$. Let 
$$
\veps:=\dist(\Ss_1,\Ss_2):=\min\{\dist(x_1,x_2):\ x_1\in\Ss_1,x_2\in\Ss_2\}>0.
$$
By Weierstrass' approximation theorem there exists a polynomial $f_0\in\R[\x]$ such that
$$
\max\{|g(x)-f_0(x)|:\ x\in\ol{\Bb}_n(0,1)\}<\frac{\veps}{3}.
$$
By \cite[Prop.2.6.2]{bcr} there exists a constant $c>0$ and $m\geq1$ such that $|f_0(x)|<c(1+\|x\|^2)^m$ on $\R^n$. Thus, $|f_0(x)|<2^mc\|x\|^{2m}$ on $\R^n\setminus\ol{\Bb}_n(0,1)$. Denote $c':=2^mc$ and let $k\geq m$ be such that 
$\frac{c'}{2^{2k}}<\frac{\veps}{3}$. Define $f:=f_0+c'\|\x\|^{2k}\in\R[\x]$. The reader can check that $\Ss_1\subset\{f<0\}$ and $\Ss_2\subset\{f>0\}$, as required.
\end{proof}

\begin{proof}[Proof of Corollary \em \ref{consq2}]
(i) By Theorem \ref{main1} and using the fact that the closed unit ball $\ol{\Bb}_d$ is the projection of the sphere $\sph^d$, there exists a Nash map $f:\R^{d+1}\to\R^n$ such that $f(\sph^d)=\Ss$. By Artin-Mazur's description of Nash maps \cite[Thm.8.4.4]{bcr} there exist $s\geq1$ and a non-singular irreducible algebraic set $Z\subset\R^{d+1+n+s}$ of dimension $d$, a connected component $M$ of $Z$ and a Nash diffeomorphism $g:\sph^d\to M$ such that the following diagram is commutative.
$$
\xymatrix{
M_\epsilon\ar@{^{(}->}[r]\ar@<-0.5ex>[ddr]_{\pi|_{M_\epsilon}}&Y=M_+\sqcup M_-\ar@{^{(}->}[r]\ar[dd]^{\pi|_Y}&Z\times\R\ar@{^{(}->}[r]\ar[d]^{\pi|_{Z\times\R}}&\R^{m+1}\ar[d]^{\pi}\\
&&Z\ar@{^{(}->}[r]&\R^{d+1}\times\R^n\times\R^s\equiv\R^m\ar[dd]^{\pi_2}\ar[ddll]_{\pi_1}\\
&M\ar@{^{(}->}[ur]\ar@<-0.5ex>[uul]_{\varphi_\epsilon}&&\\
&\sph^d\ar@{<->}[u]_{\cong}^g\ar[rr]^{f}\ar[rd]^{f}&&\R^n\\
&&\Ss\ar@{^{(}->}[ru]&
}
$$
We denote the projection of $\R^{d+1}\times\R^n\times\R^s$ onto the first space $\R^{d+1}$ with $\pi_1$ and the projection of $\R^{d+1}\times\R^n\times\R^s$ onto the second space $\R^n$ with $\pi_2$. Write $m:=d+1+n+s$. As $M$ is compact, there exists by Lemma \ref{separation} a polynomial $f:\R^m\to\R$ such that $M=Z\cap\{f>0\}$. Observe that $M$ is the projection of the algebraic set
$$
Y:=\{(z,t)\in Z\times\R:\ f(z)t^2-1=0\}
$$
under the projection $\pi:\R^m\times\R\to\R^m,\ (z,t)\mapsto z$. Fix $\epsilon=\pm1$ and let $M_\epsilon:=Y\cap\{\epsilon t>0\}$. Consider the Nash diffeomorphism 
$$
\varphi_\epsilon:M\to M_\epsilon,\ x\mapsto\Big(x,\epsilon\frac{1}{\sqrt{f(x)}}\Big)
$$
whose inverse map is the restriction of the projection $\pi$ to $M_\epsilon$.

Observe that $\{M_\epsilon\}_{\epsilon\in\{-1,1\}}$ is the collection of the connected components of $Y$. As $\pi(M_\epsilon)=M$ and using the diagram above, we deduce
$$
(\pi_2\circ\pi)(M_{\epsilon})=\pi_2(M)=(f\circ\pi_1)(M)=f(\sph^d)=\Ss.
$$
In addition, each $M_\epsilon$ is Nash diffeomorphic to $\sph^d$ and for $\veps\neq\veps'$ the polynomial map 
$$
\phi:\R^m\times\R\to\R^m\times\R,\ (x,t)\mapsto(x,-t)
$$
induces an involution of $Y$ such that $\phi(M_\epsilon)=M_{\epsilon'}$. As $Z$ is non-singular, also $Y$ is non-singular. Let $X$ be the irreducible component of $Y$ that contains $M_{+1}$. Observe that either $X=M_{+1}$ or $X=Y$.

Then $k:=d+s+2$ and the non-singular algebraic set $X$ satisfies the requirements in the statement.

In addition, $X$ has at most two connected components and each of them is Nash diffeomorphic to $\sph^d$. Thus, $X$ is Nash diffeomorphic to $\sph^d\times\{1,s\}$, where $s=1,2$ is the number of connected components of $X$.

(ii) Let $\Ss_1,\ldots,\Ss_r$ be the (compact) analytic path-connected components of $\Ss$, which satisfy $\Ss=\bigcup_{i=1}^r\Ss_i$. By (i) there exist $m\geq1$ and for each $i=1,\ldots,r$ a non-singular algebraic set $X_i\subset\R^m$ that is Nash diffeomorphic to a disjoint union of at most two spheres of $\R^{d+1}$ (each of them isometric to $\sph^{d_i}$ where $d_i:=\dim(\Ss_i)\leq d=\dim(\Ss)$) and satisfies $\pi(X_i)=\Ss_i$, where 
$$
\pi:\R^n\times\R^{m-n}\to\R^n,\ (x,y)\mapsto x
$$
is the projection onto the first $n$ coordinates. Consider the pairwise disjoint union $X:=\bigsqcup_{i=1}^rX_i\times\{i\}\subset\R^{m+1}$ and the projection 
$$
\pi':\R^n\times\R^{m+1-n}\times\R\to\R^n,\ (x,y,t)\mapsto x.
$$
Then $X$ is a non-singular algebraic set, which is Nash diffeomorphic to a finite pairwise disjoint union of spheres of dimension $\leq d$ and satisfies $\pi(X)=\Ss$, as required.
\end{proof}

The following lemma together with Example \ref{nint2} shows that Corollary \ref{consq2} is sharp.

\begin{lem}\label{nint}
Let $Z\subset\R^m$ be a non-singular irreducible algebraic set and $M$ one of its connected components of maximal dimension $d$. Suppose there exists an irreducible algebraic set $Y\subset\R^p$ of dimension $d$ and a rational map $\varphi:\R^p\dashrightarrow\R^m$ such that $\varphi|_Y:Y\to M$ is bijective. Then $M$ is the unique connected component of $Z$ of dimension $d$.
\end{lem}
\begin{proof}
Let $\widetilde{Y}\subset\C^p$ be the complexification of $Y$ and $\widetilde{Z}\subset\C$ the complexification of $Z$. Observe that $\widetilde{Y},\widetilde{Z}$ are irreducible algebraic sets of (complex) dimension $d$. Consider the rational map $\widetilde{\varphi}:\C^p\dashrightarrow\C^m$ that extends $\varphi$. As $\varphi(Y)=M\subset Z\subset\widetilde{Z}$, the Zariski closure of $\varphi(Y)$ in $\C^m$ is contained in $\widetilde{Z}$. As $M$ has (real) dimension $d$ and $\widetilde{Z}$ is an irreducible algebraic set of $\C^n$ of (complex) dimension $d$, we deduce that $\widetilde{Z}$ is the Zariski closure of $\varphi(Y)$. As $\widetilde{\varphi}:\C^p\dashrightarrow\C^m$ is continuous for the Zariski topology, $\widetilde{Y}$ is the Zariski closure of $Y$ and $\widetilde{Z}$ is the Zariski closure of $\varphi(Y)=\widetilde{\varphi}(Y)$, we conclude $\widetilde{\varphi}(\widetilde{Y})\subset\widetilde{Z}$ and the Zariski closure of $\widetilde{\varphi}(\widetilde{Y})$ is $\widetilde{Z}$. Thus, $\widetilde{\varphi}|_{\widetilde{Y}}:\widetilde{Y}\dashrightarrow\widetilde{Z}$ is a dominant rational map. Denote with ${\mathcal R}(\widetilde{Y})$ the field of rational functions on $\widetilde{Y}$ and with ${\mathcal R}(\widetilde{Z})$ the field of rational functions on $\widetilde{Z}$. The map $\widetilde{\varphi}^*:{\mathcal R}(\widetilde{Z})\to{\mathcal R}(\widetilde{Y}),\ f\mapsto f\circ\widetilde{\varphi}$ is a homomorphism of fields of the same transcendence degree $d$ over $\C$. Consequently, ${\mathcal R}(\widetilde{Y})$ is an algebraic extension of ${\mathcal R}(\widetilde{Z})$ of finite degree $m$. By \cite[Prop.7.16]{ha} the number of points in a general fiber of $\widetilde{\varphi}$ is equal to $m$. As $M$ has (real) dimension $d$, there exists a point $p\in M$ such that the fiber $\widetilde{\varphi}^{-1}(p)$ has exactly $m$ points. As $Y$ is a (real) algebraic set, $\widetilde{Y}\cap\R^p=Y$. As $\varphi$ is a real rational map and $\varphi|_Y:Y\to M$ is bijective, we conclude that $m$ is odd, because if $z\in\widetilde{Y}\setminus Y$ and $\widetilde{\varphi}(z)=p$, then $\ol{z}\in\widetilde{Y}\setminus Y$ and $\widetilde{\varphi}(\ol{z})=p$. 

Suppose $Z$ has another connected component $M'$ of dimension $d$. Then there exists $q\in M'$ such that $\widetilde{\varphi}^{-1}(q)$ has exactly $m$ points. As $\widetilde{Y}\cap\R^p=Y$ and $\widetilde{\varphi}(Y)=M$, we conclude that $\widetilde{\varphi}^{-1}(q)\subset\widetilde{Y}\setminus Y$. As $q\in Z=\widetilde{Z}\cap\R^n$, we deduce that if $z\in\widetilde{\varphi}^{-1}(q)$, also $\ol{z}\in\widetilde{\varphi}^{-1}(q)$. Thus, $\widetilde{\varphi}^{-1}(q)$ consists of an even number of elements, which is a contradiction because $m$ is odd. Consequently, $M$ is the unique connected component of $Z$ of dimension $d$, as required.
\end{proof}

\begin{example}\label{nint2}
Let $X:=\{\y^2=-(\x-1)(\x-2)(\x+1)\}\subset\R^2$, which is an irreducible non-singular cubic with two connected components of dimension $1$, one is bounded (that we denote with $C_1$) and the other one is unbounded (that we denote with $C_2$). Consider the polynomial $\x$, which satisfies $X\cap\{\x>0\}=C_1$ and $X\cap\{\x<0\}=C_2$. Let $Y:=\{(x,y,z)\in X\times\R:\ xz^2-1=0\}\subset\R^3$, which has exactly two connected components $M_1:=Y\cap\{\z>0\}$ and $M_2:=Y\cap\{\z<0\}$ and both have dimension $1$. A priori we do not know if $Y$ is irreducible, so we cannot apply Lemma \ref{nint} directly to $Y$.
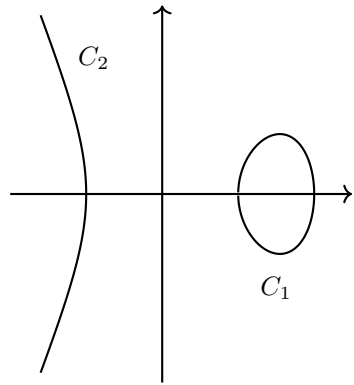
\begin{figure}[!ht]
\centering
\begin{tikzpicture}[scale=1]
\draw[->,thick] (-2, 0) -- (2.5, 0) ;
\draw[->,thick] (0, -2.5) -- (0,2.5) ;
\draw [thick, scale=1,domain=1:1.6, draw=black,samples=1000] plot ({-\x}, {sqrt((\x+1)*(\x+2)*(\x-1))});
\draw [thick, scale=1,domain=1:1.6, draw=black,samples=1000] plot ({-\x}, {-sqrt((\x+1)*(\x+2)*(\x-1))});
\draw (1.5,-1.25) node{\small $C_1$};
\draw [thick, scale=1,domain=-2:-1, draw=black,samples=1000] plot ({-\x}, {sqrt((\x-1)*(\x+2)*(\x+1))});
\draw [thick, scale=1,domain=-2:-1, draw=black,samples=1000] plot ({-\x}, {-sqrt((\x-1)*(\x+2)*(\x+1))});
\draw (-0.9,1.8) node{\small $C_2$};
\end{tikzpicture}
\caption{\small{The cubic curve $\y^2=-(\x-1)(\x-2)(\x+1)$.}}
\end{figure}

Suppose there exists an algebraic set $Z\subset\R^p$ of dimension $1$ and a polynomial map $\varphi:\R^p\to\R^3$ such that $\varphi|_Z:Z\to M_1$ is bijective. Let $\pi:\R^3\to\R^2,\ (x,y,z)\mapsto(x,y)$, which satisfies $\pi|_{M_1}:M_1\to C_1$ is bijective. Thus, the composition $\pi\circ f:\R^p\to\R^2$ is a polynomial map that satisfies $\pi|_{Z}:Z\to C_1$ is bijective, but this contradicts Lemma \ref{nint}. Consequently, there does not exist the couple $(\varphi,Z)$. 
\end{example}

The previous example suggests that in the statement of Corollary \ref{consq2}(i) two connected components Nash diffeomorphic to $\sph^d$ are needed.

\bibliographystyle{amsalpha}

\end{document}